\documentclass[preprint,12pt]{elsarticle}

\usepackage{amssymb}
\usepackage{a4}
\usepackage{graphicx,color,psfrag}
\usepackage{amsmath,amssymb} 
\usepackage[utf8]{inputenc}
\usepackage{multirow}
\usepackage{geometry}
\usepackage{xspace}
\usepackage{amsfonts}
\usepackage{mathrsfs}
\usepackage{array}
\usepackage{tabularx}
\usepackage{graphicx}
\usepackage{multirow}
\usepackage{algorithm}
\usepackage{algpseudocode}
\usepackage{amsbsy}
\usepackage{csquotes}
\usepackage{boldline}
\usepackage{subcaption} 
\usepackage[export]{adjustbox}
\usepackage[table]{xcolor}
\usepackage[english]{babel}
\usepackage[cal=boondox]{mathalfa}
\usepackage{amsthm}
\usepackage{hhline}
\usepackage{thmtools}
\usepackage{booktabs}
\usepackage{siunitx}
\usepackage{titlesec}

\newtheorem{prop}{Proposition}

\newcolumntype{?}{!{\vrule width 2pt}}
\DeclareMathOperator{\rank}{rank}
\DeclareMathOperator{\mydiag}{diag}
\DeclareMathOperator{\SVD}{SVD}
\DeclareMathOperator{\sign}{sign}

\newlength{\Oldarrayrulewidth}
\newcommand{\HHline}[2]{%
  \noalign{\global\setlength{\Oldarrayrulewidth}{\arrayrulewidth}}%
  \noalign{\global\setlength{\arrayrulewidth}{#1}}\hhline{#2}%
  \noalign{\global\setlength{\arrayrulewidth}{\Oldarrayrulewidth}}}

\algrenewcommand\algorithmicrequire{\textbf{Input:}}
\algrenewcommand\algorithmicensure{\textbf{Output:}}
\algnewcommand{\algorithmicand}{\textbf{ and }}
\algnewcommand{\algorithmicor}{\textbf{ or }}
\algnewcommand{\OR}{\algorithmicor}
\algnewcommand{\AND}{\algorithmicand}

\definecolor{matblue}{rgb}{0 0.4470 0.7410}
\definecolor{matred}{rgb}{0.6350 0.0780 0.1840}
\definecolor{matyell}{rgb}{0.9290 0.6940 0.1250}
\definecolor{matpurp}{rgb}{0.4940 0.1840 0.5560}
\definecolor{matgree}{rgb}{0.4660 0.6740 0.1880}
\definecolor{grayxgray}{rgb}{0.75,0.75,0.75}

\journal{Applied Numerical Mathematics}

\begin{document}

\begin{frontmatter}



\title{An Optimal Triangle Projector with Prescribed Area and Orientation,
Application to Position-Based Dynamics}


\author{Carlos Arango Duque\corref{cor1}}
\ead{Carlos.ARANGO_DUQUE@uca.fr}
\author{Adrien Bartoli}

\cortext[cor1]{Corresponding author}
\address{EnCoV, IGT, Institut Pascal, UMR6602 CNRS Universit\'e Clermont Auvergne}

\begin{abstract}
The vast majority of mesh-based modelling applications iteratively transform the mesh vertices under prescribed geometric conditions. This occurs in particular in methods cycling through the constraint set such as Position-Based Dynamics (PBD). A common case is the approximate local area preservation of triangular 2D meshes under external editing constraints. At the constraint level, this yields the nonconvex optimal triangle projection under prescribed area problem, for which there does not currently exist a direct solution method. In current PBD implementations, the area preservation constraint is linearised. The solution comes out through the iterations, without a guarantee of optimality, and the process may fail for degenerate inputs where the vertices are colinear or colocated. We propose a closed-form solution method and its numerically robust algebraic implementation. Our method handles degenerate inputs through a two-case analysis of the problem's generic ambiguities. We show in a series of experiments in area-based 2D mesh editing that using optimal projection in place of area constraint linearisation in PBD speeds up and stabilises convergence.
\end{abstract}



\begin{highlights}
\item A comprehensive analysis of the problems of finding the closest triangle under prescribed area or prescribed area and orientation
\item An algebraic procedure to find the optimal solution or the multiple optimal solutions of these non-convex problems
\item Experiments on an application in 2D triangular mesh editing using Position-Based Dynamics (PBD) and comparison to classical PBD with constraint linearisation
\end{highlights}


\end{frontmatter}


\section{Introduction}
A key mechanism in many mesh-based modelling applications  is to transform the mesh vertices to meet prescribed geometric conditions. For example, triangular mesh smoothing may be achieved by moving the vertices of each triangle by a specifically designed two-step stretching-shrinking transformation~\citep{sun_smoothing_2015} or by iteratively applying a local smoothing transformation~\citep{vartziotis_mesh_2008}. Another example is 3D volumetric model deformation, where realism is improved by preserving the volume of the mesh's tetrahedrons~\citep{irving_volume_2007,abu_rumman_position-based_2015}. Position-Based Dynamics (PBD) is a widely used simulation technique that directly manipulates the vertex positions of object meshes. It can model various object behaviours such as rigid body, soft body and fluids~\citep{tsai_position_2017}. Due to its simplicity, robustness and speed, PBD has become very popular in computer graphics and in the video-game industry. In general terms, PBD updates the vertex positions through simple integration of the external forces. These positions are then directly subjected to a series of constraint equations handled one at a time. If the obtained projection minimises the vertex displacement then it is qualified as optimal. For example, the projection for vertex distance preservation is a simple problem, which was solved optimally~\citep{muller_position_2007,bender_survey_2014}. The constraints simulate a wide range of effects like stretching, bending, collision, area and volume conservation~\citep{bender_survey_2014}. Over the years, improvements have been proposed to the original formulation of PBD. These include new bending constraints from simple geometric principles~\citep{kelager_triangle_2010,wang_angle_2014}, stability improvement by geometric stiffness~\citep{tournier_stable_2015} and faster convergence by constraint reordering~\citep{gu_constraint_2017}. 

Mesh editing uses a priori chosen fixed vertices and moving vertices which should respect constraints. In 2D triangular mesh editing, local area preservation is a widely used constraint. The mesh is deformed until the triangle-wise area variation is minimised. Enforcing this constraint leads to the Optimal Triangle Projection with Prescribed Area problem (OTPPA), which we will formally define shortly. OTPPA is a difficult problem and has not yet been given a closed-form solution in the literature, contrarily to optimal vertex distance preservation. We formally define OTPPA as follows.
We define the vertices $v_a$, $v_b$ and $v_c$ of a triangle in a 2D space as a 6D vector $\mathbf{v}$ with:
\begin{equation}
	\mathbf{v} = [v_a, v_b, v_c]^\top = [x_a, y_a, x_b, y_b, x_c, y_c]^\top
	\in \mathbb{R}^6. 
\end{equation}
We denote the input triangle $\tilde{\mathbf{v}}= [\tilde{x}_a, \tilde{y}_a,\tilde{x}_b, \tilde{y}_b,\tilde{x}_c, \tilde{y}_c]^\top$ and the prescribed area $A_o$. We assume $A_o>0$ out of practical considerations\footnote{$A_o=0$ implies that the resulting vertices are colinear, in which case they are simply given by the orthogonal projection of the input vertices onto a best-fit least-squares line to the input vertices.}.
The general OTPPA problem is stated as:
\begin{equation}
    \min_{\mathbf{v} \in \mathbb{R}^6} \mathscr{C}(\mathbf{v}) \quad \text{s.t.} \quad f(\mathbf{v})=0,
\end{equation}
where $\mathscr{C}(\mathbf{v})$ is the least-squares {\em displacement cost}:
\begin{equation}
	\mathscr{C}(\mathbf{v}) = \|\mathbf{v}-\tilde{\mathbf{v}}\|^2 ,
	\label{eq:cost}
\end{equation}
and $f(\mathbf{v})$ is the nonconvex {\em area preservation constraint}, defined from the triangle area function $A(\mathbf{v})$ as:
\begin{equation}
	f(\mathbf{v}) = A(\mathbf{v}) - A_o.
	\label{eq:constarea}
\end{equation}
Areas are positive quantities. This means that when we calculate the area of a triangle using its vertices, we obtain a positive value regardless of their orientation. In this formulation, the area constraint is the difference of two non-negative values. That is, the area of the triangle is constrained but not its orientation. For instance, $\mathbf{v}$ could be mirrored or two vertices could be swapped and the area constraint would still be satisfied. This could result in undesired triangle inversions in mesh editing. We thus introduce a related problem that additionally constrains the triangle orientation. We first define the triangle area function as $A(\mathbf{v}) = |A^*(\mathbf{v})|$, where $A^*(\mathbf{v})$ is the {\em signed area} given by the shoelace formula:
\begin{equation}
	A^*(\mathbf{v}) = \frac{(x_a-x_c)(y_b-y_a)-(x_a-x_b)(y_c-y_a)}{2}.
	\label{eq:areaeq}
\end{equation}
The signed area is more informative than the area. Specifically, $\sign(A^*(\mathbf{v}))$ gives the triangle orientation. We use this property to define the additional orientation constraint. This leads to the Optimal Triangle Projection with Prescribed Area and Orientation problem (OTPPAO), stated as:
\begin{equation}
    \min_{\mathbf{v} \in \mathbb{R}^6} \mathscr{C}(\mathbf{v}) \quad \text{s.t.} \quad f(\mathbf{v})=0 \quad \text{and} \quad g(\mathbf{v})=0,
    \label{eq:mini2}
\end{equation}
where $g(\mathbf{v})$ is the {\em orientation preservation constraint}:
\begin{equation}
    g(\mathbf{v}) = \sign(A^*(\mathbf{v}))-s,
    \label{eq:orconst}
\end{equation}
and $s\in \{-1,1\}$ specifies the prescribed orientation. The value chosen for $s$ depends on the application. For instance, in PBD, one would choose the orientation of the reference mesh triangle, while another possibility would be to preserve the orientation of the input triangle by setting $s=\sign(A^*(\tilde{ \mathbf{v}}))$.

For both OTPPA and OTPPAO, current PBD implementations linearise the area preservation constraint, resulting in suboptimal projection.
For OTPPAO, they also check and enforce the orientation constraint a posteriori in the inner optimisation loop.
Furthermore, they degenerate for inputs where the vertices are colinear or colocated, whereas a reliable solution should handle any input. We can thus expect an optimal projection to improve PBD convergence compared to linearisation. 
We propose a closed-form method to OTPPAO and OTPPA. Our method for OTPPAO handles degenerate inputs through a two-case analysis, guaranteeing it to find the optimal solution and returning multiple optimal solutions for ambiguous inputs. Our method for OTPPA directly relies on OTPPAO and shares the same features. We use our closed-form method to implement PBD, hence dubbed PBD-opt, for mesh editing and compare its performance with respect to the existing PBD implementation with linearisation, dubbed PBD-lin. To illustrate our proposal, we present a one triangle toy example in Figure~\ref{fig:triapbd} in which we wish to resize the triangle to half its initial area. PBD-lin takes several iterations to reach the prescribed area, whereas PBD-opt achieves it directly. The cost evolution shows that PBD-lin starts with a lower cost, but by the time it complies with the area constraint, it reaches a larger cost than PBD-opt, indicating convergence to a local suboptimal minimum.

\begin{figure}
    \centering
  \begin{subfigure}[b]{0.49\textwidth}
    \includegraphics[width=\textwidth]{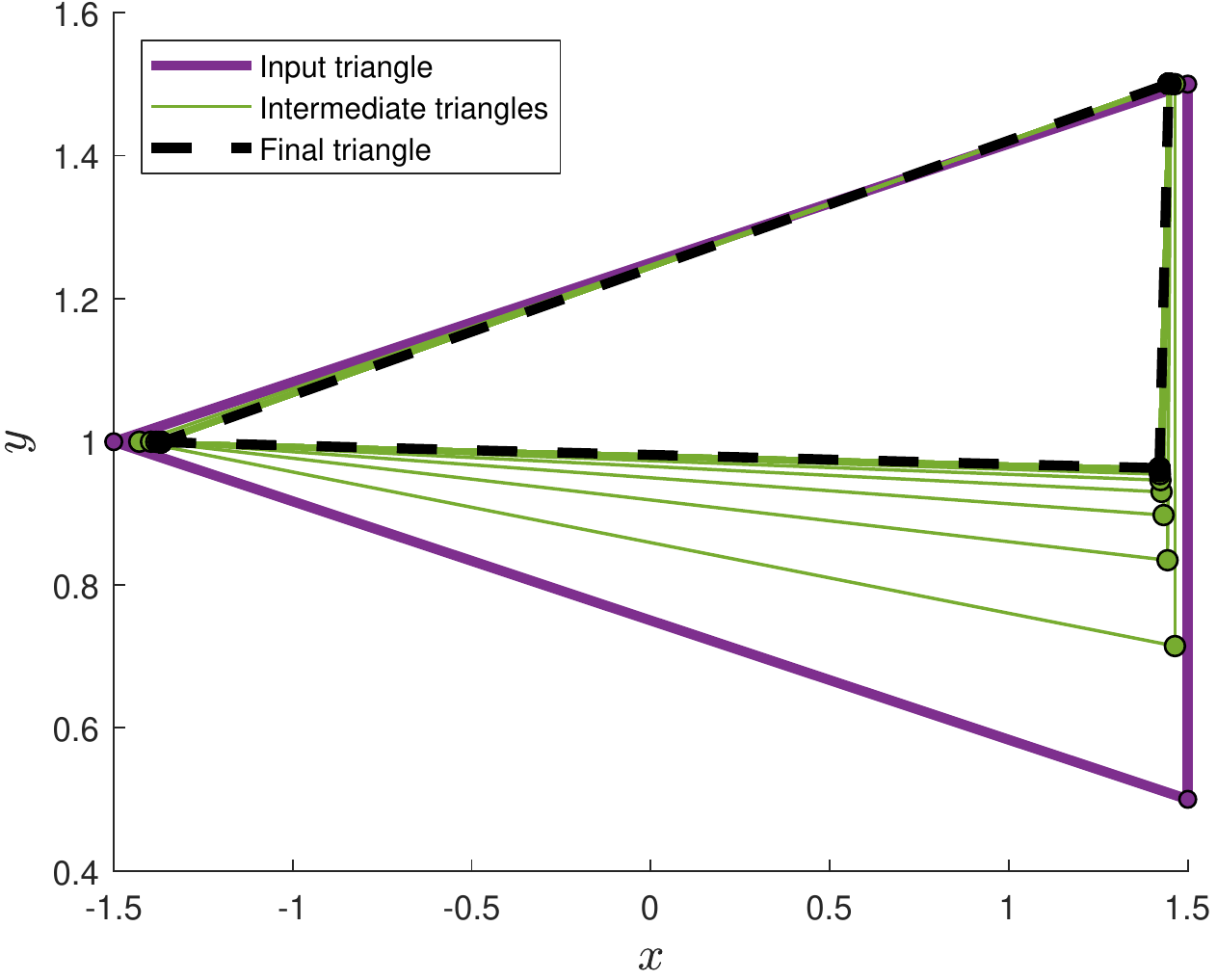}
    \caption{PBD-lin}
  \end{subfigure}
  \begin{subfigure}[b]{0.49\textwidth}
    \includegraphics[width=\textwidth]{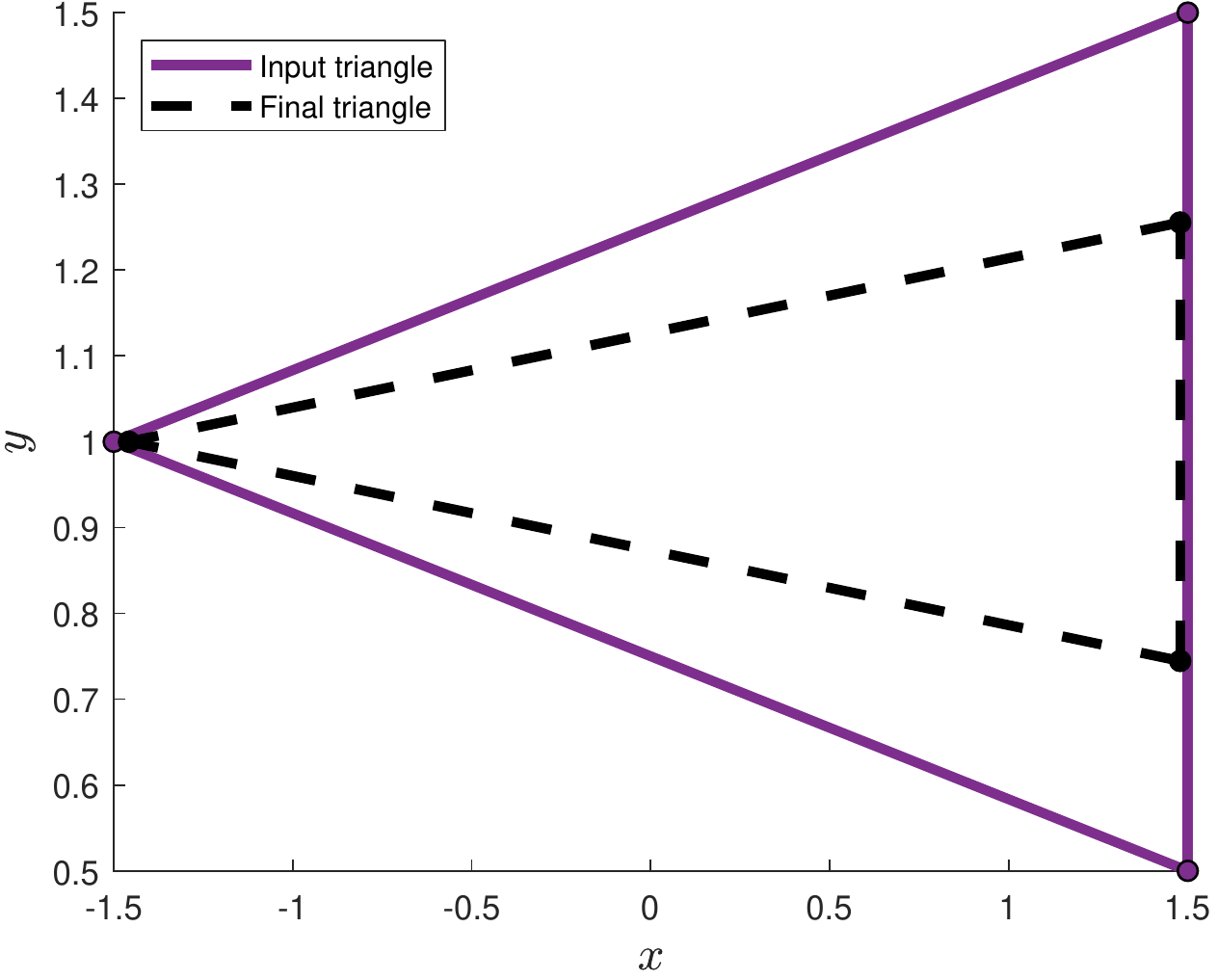}
    \caption{PBD-opt}
  \end{subfigure}\\
  \begin{subfigure}[b]{0.49\textwidth}
    \includegraphics[width=\textwidth]{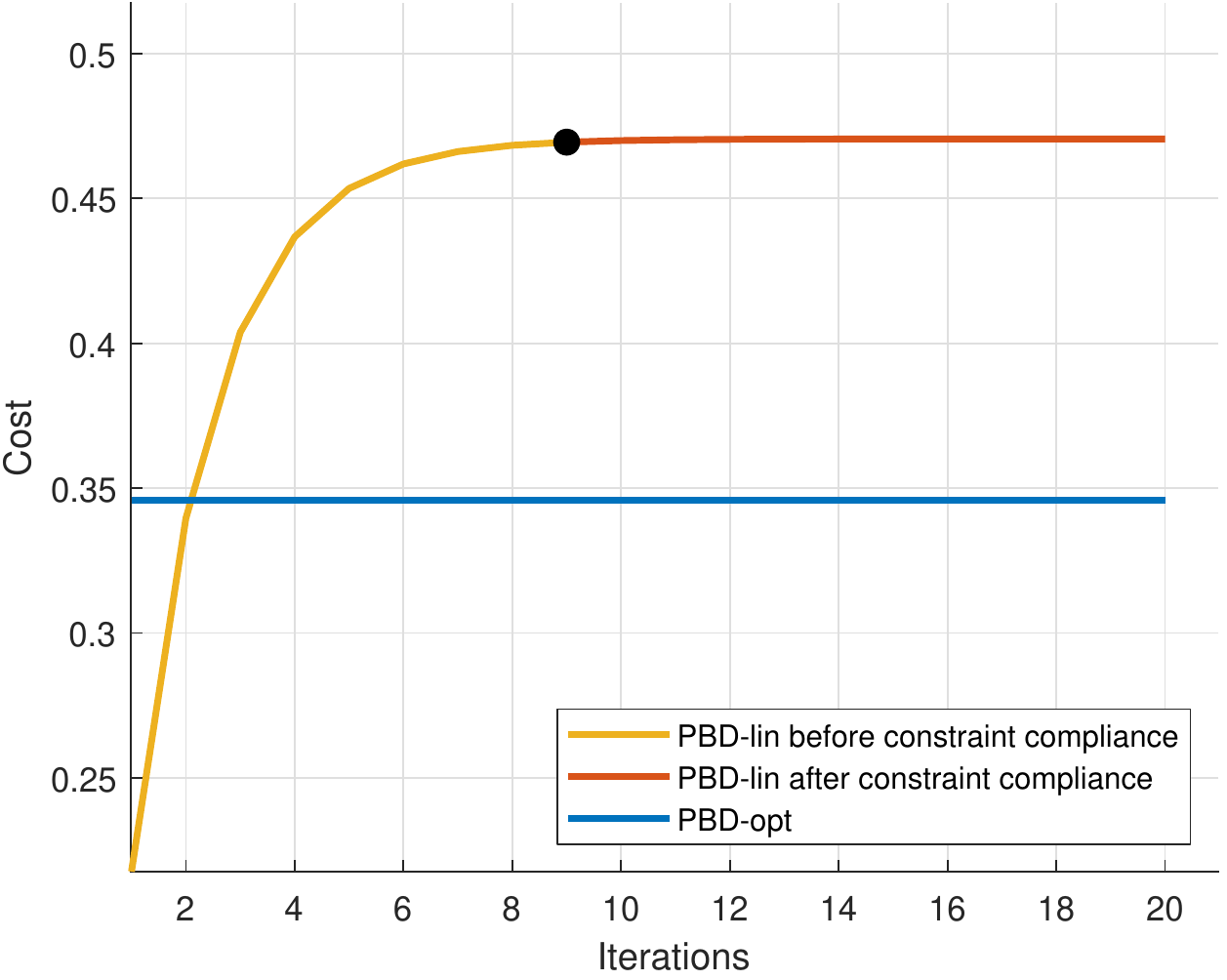}
    \caption{Evolution of cost}
  \end{subfigure}
  \begin{subfigure}[b]{0.49\textwidth}
    \includegraphics[width=\textwidth]{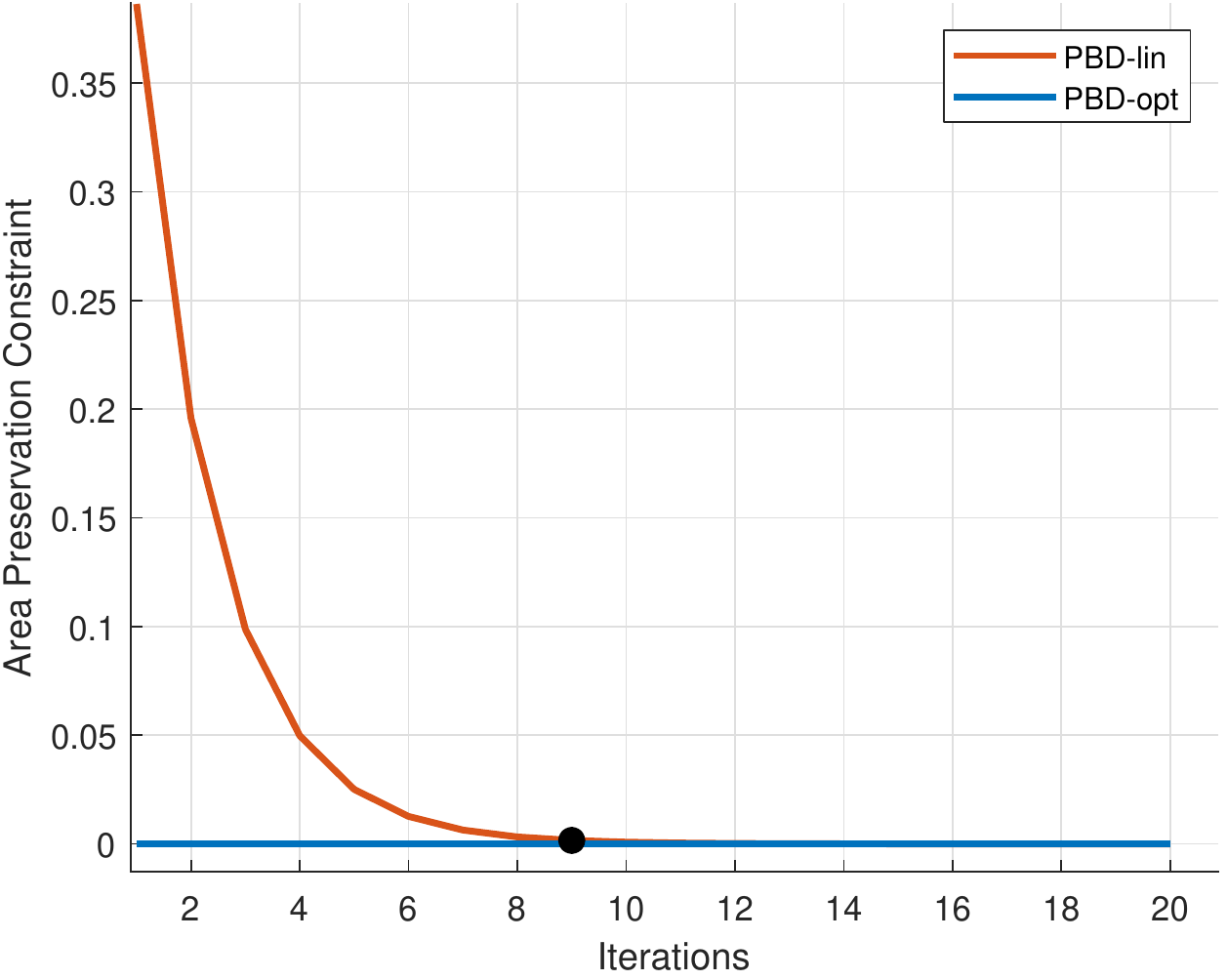}
    \caption{Evolution of area constraint}
  \end{subfigure}\\
  \caption{Method comparison on a one triangle toy example. (a), PBD-lin resizes 
  the initial triangle (purple) into smaller intermediary triangles (green) until it reaches a triangle with the prescribed area (black dashed). (b), PBD-opt directly reaches the prescribed area. (c), Comparison of the evolution of the cost of PBD-lin and the fixed cost of PBD-opt. (d), Comparison of the evolution of the prescribed area constraint of PBD-lin and the fixed area of PBD-opt.
  The cost of PBD-opt (blue) is constant as it gives a direct solution. The cost of PBD-lin is lower during the first iterations (yellow) but the area constraint is not yet fulfilled. By the time it reaches the prescribed area (black dot), the cost of PBD-lin (red) has become larger than the one of PBD-opt. Furthermore, after 20 iterations the area constraint for PBD-lin is \num{8.713e-7} compared to \num{5.507e-15} after one iteration of PBD-opt.}
  \label{fig:triapbd}
\end{figure}

This paper has two parts. In the first part, we derive our closed-form methods. We show that they deal with generic ambiguities. We then implement our methods as numerically robust algebraic procedures. In the second part, we embed our algebraic procedure for OTTPAO in PBD to form an implementation of PBD-opt. We compare its performance in convergence speed and stability with respect to the existing PBD-lin in a series of experiments. We finally give a complementary section where we present our solution to OTPPA and specialise our methods to cases where one or two triangle vertices are fixed, which is typically applicable to triangles of the domain boundary in mesh editing.

\section{Optimal Triangle Projection with Prescribed Area and Orientation}
The derivation of our closed-form method to OTPPAO starts by combining the two constraints into a single one related to both triangle area and orientation. We then construct the Lagrangian which leads to a nonconvex problem, which we handle with two cases. We distinguish and geometrically interpret the two cases based on the input vertices. In the first case, we reformulate the problem as a depressed quartic equation and solve it analytically. In the second case, we reformulate the problem as a series of homogeneous equations and find its null space. Based on these procedures, we develop a numerically robust algebraic implementation and show the results in a series of illustrative examples.

\subsection{Single Constraint Reformulation}
\label{sec:singconst}
We reformulate OTPPAO by merging the area and orientation constraints into a single equivalent constraint:
\begin{equation}
    f^*(\mathbf{v}) = 0 \quad \text{with} \quad f^*(\mathbf{v}) = sA^*(\mathbf{v})-A_o.
    \label{eq:sigconstarea}
\end{equation}
We have $\left(f(\mathbf{v})=0\right) \land  \left( g(\mathbf{v})=0 \right) \Leftrightarrow f^*(\mathbf{v})=0$.
The forward implication is obtained by rewriting $f(\mathbf{v})=0$ as $\sign(A^*(\mathbf{v}))A^*(\mathbf{v})-A_o=0$ and substituting $s=\sign(A^*(\mathbf{v}))$, as obtained from $g(\mathbf{v})=0$, directly giving $f^*(\mathbf{v})=0$.
The reverse implication is obtained by rewriting $f^*(\mathbf{v}) = 0$ as $sA^*(\mathbf{v})=A_o$, whose absolute value gives $f(\mathbf{v})=0$ and whose sign gives $g(\mathbf{v})=0$.
With this new constraint, we reformulate OTPPAO as:
\begin{equation}
    \min_{\mathbf{v} \in \mathbb{R}^6} \mathscr{C}(\mathbf{v}) \quad \text{s.t.} \quad f^*(\mathbf{v})=0.
    \label{eq:mini}
\end{equation}
This reformulation increases compactness but, more importantly, in contrast to the previous area constraint, the new constraint does not involve an absolute value.
More specifically, $f^*(\mathbf{v})$ is a nonconvex but smooth function of $\mathbf{v}$, meaning that a Lagrangian formulation can now be safely constructed.

\subsection{Lagrangian Formulation}
The Lagrangian of the OTPPAO problem~(\ref{eq:mini}) is:
\begin{equation}
	\mathscr{L}(\mathbf{v},\lambda) = \mathscr{C}(\mathbf{v}) + \lambda f^*(\mathbf{v}),
	\label{eq:lagfun}
\end{equation}
where $\lambda$ is the Lagrangian multiplier. Setting the gradient to nought we obtain:
\begin{align}
    &\frac{\partial \mathscr{L}}{\partial \lambda} = f(\mathbf{v}) =
	  sA^*(\mathbf{v}) - A_o = 0 \\
	&\frac{\partial \mathscr{L}}{\partial \mathbf{v}} = 
	\frac{\partial \mathscr{C}(\mathbf{v})}{\partial \mathbf{v}} + 
	\lambda \frac{\partial f(\mathbf{v})}{\partial \mathbf{v}} = 2(\mathbf{v}-\tilde{\mathbf{v}}) + s\lambda \frac{\partial A^*(\mathbf{v})}{\partial \mathbf{v}} = 0. 
	\label{eq:lagmult}
\end{align}
Expanding $\frac{\partial \mathscr{L}}{\partial \mathbf{v}}$, we obtain the following six equations:
\begin{align*} 
&\frac{\partial \mathscr{L}}{\partial x_a} = 2(x_a-\tilde{x}_a) + s\frac{\lambda}{2} (y_b-y_c) = 0	\\
&\frac{\partial \mathscr{L}}{\partial y_a} = 2(y_a-\tilde{y}_a) + s\frac{\lambda}{2} (x_c-x_b) = 0	\\
&\frac{\partial \mathscr{L}}{\partial x_b} = 2(x_b-\tilde{x}_b) + s\frac{\lambda}{2} (y_c-y_a) = 0	\\
&\frac{\partial \mathscr{L}}{\partial y_b} = 2(y_b-\tilde{y}_b) + s\frac{\lambda}{2} (x_a-x_c) = 0	\\
&\frac{\partial \mathscr{L}}{\partial x_c} = 2(x_c-\tilde{x}_c) + s\frac{\lambda}{2} (y_a-y_b) = 0	\\
&\frac{\partial \mathscr{L}}{\partial y_c} = 2(y_c-\tilde{y}_c) + s\frac{\lambda}{2} (x_b-x_a) = 0.
\end{align*}
We rewrite these equations in matrix form as:
\begin{equation}
    X\mathbf{v} = \tilde{\mathbf{v}},
    \label{eq:mateq}
\end{equation} 
where $X\in\mathbb{R}^{6\times6}$ is given by: 
\begin{equation}
    X =
	\begin{bmatrix}
	1 		     & 		      0 & 		     0 & s\lambda/4 & 		   0 & -s\lambda/4 \\
	0		     &   		  1 & -s\lambda/4 &		    0 &  s\lambda/4 &            0 \\
	0 		     & -s\lambda/4 & 		     1 &		    0 & 		   0 &  s\lambda/4 \\
	 s\lambda/4 & 		      0 &            0 &            1 & -s\lambda/4 &            0 \\
	0            &  s\lambda/4 &		     0 & -s\lambda/4 &			   1 & 	     	  0 \\
	-s\lambda/4 &            0 &  s\lambda/4 & 			0 &  		   0 &		      1
	\end{bmatrix}. 
	\label{eq:mat_lam}
\end{equation}

\subsection{Solving with Two Cases}
\label{sec:twocas}
We want to solve for $\mathbf{v}$ from equation~(\ref{eq:mat_lam}). We first check the invertibility of $X$ from its determinant:
\begin{equation}
	\det(X) = \frac{{(3\lambda^2-16)}^2}{256}.
\end{equation}
We thus have:
\begin{equation}
	\det(X)=0 \quad \Leftrightarrow \quad |\lambda| = \lambda_o,
	\label{eq:lambdanull}
\end{equation}
where $\lambda_o$ correspond to the inverted area of a normalized equilateral triangle:
\begin{equation}
    \lambda_o = \frac{4}{\sqrt{3}}.
\end{equation}
We show in the next section that this special case is related to input vertices representing an equilateral triangle or being colocated. We thus solve system~(\ref{eq:mat_lam}) with two cases. In Case I, which is the most general one, we have $|\lambda| \neq \lambda_o$. In Case II, we have $|\lambda| = \lambda_o$.

\subsection{Geometrically Interpreting and Distinguishing the Two Cases}

Cases I and II can be distinguished and geometrically interpreted based on three criteria: 
the linear deficiency of the input vertices $\tilde{\mathbf{v}}$, the orientation change of $\tilde{\mathbf{v}}$ and the scale of the input area $A(\tilde{\mathbf{v}})$ with respect the prescribed area $A_o$. The first criteria is evaluated as the rank of matrix $M \in \mathbb{R}^{3\times3}$ containing $\tilde{\mathbf{v}}$ in homogeneous coordinates as:
\begin{equation}
	M \stackrel{\mathrm{def}}{=} 
	\begin{bmatrix}
	\tilde{x}_a & \tilde{y}_a & 1 \\
	\tilde{x}_b & \tilde{y}_b & 1 \\
	\tilde{x}_c & \tilde{y}_c & 1
	\end{bmatrix}.
	\label{eq:Mmat}
\end{equation}
For most configurations $\rank(M) = 3$, which means that the vertices are not aligned and represent any given triangle whose area $A(\tilde{\mathbf{v}})$ is non-zero. In contrast, $\rank(M) = 2$ means that the three vertices are colinear, in which case $A(\tilde{\mathbf{v}})=0$. Finally, $\rank(M) = 1$ means that the three vertices are colocated and also implies $A(\tilde{\mathbf{v}}) = 0$. 
The second criteria is evaluated by comparing the input vertices orientation $\sign(A^*(\tilde{\mathbf{v}}))$ and the prescribed orientation $s$. When the orientation of the input vertices is preserved then $s\sign(A^*(\tilde{\mathbf{v}}))=1$. On the other hand, when the orientation of the input vertices is inverted then $s\sign(A^*(\tilde{\mathbf{v}}))=-1$. The third criteria is only useful in the special case of an equilateral triangle with preserved orientation. It refers to whether the absolute value of the scaled input area $|zA^*(\tilde{\mathbf{v}})|$ is larger, equal to or smaller than the prescribed area $A_o$ for some $z\in\mathbb{R}>0$.
The interpretation of cases I and II with the above criteria is given in table~\ref{tab:cases} and summarised by the following proposition.

\begin{prop}
\label{prop:gen}
We define a problem setting as the input vertices $\tilde{\mathbf{v}}$, the prescribed area $A_o$ and orientation $s$. Most settings fall in Case I, they are denoted $S_o$. Exceptions handled with Case II are:
\begin{itemize}
    \item $S_1$: $\tilde{\mathbf{v}}$ is a single point
    \item $S_2$: $\tilde{\mathbf{v}}$ is an equilateral triangle and
            $s\sign(A^*(\tilde{\mathbf{v}}))=-1$
    \item $S_3$: $\tilde{\mathbf{v}}$ is an equilateral triangle, $A(\tilde{\mathbf{v}})/4\geq A_o$ and $s\sign(A^*(\tilde{\mathbf{v}}))=1$
\end{itemize}
\label{prop:main}
\end{prop}
\noindent The proof of proposition~\ref{prop:main} is  based on the following five lemmas.
\newtheorem{lem}{Lemma}
\begin{lem}
$S_1 \iff A(\tilde{\mathbf{v}})=0$ and $|\lambda|=\lambda_o$.
\end{lem}
\begin{lem}
$S_2 \iff A(\tilde{\mathbf{v}}) \neq 0$ and $\lambda=-\lambda_o$.
\end{lem}
\begin{lem}
$S_3 \Rightarrow A(\tilde{\mathbf{v}}) \neq 0$ and $\lambda \in \bigg\{ \lambda_o, -\lambda_o+\sqrt{\frac{\lambda_o}{A_o}} , -\lambda_o-\sqrt{\frac{\lambda_o}{A_o}}\bigg\}$.
\end{lem}
\begin{lem}
$S_3 \Leftarrow A(\tilde{\mathbf{v}}) \neq 0$ and $\lambda = \lambda_o$.
\end{lem}
\begin{lem}
Choosing $\lambda=\lambda_o$ leads to the optimal solution for $S_3$.
\end{lem}
\noindent The proofs of these lemmas are given in Appendix~\ref{sec:lempro}. 
\begin{proof}[Proof of proposition 1]
We recall that Case I occurs for $|\lambda| \neq \lambda_o$ and Case II for $|\lambda| = \lambda_o$. Lemmas 1, 2 and 4 show that $S_1$, $S_2$ and $S_3$ are the only possible settings corresponding to $|\lambda| = \lambda_o$, hence possibly to Case II. This proves that Case I is the general case. Lemmas 1 and 2 then trivially prove that $S_1$ and $S_2$ are handled by Case II. Finally, lemmas 3 and 5 prove that $S_3$ is also handled by Case II.
\end{proof}

\subsection{Case I}
\label{sec:case1}
Case I is the most general one. It occurs for $|\lambda| \neq \lambda_o$, equivalent to $\det(X)\neq0$. From proposition~\ref{prop:gen}, we have $\rank(M) \geq 2$, in other words, at least one of the initial vertices $\tilde{\mathbf{v}}$ is different from the other two (except if the input is an equilateral triangle under the conditions of proposition 1). We follow two steps. We first eliminate the vertices from the equations, which leads to a depressed quartic in $\lambda$. We then find the roots of this quartic using Ferrari's method and trivially solve for the vertices from the initial linear system~(\ref{eq:mateq}).

\begingroup
\setlength{\tabcolsep}{3pt} 
\begin{table}
\footnotesize
\centering
	\begin{tabular}{ccc >{\centering\arraybackslash}m{1.5cm} ccccc}
    \toprule
	\multirow{2}{*}{Case} &\multirow{2}{*}{Setting} & \multirow{2}{*}{Input} & \multirow{2}{*}{$\det(X)$} & \multirow{2}{*}{$\rank(M)$} & \multirow{2}{*}{$A(\tilde{\mathbf{v}})$} & \multirow{2}{*}{$\sigma^2(\tilde{\mathbf{v}})$} & Number of & \multirow{2}{*}{$s\sign(A(\tilde{\mathbf{v}}))$} \\
	& & & & & & & solutions \\  
    \midrule
	\multirow{6}{*}{I} & \multirow{6}{*}{$S_o$} &  \parbox[c]{1.4cm}{\includegraphics[width=1.8cm]{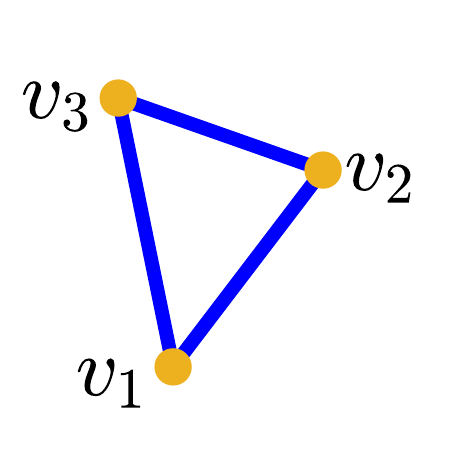}} & \multirow{6}{*}{non zero} & 3 & non zero & \multirow{6}{*}{non zero} & \multirow{6}{*}{$\leq 4$} & \multirow{6}{*}{$\pm1$} \\
    & & \parbox[c]{1.4cm}{\includegraphics[width=1.8cm]{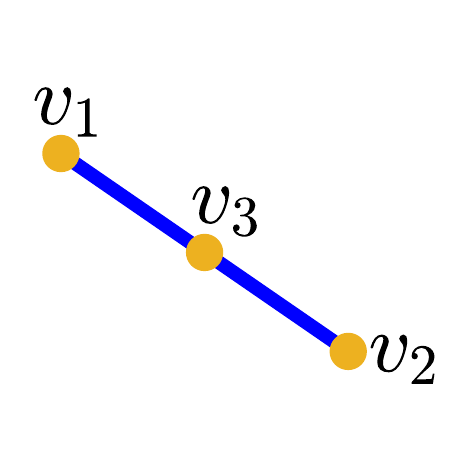}} & & 2 & zero & & & \\	
    \midrule
    II & $S_1$ &\parbox[c]{1.5cm}{\includegraphics[width=2cm, trim={0 1cm 0 0},clip]{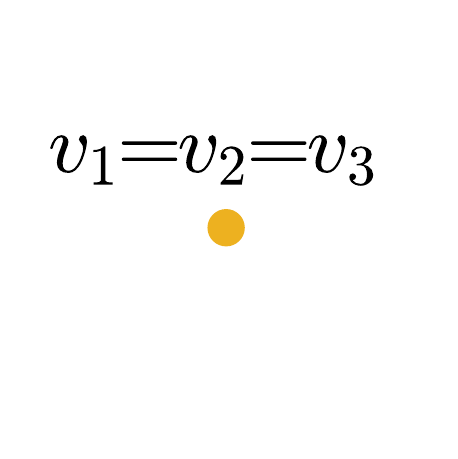}} & zero & 1 & zero & zero & $\infty$ & $\pm1$\\
    \midrule
    \multirow{2}{*}{I} & \multirow{2}{*}{$S_o$} & \multirow{6}{*}{\parbox[c]{1.55cm}{\includegraphics[width=1.8cm]{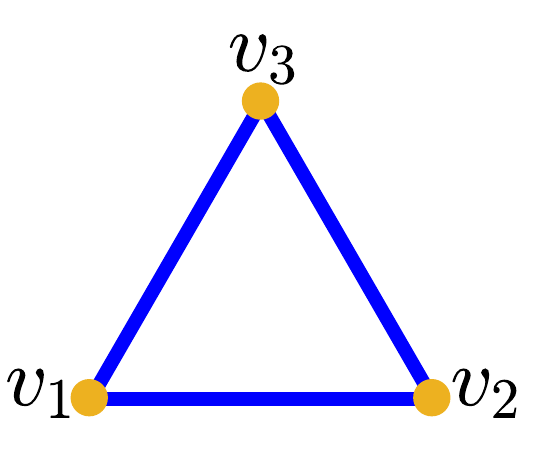}}} & \multirow{2}{*}{non zero} & \multirow{6}{*}{3} & \multirow{2}{*}{$A(\tilde{\mathbf{v}})/4 \leq A_o$} & \multirow{6}{*}{non zero} & \multirow{2}{*}{$\leq 4$} & \multirow{2}{*}{1} \\ \\
    \multirow{2}{*}{II} & \multirow{2}{*}{$S_3$} & & \multirow{2}{*}{zero} & & \multirow{2}{*}{$A(\tilde{\mathbf{v}})/4 \geq A_o$} & & \multirow{2}{*}{$\infty$} & \multirow{2}{*}{1}  \\ \\
    \multirow{2}{*}{II} & \multirow{2}{*}{$S_2$} & & \multirow{2}{*}{zero} & & \multirow{2}{*}{non zero} & & \multirow{2}{*}{$\infty$} & \multirow{2}{*}{-1}  \\ \\
    
    \bottomrule
	\end{tabular}
	\caption{Characteristics of cases I and II and number of solutions.}
	\label{tab:cases}
\end{table}
\endgroup

\subsubsection{Polynomial Reformulation}
Our reformulation proceeds by expressing the vertices in $\tilde{\mathbf{v}}$ as a function of $\lambda$ scaled by the determinant and substituting in the signed area constraint $f^*(\mathbf{v})=0$. We start by multiplying equation~(\ref{eq:mat_lam}) by the adjugate $X^*$ of $X$ and obtain:
\begin{equation}
    \det(X)\mathbf{v} = X^*\tilde{\mathbf{v}},
    \label{eq:mateqadj}
\end{equation}
where the adjugate is:
\begin{equation}
	X^*=\frac{\delta}{256}Y =\frac{3\lambda^2-16}{256}
	\begin{bmatrix}
	\lambda^2-16 &			  0 & \lambda^2    &  4s\lambda  &   \lambda^2  & -4s\lambda \\
			   0 & \lambda^2-16 & -4s\lambda  &   \lambda^2  &  4s\lambda  & \lambda^2   \\
    \lambda^2    & -4s\lambda  & \lambda^2-16 &           0  &   \lambda^2  &  4s\lambda \\
    4s\lambda & \lambda^2    &            0 & \lambda^2-16 & -4s\lambda  & \lambda^2   \\
       \lambda^2 &  4s\lambda  & \lambda^2    & -4s\lambda  & \lambda^2-16 & 0           \\
    -4s\lambda  & \lambda^2    &  4s\lambda  & \lambda^2    &            0 & \lambda^2-16
	\end{bmatrix},
	\label{eq:adjY}
\end{equation}
with $\delta = 3\lambda^2-16$ and $Y\in\mathbb{R}^{6\times6}$. We notice the following:
\begin{equation}
    \det(X) = \frac{\delta^2}{256}.
    \label{eq:deltasq}
\end{equation}
We substitute equations~(\ref{eq:adjY}) and~(\ref{eq:deltasq}) in equation~(\ref{eq:mateqadj}) and obtain:
\begin{equation}
    \delta\mathbf{v} = Y\tilde{\mathbf{v}}.
    \label{eq:mateqadj2}
\end{equation}
We observe that the signed area $A^*(\delta\mathbf{v}) = \delta^2 A^*(\mathbf{v})$. Thus, we calculate the signed area of both sides of equation~(\ref{eq:mateqadj}) and obtain:
\begin{equation}
    \delta^2 A^*(\mathbf{v}) = A^*(Y\tilde{\mathbf{v}}).
    \label{eq:detarea}
\end{equation}
After some minor manipulations, we obtain:
\begin{equation}
    A^*(Y\tilde{\mathbf{v}}) = a_2 \lambda^2 + a_1 \lambda + a_o,
    \label{eq:areayv}
\end{equation}
where:
\begin{align*}
	a_0  &=  128((\tilde{x}_a-\tilde{x}_c)(\tilde{y}_b-\tilde{y}_a)-(\tilde{x}_a-\tilde{x}_b)(\tilde{y}_c-\tilde{y}_a)) \\
	a_1  &= -64s(\tilde{x}_a^2+\tilde{x}_b^2 + \tilde{x}_c^2 +\tilde{y}_a^2+\tilde{y}_b^2 +\tilde{y}_c^2 -\tilde{x}_a\tilde{x}_b-\tilde{x}_a\tilde{x}_c-\tilde{x}_b\tilde{x}_c
	- \tilde{y}_a\tilde{y}_b-\tilde{y}_a\tilde{y}_c-\tilde{y}_b\tilde{y}_c) \\
	a_2  &= 24((\tilde{x}_a-\tilde{x}_c)(\tilde{y}_b-\tilde{y}_a)-(\tilde{x}_a-\tilde{x}_b)(\tilde{y}_c-\tilde{y}_a)).
\end{align*}
We can rewrite these coefficients more compactly. Concretely, $a_0$ and $a_2$ contain the signed area of the input vertices $A^*(\tilde{\mathbf{v}})$. Furthermore, $a_1$ is proportional to the variance of the input vertices as $a_1 = -96s\sigma^2(\tilde{\mathbf{v}})$. We substitute equation~(\ref{eq:detarea}) in the signed area constraint~(\ref{eq:sigconstarea}) multiplied by $\delta^2$ and obtain:
\begin{equation}
    sA^*(Y\tilde{\mathbf{v}}) - \delta^2 A_o = 0.
    \label{eq:subconsiar}
\end{equation}
This way, the signed area only depends on the known initial vertices $\tilde{\mathbf{v}}$ and prescribed sign $s$. Because the signed area is quadratic in the vertices, and the vertices are quadratic rational in $\lambda$, the resulting equation is a quartic in $\lambda$:
\begin{equation}
	9A_o\lambda^4 - 48(2A_o+sA^*(\tilde{\mathbf{v}}))\lambda^2 + 96 \sigma^2(\tilde{\mathbf{v}})\lambda +256(A_o - s A^*(\tilde{\mathbf{v}})) = 0.
	\label{eq:quart}
\end{equation}
This is a depressed quartic because it does not have a cubic term. 
We can thus rewrite it to the standard form by simply dividing by $9A_o$, giving:
\begin{equation}
	\lambda^4 + p\lambda^2 + q\lambda + r = 0,
	\label{eq:depquar}
\end{equation}
with:
\begin{align}
	p &= - \frac{16(2A_o + sA^*(\tilde{\mathbf{v}}))}{3A_o} \\
	q &=  \frac{32\sigma^2(\tilde{\mathbf{v}})}{3A_o} \\
	r &=  \frac{256(A_o- sA^*(\tilde{\mathbf{v}}))}{9A_o}.
\end{align}
An important question is whether we can further simplify the depressed quartic by nullifying one of its coefficients. The possible actions lie in choosing the coordinate frame in which the vertices are expressed using a proper scaled Euclidean transformation. Coefficients $p$ and $r$ are linear combinations of the triangle areas $A^*(\tilde{\mathbf{v}})$ and $A_o$ and also they are normalised by $A_o$. This means that they are scale, rotation and translation invariant, and thus cannot be cancelled. Coefficient $q$ is proportional to the variance $\sigma^2(\tilde{\mathbf{v}})$ which is also rotation and translation invariant, and thus cannot be cancelled, just scaled. Therefore, the depressed quartic cannot be simplified. The next step is to solve the depressed quartic.

\subsubsection{Solution using Ferrari's Method}
We have chosen Ferrari’s method~\citep{tignol_galois_2001,cardano_rules_1545} to solve the quartic equation\footnote{There are five main types of solution methods for a quartic equation. There does not seem to exist a consensus as to which one should be preferred in terms of stability~\citep{herbison-evans_solving_1995,harald_helfgott_modern_2010,shmakov_universal_2011}.}.
The method details and proof may be found in Appendix~\ref{sec:ferrari}.
We here give its main steps for the sake of completeness and for the construction of our numerically robust procedure in section~\ref{sec:numimp}. We first extract the resolvent cubic for equation~(\ref{eq:depquar}). 
We then use Cardano's formula to extract the real root $\alpha_o$ of the resolvent cubic as:
\begin{equation}
	\alpha_o = \sqrt[3]{Q_2+\sqrt{Q_1^3+Q_2^2}} + \sqrt[3]{Q_2-\sqrt{Q_1^3+Q_2^2}} - \frac{p}{3}
\end{equation}
where:
\begin{align}
	Q_1 &= -\frac{p^2+12r}{36} \\
	Q_2 &= \frac{2p^3 - 72rp+ 27q^2}{432}.
\end{align}
The expansion of $Q_1$ and $Q_2$ does not bring simplified expressions. We note that proposition 1 implies $q \neq 0$, thus $\alpha_o \neq 0$.
We finally use $\alpha_o$ to extract the roots of the depressed quartic as:
\begin{equation}
	\lambda = \frac{s_1\sqrt{\alpha_o}+s_2\sqrt{-\left(p+\alpha_o+s_1\frac{q}{\sqrt{2\alpha_o}}\right)}}{\sqrt{2}},
	\label{eq:roots}
\end{equation}
where $s_1,s_2 \in \{-1,1\}$, leaving four possibilities, hence four roots. 
Substituting these roots in equation~(\ref{eq:mateq}), we obtain four sets of vertices, at least one of which representing an optimal solution to OTPPAO. 

\subsection{Case II}
\label{sec:case2}
Case II is a special case. It occurs for $|\lambda| = \lambda_o$, equivalent to $\det(X)=0$. From proposition~\ref{prop:gen}, this means that the initial vertices $\tilde{\mathbf{v}}$ either are colocated as $\tilde{v}_a=\tilde{v}_b=\tilde{v}_c$ or represent equilateral triangles under the conditions of proposition~\ref{prop:gen}. We show that the problem is represented by translated homogeneous and linearly dependent equations. We find their null space and then a subset constrained by the prescribed area.

We translate the coordinate system to bring the input triangle's centroid to the origin as $\tilde{\mathbf{v}}'=\tilde{\mathbf{v}}-\bar{\mathbf{v}}$, which also translates the unknown vertices to $\mathbf{v}'=\mathbf{v}-\bar{\mathbf{v}}$.
Substituting $|\lambda| = \lambda_o$ in matrix $X$, we obtain:
\begin{equation}
    X =
	\begin{bmatrix}
	1 & 0 & 0 & \frac{s\sign(\lambda)}{\sqrt{3}} & 0 & -\frac{s\sign(\lambda)}{\sqrt{3}} \\
	0 & 1 & -\frac{s\sign(\lambda)}{\sqrt{3}} & 0 & \frac{s\sign(\lambda)}{\sqrt{3}} & 0 \\
	0 & -\frac{s\sign(\lambda)}{\sqrt{3}} & 1 & 0 & 0 & \frac{s\sign(\lambda)}{\sqrt{3}} \\
	\frac{s\sign(\lambda)}{\sqrt{3}} & 0 & 0 & 1 & -\frac{s\sign(\lambda)}{\sqrt{3}} & 0 \\
	0 & \frac{s\sign(\lambda)}{\sqrt{3}} & 0 & -\frac{s\sign(\lambda)}{\sqrt{3}} & 1 & 0 \\
	-\frac{s\sign(\lambda)}{\sqrt{3}} & 0 & \frac{s\sign(\lambda)}{\sqrt{3}} & 0 & 0 & 1
	\end{bmatrix}.
	\label{eq:mat_lamcas2}
\end{equation}
We have $\det(X) = 0$, as expected, independently of $\sign(\lambda)$.
In addition, all $5\times5$ minors of $X$ are zero and the leading $4\times4$ minor is non-zero: 
\begin{equation}
    \det\left(
	\begin{bmatrix}
	1 & 0 & 0 & \frac{s\sign(\lambda)}{\sqrt{3}} \\
	0 & 1 & -\frac{s\sign(\lambda)}{\sqrt{3}} & 0 \\
	0 & -\frac{s\sign(\lambda)}{\sqrt{3}} & 1 & 0 \\
	\frac{s\sign(\lambda)}{\sqrt{3}} & 0 & 0 & 1 
	\end{bmatrix}\right) = \frac{4}{9}.
	\label{eq:mat_lamcas2b}
\end{equation}
This means that $\rank(X)=4$. Thus, $X\mathbf{v}' = \tilde{\mathbf{v}}'$ is solvable if and only if $\tilde{\mathbf{v}}'$ lies in the column space $C(X)$. The column space can be calculated by factoring $X$ into its singular value decomposition (SVD) $X = U\Sigma U^\top$ ($X$ is symmetric) and taking the first $\rank(X)$ columns of the unitary matrix $U$.
For each value of $s\sign(\lambda)$, we have column spaces expressed as four-dimensional linear subspaces $
\{ \gamma_1\mathbf{u}_1^- +\gamma_2\mathbf{u}_2^- + \gamma_3\mathbf{u}_3^- +\gamma_4\mathbf{u}_4^- \}$ and $
\{\gamma_1\mathbf{u}_1^+ +\gamma_2\mathbf{u}_2^+ + \gamma_3\mathbf{u}_3^+ +\gamma_4\mathbf{u}_4^+\}$ where $\gamma_1,\gamma_2,\gamma_3,\gamma_4 \in \mathbb{R}$ and with bases $\mathbf{u}_1^-,\mathbf{u}_2^-,\mathbf{u}_3^-,\mathbf{u}_4^-\in \mathbb{R}^6$ and $\mathbf{u}_1^+,\mathbf{u}_2^+,\mathbf{u}_3^+,\mathbf{u}_4^+ \in \mathbb{R}^6$ such that:
\begin{equation}
	\begin{bmatrix}
		\mathbf{u}_1^- & \mathbf{u}_2^- & \mathbf{u}_3^- & \mathbf{u}_4^-
	\end{bmatrix} = 
	\begin{bmatrix}
		           0 &  \lambda_o/4 & \lambda_o/4 &           0 \\
		 \lambda_o/4 &            0 &           0 & \lambda_o/4 \\
                 1/2 & -\lambda_o/8 & \lambda_o/4 &           0 \\
		-\lambda_o/8 &         -1/2 &           0 & \lambda_o/4 \\
				-1/2 & -\lambda_o/8 & \lambda_o/4 &           0 \\
		-\lambda_o/8 &          1/2 &           0 & \lambda_o/4
	\end{bmatrix},
\end{equation}
and:
\begin{equation}
	\begin{bmatrix}
		\mathbf{u}_1^+ & \mathbf{u}_2^+ & \mathbf{u}_3^+ & \mathbf{u}_4^+
	\end{bmatrix} = 
	\begin{bmatrix}
		 \lambda_o/4 &            0 &           0 & \lambda_o/4 \\
		           0 &  \lambda_o/4 & \lambda_o/4 &           0 \\
		-\lambda_o/8 &         -1/2 &          0 &  \lambda_o/4 \\
		         1/2 & -\lambda_o/8 & \lambda_o/4 &           0 \\
		-\lambda_o/8 &          1/2 &           0 & \lambda_o/4 \\
		        -1/2 & -\lambda_o/8 & \lambda_o/4 &           0 
	\end{bmatrix}. 
\end{equation}
We have that $\mathbf{u}_1^-,\mathbf{u}_2^-,\mathbf{u}_1^+$ and $\mathbf{u}_2^+$ represent centred equilateral triangles of the same area of $\frac{1}{\lambda_o}$ with orientation $s\sign(\lambda)$ and we have that $\mathbf{u}_3^-,\mathbf{u}_4^-,\mathbf{u}_3^+$ and $\mathbf{u}_4^+$ represent sets of colocated points. The linear combinations $\gamma_1\mathbf{u}_1^- +\gamma_2\mathbf{u}_2^-$ and $\gamma_1\mathbf{u}_1^+ +\gamma_2\mathbf{u}_2^+$ represent equilateral triangles of any area and opposite orientations (or colocated points if $\gamma_1 = \gamma_2 = 0$), whilst $\gamma_3\mathbf{u}_3^- +\gamma_4\mathbf{u}_4^-$ and $\gamma_3\mathbf{u}_3^+ +\gamma_4\mathbf{u}_4^+$ represent colocated points, hence act as a translation for the vertices of the previous linear combination. 
This shows that the system is solvable if and only if $\tilde{\mathbf{v}}'$ represents an equilateral triangle of orientation $\sign(A^*(\tilde{\mathbf{v}}')) = s\sign(\lambda)$ or colocated vertices.

The system $X\mathbf{v}' = \tilde{\mathbf{v}}'$ is solved by first finding the solutions of the homogeneous system $X\mathbf{v}_h = 0$ and translating them by a particular solution $\mathbf{v}_p$, obtaining $\mathbf{v}'=\mathbf{v}_h + \mathbf{v}_p$.
The homogeneous system has an infinite number of solutions which come from the null space of $X$. 
This can be represented as a two-dimensional linear subspace $\mathbf{v}_h=\beta_1\mathbf{v}_1 +\beta_2\mathbf{v}_2$ where the coefficients $\beta_1,\beta_2 \in \mathbb{R}$ are not both zero and with bases $\mathbf{v}_1,\mathbf{v}_2\in \mathbb{R}^6$ such that:
\begin{equation}
	\begin{bmatrix}
		\mathbf{v}_1 & \mathbf{v}_2
	\end{bmatrix} = 
	\begin{bmatrix}
		-\frac{1}{2} & \frac{2s\sign(\lambda)}{\lambda_o} \\ -\frac{2s\sign(\lambda)}{\lambda_o} & -\frac{1}{2} \\ -\frac{1}{2} & -\frac{2s\sign(\lambda)}{\lambda_o} \\ 
		\frac{2s\sign(\lambda)}{\lambda_o} & -\frac{1}{2} \\ 1 & 0 \\ 0 & 1 
	\end{bmatrix}.
	\label{eq:family}
\end{equation}
We have that $\mathbf{v}_1,\mathbf{v}_2$ represent centred equilateral triangles of the same area of $\frac{3}{\lambda_o}$. The linear combination $\mathbf{v}_h=\beta_1\mathbf{v}_1 +\beta_2\mathbf{v}_2$ generates centred equilateral triangles of any area of orientation $-s\sign(\lambda)$.
We then calculate the particular solution $\mathbf{v}_p$ using the pseudo-inverse as:
\begin{equation}
    \mathbf{v}_p = X^\dagger\tilde{\mathbf{v}}' = 
    \begin{bmatrix}
		\frac{\tilde{x}_a'}{2}+\frac{\tilde{x}_b'}{4}+\frac{\tilde{x}_c'}{4} +\sign(A^*(\tilde{\mathbf{v}}'))\frac{\tilde{y}_b'-\tilde{y}_c'}{3\lambda_o}\\ 
		\frac{\tilde{y}_a'}{2}+\frac{\tilde{y}_b'}{4}+\frac{\tilde{y}_c'}{4} -\sign(A^*(\tilde{\mathbf{v}}'))\frac{\tilde{x}_b'-\tilde{x}_c'}{3\lambda_o}\\ 
		\frac{\tilde{x}_a'}{4}+\frac{\tilde{x}_b'}{2}+\frac{\tilde{x}_c'}{4} -\sign(A^*(\tilde{\mathbf{v}}'))\frac{\tilde{y}_a'-\tilde{y}_c'}{3\lambda_o}\\ 
		\frac{\tilde{y}_a'}{4}+\frac{\tilde{y}_b'}{2}+\frac{\tilde{y}_c'}{4} +\sign(A^*(\tilde{\mathbf{v}}'))\frac{\tilde{x}_a'-\tilde{x}_c'}{3\lambda_o}\\ 
		\frac{\tilde{x}_a'}{4}+\frac{\tilde{x}_b'}{4}+\frac{\tilde{x}_c'}{2} +\sign(A^*(\tilde{\mathbf{v}}'))\frac{\tilde{y}_a'-\tilde{y}_b'}{3\lambda_o}\\ 
		\frac{\tilde{y}_a'}{4}+\frac{\tilde{y}_b'}{4}+\frac{\tilde{y}_c'}{2}  -\sign(A^*(\tilde{\mathbf{v}}'))\frac{\tilde{x}_a'-\tilde{x}_b'}{3\lambda_o}
	\end{bmatrix}.
\end{equation}
We then translate the null space with the particular solution and obtain $\mathbf{v}' = \beta_1\mathbf{v}_1 +\beta_2\mathbf{v}_2 + \mathbf{v}_p$. This linear combination generates centred triangles of any area. These generated triangles are equilateral only if $\mathbf{v}_p = 0$.

The next step is to constrain these triangles to the prescribed area and orientation. After some minor algebraic manipulations, we obtain the signed area of the subspace as:
\begin{equation}
    A^*(\mathbf{v}') = \left(1 + s\sign(A^*(\tilde{\mathbf{v}}'))\sign(\lambda)\right)\frac{A^*(\tilde{\mathbf{v}}')}{8} - s\sign(\lambda)\frac{3(\beta_1^2 + \beta_2^2)}{\lambda_o}.
\end{equation}
When $A^*(\tilde{\mathbf{v}}')\neq0$ we have $s\sign(\lambda) = \sign(A^*(\tilde{\mathbf{v}}'))$ thus $\sign(\lambda) = s\sign(A^*(\tilde{\mathbf{v}}'))$. 
However, when $A^*(\tilde{\mathbf{v}}')=0$ we have $\sign(A^*(\mathbf{v}'))=-s\sign(\lambda)$, which implies that $\sign(\lambda)=-s$. Using the orientation constraint~(\ref{eq:orconst}), we can express $\mathbf{v}'$ as:
\begin{equation}
\begin{aligned}
    \qquad & \mathbf{v}' = \beta_1\mathbf{v}_1 + \beta_2\mathbf{v}_2 + \mathbf{v}_p \\
    \text{s.t.} \qquad&  \left(s + \sign(\lambda)\sign(A^*(\tilde{\mathbf{v}}'))\right)\frac{A^*(\tilde{\mathbf{v}}')}{8} - \sign(\lambda)\frac{3(\beta_1^2 + \beta_2^2)}{\lambda_o} - A_o = 0
    \label{eq:betconst}
\end{aligned}
\end{equation}
Because of the area and orientation constraints, and because $\mathbf{v}_1$ and $\mathbf{v}_2$ are rotated copies of each other, the family defined by equation~(\ref{eq:family}) can be generated by scaling $\mathbf{v}_1$ by: 
\begin{equation}
    \phi = \sqrt{\beta_1^2+\beta_2^2} = \sqrt{\frac{\lambda_o(\sign(A^*(\tilde{\mathbf{v}}'))A^*(\tilde{\mathbf{v}}')-4kA_o)}{12}},
\end{equation}
where $k$ depends on the type of input:
\begin{equation}
    k = \begin{cases}
s\sign(A^*(\tilde{\mathbf{v}}')) & \text{if $A^*(\tilde{\mathbf{v}})\neq0$}\\
-1 &\text{if $A^*(\tilde{\mathbf{v}})=0$},
\end{cases}
\end{equation}
so that the area constraint is met.
We can then rotate $\mathbf{v}_1$ by some arbitrary angle $\theta$. 
We note that when $k=s\sign(A^*(\tilde{\mathbf{v}}'))=1$, then $\phi \in \mathbb{R}$ as long as $A(\tilde{\mathbf{v}})/4\geq A_o$ (which corresponds to setting $S_3$). 
We define a new basis vector $\mathbf{v}_c$ as:
\begin{equation}
    \mathbf{v}_c = \phi
    \begin{bmatrix}
    -\frac{1}{2} & -ks\frac{2}{\lambda_o} & -\frac{1}{2} & ks\frac{2}{\lambda_o} & 1 & 0
    \end{bmatrix}^\top.
\end{equation}
We translate it to the original coordinates by $\mathbf{v}_t$, which is the addition of the particular solution $\mathbf{v}_p$ and the input's centroid $\bar{\mathbf{v}}$, and obtain:
\begin{equation}
    \mathbf{v} =  \mathcal{R}(\theta)\mathbf{v}_c + \mathbf{v}_t,
    \label{eq:case2}
\end{equation}
where $\mathcal{R}(\theta)$ is a block diagonal matrix replicating the 2D rotation matrix $R(\theta)$ three times as $\mathcal{R}(\theta)=\mydiag(R(\theta),R(\theta),R(\theta))$.
All the possible solutions have the same cost.

\subsection{Properties of the Solutions}
An important property of the solutions to OTPPAO is that they preserve the centroid of the input triangle $\tilde{\mathbf{v}}$. For case I, this is shown by substituting the vertices $\mathbf{v}$ from equation~(\ref{eq:mateqadj2}) in the centroid formula as:
\begin{align}
\begin{split}
    \delta\bar{\mathbf{v}} =
    \frac{\delta}{3}   
    \begin{bmatrix}
        x_a+x_b+x_c \\
        y_a+y_b+y_c
    \end{bmatrix}
     & = 
     \frac{1}{3}   
     \biggl[\begin{matrix}
        (\lambda^2-16)(\tilde{x}_a+\tilde{x}_b+\tilde{x}_c)+2\lambda^2(\tilde{x}_a+\tilde{x}_b+\tilde{x}_c)  \\ 
       (\lambda^2-16)(\tilde{y}_a+\tilde{y}_b+\tilde{y}_c)+2\lambda^2(\tilde{y}_a+\tilde{y}_b+\tilde{y}_c)
     \end{matrix}
     \\
     & \qquad 
    \begin{matrix}
        +4s\lambda(\tilde{y}_b - \tilde{y}_c  + \tilde{y}_c - \tilde{y}_a + \tilde{y}_a - \tilde{y}_b) \\
        +4s\lambda(\tilde{x}_b - \tilde{x}_c  + \tilde{x}_c - \tilde{x}_a + \tilde{x}_a - \tilde{x}_b) 
     \end{matrix} \biggr]
     \\
     & =
     \frac{1}{3}   
     \begin{bmatrix}
        (3\lambda^2-16)(\tilde{x}_a+\tilde{x}_b+\tilde{x}_c) \\ 
        (3\lambda^2-16)(\tilde{y}_a+\tilde{y}_b+\tilde{y}_c)
        \end{bmatrix} = 
        \frac{\delta}{3}
        \begin{bmatrix}
        \tilde{x}_a+\tilde{x}_b+\tilde{x}_c \\
        \tilde{y}_a+\tilde{y}_b+\tilde{y}_c
    \end{bmatrix}.
\end{split}
\end{align}
For case II, we similarly substitute the vertices $\mathbf{v}$ from equation~(\ref{eq:betconst}) in the centroid formula and obtain:
\begin{align}
\begin{split}
    \bar{\mathbf{v}} =
    \frac{1}{3}
    \begin{bmatrix}
        x_a+x_b+x_c \\
        y_a+y_b+y_c
    \end{bmatrix}
     & =
     \begin{bmatrix}
        (0)\beta_1 + (0)\beta_2s\sign(\lambda)\frac{2}{\lambda_o} \\ 
        (0)\beta_1s\sign(\lambda)\frac{2}{\lambda_o} + (0)\beta_2
     \end{bmatrix} \\
     &
     +
     \frac{1}{3}
    \begin{bmatrix}
        \tilde{x}_a+\tilde{x}_b+\tilde{x}_c + \sign(\lambda)\frac{(\tilde{y}_b-\tilde{y}_c-\tilde{y}_a+\tilde{y}_c+\tilde{y}_a-\tilde{y}_b)}{\lambda_o}\\
        \tilde{y}_a+\tilde{y}_b+\tilde{y}_c-\sign(\lambda)\frac{(\tilde{x}_b-\tilde{x}_c-\tilde{x}_a+\tilde{x}_c+\tilde{x}_a-\tilde{x}_b)}{\lambda_o}
    \end{bmatrix}
     \\
     & = 
    \frac{1}{3}
    \begin{bmatrix}
        \tilde{x}_a+\tilde{x}_b+\tilde{x}_c \\
        \tilde{y}_a+\tilde{y}_b+\tilde{y}_c
    \end{bmatrix}.
\end{split}
\end{align}

\subsection{Numerical Implementation}
\label{sec:numimp}
We use the theory developed in the previous sections to construct a numerically robust procedure, given in Algorithm 1, to solve OTPPAO. In theory, the first step would be to calculate $\rank(M)$, $A^*(\tilde{\mathbf{v}})$, $s\sign(A^*(\tilde{\mathbf{v}}))$ and the distance between vertices (to check if $\tilde{\mathbf{v}}$ is an equilateral triangle) and use them to branch $\tilde{\mathbf{v}}$ on Case I or Case II. However, the round-off error makes these tests unreliable. In order to deliver a numerically robust solution, both cases must be attempted, and the optimal solution chosen a posteriori. Algorithm 1 uses the input vertices $\tilde{\mathbf{v}}$, prescribed area $A_o$ and orientation $s$ as inputs. It also uses an area error tolerance $E$ to handle round-off in the area constraint~(\ref{eq:sigconstarea}). Algorithm 1 starts by generating the solutions from Case I, then Case II, and chooses the optimal one. 
For Case I, we obtain a list $\mathcal{v}_1$ of at most 4 solutions.
For Case II, we obtain a single best solution $\mathbf{v}_2$, the optimally rotated one, and the basis and offset to generate all solutions following equation~(\ref{eq:case2}).
The overall optimal solution $\mathbf{v}_o$ is chosen amongst $\mathcal{v}_1$ and $\mathbf{v}_2$.
The algorithm returns the optimal solution, along with all the solutions from Case I and Case II. 
This allows the user to deal with possible ambiguities and make the final choice depending on application specific priors and constraints.

\begin{algorithm}
    \caption{Optimal Triangle Projection with a Prescribed Area and Orientation}
    \label{alg:solsel}
    \begin{algorithmic}[1]
    \Require $\tilde{\mathbf{v}}$ - input vertices, $A_o$ - prescribed area, $s$ - prescribed orientation, $E$ - area error tolerance
    \Ensure $\mathbf{v}_o$ - optimal triangle, $\mathcal{v}_1$ - Case I triangle set, $\mathbf{v}_2$ - Case II optimal triangle, $\mathbf{v}_c, \mathbf{v}_t$ - Case II basis and translation 
    \Function{OTTPAO}{$\tilde{\mathbf{v}},A_o,s, E = 10^{-3}$}
    \State $\mathcal{v}_1 \gets$ \Call{SolveCase1}{$\tilde{\mathbf{v}},A_o,s,E$}  \Comment{Compute Case I solutions}
    \State $(\mathbf{v}_2, \mathbf{v}_c, \mathbf{v}_t) \gets$ \Call{SolveCase2}{$\tilde{\mathbf{v}},A_o,s,E$}  \Comment{Compute Case II solutions}
    \State $\mathbf{v}_o\gets$ \Call{FindTriangleOfMinimalCost}{$\mathcal{v}_1 \cup \{\mathbf{v}_2\}$}
    \Comment{Select the optimal solution}
    \State \textbf{return} $\mathbf{v}_o, \mathcal{v}_1, \mathbf{v}_2, \mathbf{v}_c, \mathbf{v}_t$
    \EndFunction
    \end{algorithmic}
\end{algorithm}

Algorithm~\ref{alg:case1} computes the possible solutions for Case I.
It first computes the coefficients $p,q,r$ of the depressed quartic equation (lines 2, 3 and 4). Then, it uses Ferrari's method, given by Algorithm~\ref{alg:ferrari}, to find possible values of the Lagrange multiplier in vector $\boldsymbol{\lambda}$.  Some values in $\boldsymbol{\lambda}$ may be complex because they do not represent a solution or because of round-off error. We thus extract the real part of $\boldsymbol{\lambda}$ (line~\ref{lst:lineverre}).
In theory, the next step would be to verify that $\lambda \neq \lambda_o$ or $\delta \neq 0$, because this would create a rank-deficiency and division by zero. 
However, this cannot be directly tested because of round-off error. 
This is better handled by taking the pseudo-inverse $\delta^{\dagger} = (3\lambda^2-16)^{\dagger}$ (line \ref{lst:linevert}), recalling that $0^\dagger=0$.
We can then simply check that the triangle complies with the area and orientation constraints (line \ref{lst:vertverif}). 

\begin{algorithm}
    \caption{Closed-form Analytic Solution to Case I of OTPPAO}
    \label{alg:case1}
    \begin{algorithmic}[1]
        \Require $\tilde{\mathbf{v}}$ - input vertices, $A_o$ - prescribed area, $s$ - prescribed orientation, $E$ - area error tolerance
        \Ensure $\mathcal{v}_1$ - solution list
		\Function{SolveCase1}{$\tilde{\mathbf{v}},A_o,s,E$}
		\State $p \gets - \frac{16(2A_o + sA^*(\tilde{\mathbf{v}}))}{3A_o}$ \Comment{Compute the coefficients of the depressed quartic} 
		\State $q \gets  \frac{32\sigma^2(\tilde{\mathbf{v}})}{3A_o}$
		\State $r \gets  \frac{256(A_o- sA^*(\tilde{\mathbf{v}}))}{9A_o}$ 
		\State $\boldsymbol{\lambda} \gets $ \Call{FerrariSolution}{$p,q,r$} \Comment{Solve for the four possible Lagrange multipliers}
		\State $\mathcal{v}_1 \gets \emptyset$ \Comment{Create an empty set of solutions}
    	\For{$t\gets 1,\dots,4$}
    	\Comment{Generate and select the triangles}
    	    \State $\lambda_o \gets \operatorname{Re}(\boldsymbol{\lambda}(t))$ \label{lst:lineverre} \Comment{Keep the real part} 
    	    \State $\delta \gets 3\lambda^2-16$ \Comment{Compute $\delta$} 
    	    \State $\mathbf{v} \gets \delta^\dagger
			    \begin{bmatrix}
	                (\lambda_o^2-16)\tilde{x}_a+\lambda_o^2(\tilde{x}_b+\tilde{x}_c) +4s\lambda_o(\tilde{y}_b - \tilde{y}_c) \\
	                (\lambda_o^2-16)\tilde{y}_a+\lambda_o^2(\tilde{y}_b+\tilde{y}_c) +4s\lambda_o(\tilde{x}_c - \tilde{x}_b) \\
	                (\lambda_o^2-16)\tilde{x}_b+\lambda_o^2(\tilde{x}_a+\tilde{x}_c) +4s\lambda_o(\tilde{y}_c - \tilde{y}_a) \\
	                (\lambda_o^2-16)\tilde{y}_b+\lambda_o^2(\tilde{y}_a+\tilde{y}_c) +4s\lambda_o(\tilde{x}_a - \tilde{x}_c) \\
	                (\lambda_o^2-16)\tilde{x}_c+\lambda_o^2(\tilde{x}_a+\tilde{x}_b) +4s\lambda_o(\tilde{y}_a - \tilde{y}_b) \\
	                (\lambda_o^2-16)\tilde{y}_c+\lambda_o^2(\tilde{y}_a+\tilde{y}_b) +4s\lambda_o(\tilde{x}_b - \tilde{x}_a) 
			    \end{bmatrix}$ 
			    \Comment{\parbox[t]{.15\linewidth}{Compute the vertices}} \label{lst:linevert}
            \If{$|s A^*(\mathbf{\mathbf{v}})-A_o|\leq E$} 
            \Comment{Check the area constraint} \label{lst:vertverif}
            \State $\mathcal{v}_1 \gets \mathcal{v}_1 \cup \{ \mathbf{v} \}$ \Comment{Add the vertices to the solution set}
    	    \EndIf
    	\EndFor 
    	\State \textbf{return} $\mathcal{v}_1$
		\EndFunction
	\end{algorithmic}
\end{algorithm}

\begin{algorithm}
    \caption{Ferrari's Solution to the Depressed Quartic}
    \label{alg:ferrari}
    \begin{algorithmic}[1]
    \Require $p,q,r$ - coefficients of depressed quartic
    \Ensure $\boldsymbol{\lambda}$ - set of four roots
		\Function{FerrariSolution}{$p,q,r$}
			\State $Q_1 \gets -\frac{p^2+12r}{36}$ \Comment{Compute Cardano's formula coefficients}
			\State $Q_2 \gets \frac{2p^3-72rp+27q^2}{432}$
			\State $\alpha_o \gets \sqrt[3]{Q_2+\sqrt{Q_1^3+Q_2^2}} + \sqrt[3]{Q_2-\sqrt{Q_1^3+Q_2^2}} - \frac{p}{3}$ \Comment{\parbox[t]{.3\linewidth}{Compute the real root of the resolvent cubic}}
			\State $\boldsymbol{\lambda} \gets \emptyset$ \Comment{Create an empty solution set}
			\For{$k_1\gets \{1,2\}$}
				\For{$k_2\gets \{1,2\}$}
					\State $\lambda_o \gets  \frac{(-1)^{k_1}\sqrt{2\alpha_o}+(-1)^{k_2}\sqrt{-\left(2p+2\alpha_o+(-1)^{k_1}\frac{2q}{\sqrt{2\alpha_o}}\right)}}{2}$ \Comment{Compute the root}
					\State $\boldsymbol{\lambda} \gets \boldsymbol{\lambda} \cup \{ \lambda_o\}$ \Comment{Add it to the solution set} 
				\EndFor
			\EndFor
		    \State \textbf{return} $\boldsymbol{\lambda}$
		\EndFunction
	\end{algorithmic}
\end{algorithm}

\begin{algorithm}
    \caption{Closed-form Analytic Solution to Case II of OTPPAO}
    \label{alg:case2}
    \begin{algorithmic}[1]
    \Require $\tilde{\mathbf{v}}$ - input vertices, $A_o$ - prescribed area, $s$ - prescribed orientation, $E$ - area error tolerance
    \Ensure $\mathbf{v}_2$ - optimal triangle, $\mathbf{v}_c, \mathbf{v}_t $ - triangle basis and translation
    \Function{SolveCase2}{$\tilde{\mathbf{v}},A_o,s,E$}
        \If{$|A^*(\mathbf{\mathbf{v}})|\leq E$} 
            \Comment{Check the Input's area} 
            \State $k \gets s\sign(A^*(\mathbf{\mathbf{v}}))$ \Comment{Compute $k$ for an equilateral triangle}
        \Else
            \State $k \gets -1$ \Comment{Compute $k$ for a single point}
    	\EndIf
    	\State $\bar{\mathbf{v}} \gets \frac{1}{3}
            \begin{bmatrix}
                \tilde{x}_a+\tilde{x}_b+\tilde{x}_c \\
                \tilde{y}_a+\tilde{y}_b+\tilde{y}_c
            \end{bmatrix}$ \Comment{Compute the centroid of the input vertices} \label{lst:linecentr}
        \State $\tilde{\mathbf{v}}' \gets \tilde{\mathbf{v}} - \bar{\mathbf{v}}$ \Comment{Translate the input vertices}
    	\State $\phi \gets  \sqrt{\frac{\lambda_o(\sign(A^*(\tilde{\mathbf{v}}))A^*(\tilde{\mathbf{v}})-4kA_o)}{12}}$ \Comment{Computes the area constraint parameter}
        \State $\mathbf{v}_c \gets \operatorname{Re}(\phi)
                \begin{bmatrix}
                    -\frac{1}{2} & -ks\frac{2}{\lambda_o} & -\frac{1}{2} & ks\frac{2}{\lambda_o} & 1 & 0
                \end{bmatrix}^\top $ \Comment{Compute the solution basis} \label{lst:linenull}
        \State $\mathbf{v}_p \gets 
			    \begin{bmatrix}
		            \frac{\tilde{x}_a'}{2}+\frac{\tilde{x}_b'}{4}+\frac{\tilde{x}_c'}{4} +\sign(A^*(\tilde{\mathbf{v}}))\frac{\tilde{y}_b'-\tilde{y}_c'}{3\lambda_o}\\ 
		            \frac{\tilde{y}_a'}{2}+\frac{\tilde{y}_b'}{4}+\frac{\tilde{y}_c'}{4} -\sign(A^*(\tilde{\mathbf{v}}))\frac{\tilde{x}_b'-\tilde{x}_c'}{3\lambda_o}\\ 
		            \frac{\tilde{x}_a'}{4}+\frac{\tilde{x}_b'}{2}+\frac{\tilde{x}_c'}{4} -\sign(A^*(\tilde{\mathbf{v}}))\frac{\tilde{y}_a'-\tilde{y}_c'}{3\lambda_o}\\ 
		            \frac{\tilde{y}_a'}{4}+\frac{\tilde{y}_b'}{2}+\frac{\tilde{y}_c'}{4} +\sign(A^*(\tilde{\mathbf{v}}))\frac{\tilde{x}_a'-\tilde{x}_c'}{3\lambda_o}\\ 
		            \frac{\tilde{x}_a'}{4}+\frac{\tilde{x}_b'}{4}+\frac{\tilde{x}_c'}{2} +\sign(A^*(\tilde{\mathbf{v}}))\frac{\tilde{y}_a'-\tilde{y}_b'}{3\lambda_o}\\ 
		            \frac{\tilde{y}_a'}{4}+\frac{\tilde{y}_b'}{4}+\frac{\tilde{y}_c'}{2}  -\sign(A^*(\tilde{\mathbf{v}}))\frac{\tilde{x}_a'-\tilde{x}_b'}{3\lambda_o}
	            \end{bmatrix}$
			    \Comment{Compute the particular solution} \label{lst:linepart}
		\State $\tilde{\mathbf{v}}_2 \gets$ rearrange $\tilde{\mathbf{v}}'$  into a $2\times3$ matrix
		\State $\mathbf{v}_{c2} \gets$ rearrange $\mathbf{v}_{c}$  into a $2\times3$ matrix
        \State $(U_1,\Sigma,U_2) \gets \SVD\left(\tilde{\mathbf{v}}_2\mathbf{v}_{c2}^{\top}\right)$\Comment{Compute the optimal rotation}
        \State $D \gets \mydiag(1,\det(U_1U_2))$
        \State $\mathbf{v}_t \gets \mathbf{v}_p + \bar{\mathbf{v}}$ \Comment{Compute the translation vector}
        \State $R \gets U_2 D U_1^\top$ 
        \State $\mathbf{v}_2 \gets  \mydiag(R,R,R)\mathbf{v}_c + \mathbf{v}_t$
            \Comment{Compute the optimal solution} \label{lst:linevercaseii}
	    \State \textbf{return} $\mathbf{v}_2, \mathbf{v}_c, \mathbf{v}_t$ 
	\EndFunction
    \end{algorithmic}
\end{algorithm}

Algorithm \ref{alg:case2} computes all the possible solutions for Case II. It achieves this by returning the rotational solution basis $\mathbf{v}_c$ (line~\ref{lst:linenull}), particular solution $\mathbf{v}_p$ (line~\ref{lst:linepart}) and offset $\bar{\mathbf{v}}$ (line~\ref{lst:linecentr}). 
With these three components, the user can generate any solution by choosing an angle $\theta$ in equation~(\ref{eq:case2}).
All solutions generated this way are theoretically equivalent and they all fulfil the area and orientation constraints.
However, because the input vertices might not be exactly colocated numerically (which is the theoretical prerequisite of Case II for a colocated vertices input), one of the solutions in the basis may stand out as having a lower score than any other one.
This solution may be, when the input vertices are close to each other, the optimal solution, even compared to Case I, owing to numerical round-off error. 
This solution is obtained by finding the optimal rotation for the cost function, the translation already being the optimal one, by solving:
\begin{equation}
	\min_{R \in SO(2)}  \|(\mathcal{R}\mathbf{v}_c+\mathbf{v}_t)-\tilde{\mathbf{v}}\|^2 \quad \mbox{with} \quad \mathcal{R}=\mydiag(R,R,R).
\end{equation}
This problem has a closed-form solution~\citep{arun_least-squares_1987-1}. 
We first rearrange $\mathbf{v}_c$ and $\tilde{\mathbf{v}}'$ into $2\times 3$ matrices $\mathbf{v}_{c2}$ and $\tilde{\mathbf{v}}_2$. We then compute the cross-covariance matrix $W =\mathbf{v}_{c2}\tilde{\mathbf{v}}_2^{\top}$ and its SVD $W = U_1\Sigma U_2^\top$.
The optimal orthogonal matrix, which could potentially contain a reflection in addition to the rotation, is $U_2 U_1^\top$. In order to preserve the triangle orientation we restrict $R$ to be a rotation only by setting $R = U_2 D U_1^\top$, where $D = \mydiag(1,\det(U_1 U_2^\top))$.
We finally use $R$ to generate the optimal solution $\mathbf{v}_2$ (line~\ref{lst:linevercaseii}). 

\subsection{Numerical Examples}
We show the results of our algebraic procedure in a series of illustrative examples presented in tables~\ref{tab:numexa1} and \ref{tab:numexa3}. Each row represents an example with a different type of input. The first column contains the input parameters (input vertices $\tilde{\mathbf{v}}$ with area $A^*(\tilde{\mathbf{v}})$, prescribed area $A_o$ and orientation $s$). The second column shows the cost $\mathscr{C}(\mathbf{v})$ and the generated area $A^*(\mathbf{v})$. For these examples, we opted to show the four solutions from Case I and the optimally rotated solution from Case II, and highlight the overall optimal solution. The third column shows the input triangle and the generated solution triangles. We draw a circle and a square in two of the vertices of the triangles to visualise the potential inversions. 

In table~\ref{tab:numexa1}, the inputs are random general triangles. For each example we want to find the optimal triangle that has the prescribed area and orientation. The first and second examples are triangles whose orientation matches the prescribed orientation, meaning that $\sign(A^*(\tilde{\mathbf{v}}))=s$, while the third example represents the opposite case, meaning that $\sign(A^*(\tilde{\mathbf{v}}))=-s$. The inputs in the second and third examples are identical, except for the prescribed orientation $s$.
In all three examples, for Case I, we observe that the first and second solutions respect the signed area constraint while the third and fourth do not. This happens because the third and fourth roots of the depressed quartic are complex and the vertices produced by these solutions are altered once we extract their real part in Algorithm 2. We also observe that the third and fourth solutions of Case I are the same. The reason is that they correspond to complex conjugate roots, and thus have the same real part. The solutions given by Case II also respect the area constraint. 
In the first and second examples, the vertices of the second solution of Case I and the solution of Case II are close to the input vertices, resulting in lower costs, however the solution of Case II is always an equilateral triangle. In both examples, the minimal cost is given by the second solution of Case I and is considered optimal. In the third example, the resulting vertices are not simple inversions of the previous solutions but new solutions that are accommodated to the prescribed orientation. In this case an optimal solution is also found, albeit at a higher cost.

In table~\ref{tab:numexa3}, the inputs are special configurations. 
The first example represents a flat triangle with colinear input vertices. In this instance, our algorithm behaves as expected, similarly to the examples with non-flat triangles, and returns an optimal solution (solutions 3 and 4 of Case I return a triangle considerably bigger than the input triangle). The second example represents an example where all the input vertices are colocated ($\tilde{v}_a=\tilde{v}_b=\tilde{v}_c$). In this instance, Case I solutions are ignored and the optimal solution is given by Case II. The solution given by Case II can be rotated at any angle but the cost remains constant. The third example is for an equilateral triangle whose orientation is the opposite of the prescribed orientation. In this example, none of the solutions give by Case I respects the area constraint and $\delta$ is very small (especially in solutions 3 and 4 of Case I where $\delta^\dagger$ is close to zero and thus returns a triangle \num{10e15} times bigger than the input triangle).
Similarly to our previous example, the optimal solution is given by Case II and it can be rotated at any angle with the cost remaining constant. In the end, our algorithmic procedure always computes the optimal solution for all six examples.

\begingroup
\setlength{\tabcolsep}{3pt} 
{\renewcommand{\arraystretch}{1.4}
\begin{table}
    \footnotesize
    \centering
	\begin{tabular}{V{4}>{\centering\arraybackslash}m{3.7cm}|>{\centering\arraybackslash}m{0.8cm}|>{\centering\arraybackslash}m{1.0cm}|>{\centering\arraybackslash}m{0.25cm}|>{\centering\arraybackslash}m{2.5cm} >{\centering\arraybackslash}m{2.5cm} >{\centering\arraybackslash}m{2.7cm}V{4}}
	\hlineB{4} 
	Input & \multicolumn{3}{c|}{Output} & \multicolumn{3}{c?}{Generated Triangles}  \\\hlineB{4} 
    \multirow{2}{*}{\vspace{-3mm}\shortstack[c]{Input: Negative \\ Oriented Triangle}} & Cost & $A^*(\mathbf{v})$ & \cellcolor{grayxgray} & 
	\multirow{4}{*}{\parbox[c]{2.8cm}{\includegraphics[scale=0.36]{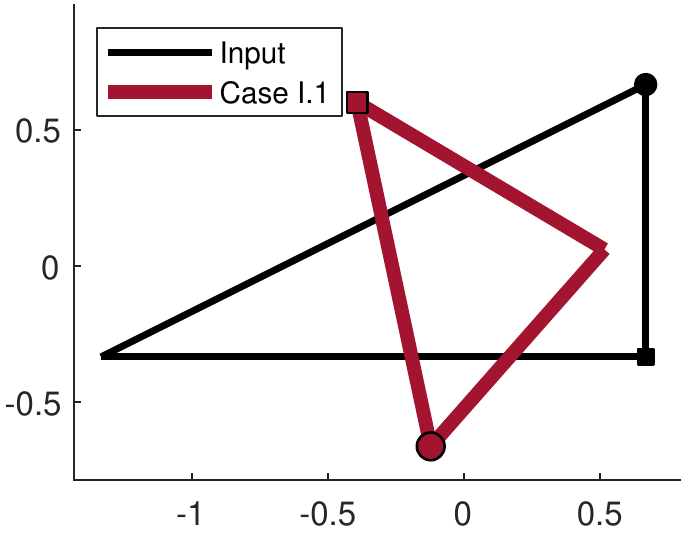}}} &
	\multirow{4}{*}{\parbox[c]{2.8cm}{\includegraphics[scale=0.36]{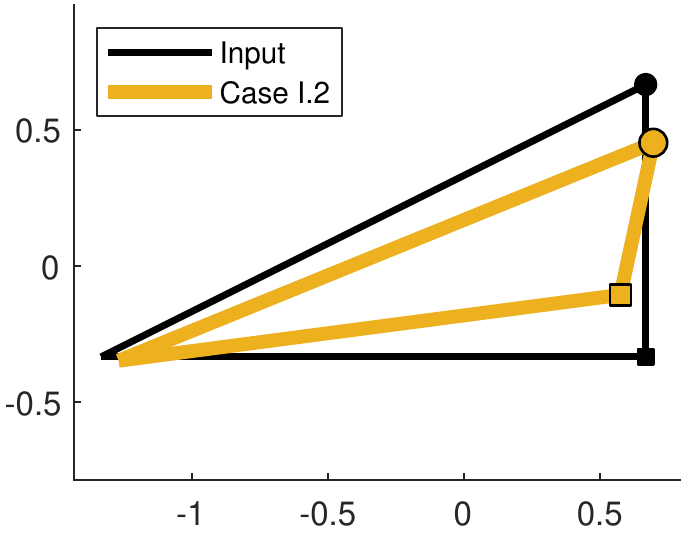}}} & \\ \cline{2-4} 
 	& \multicolumn{3}{c|}{Case I} & & & \\ \cline{2-4}
    $s=-1$ & 2.820 & -0.500 & \cellcolor{matred} & & &\multirow{4}{*}{\parbox[c]{2.8cm}{\includegraphics[scale=0.36]{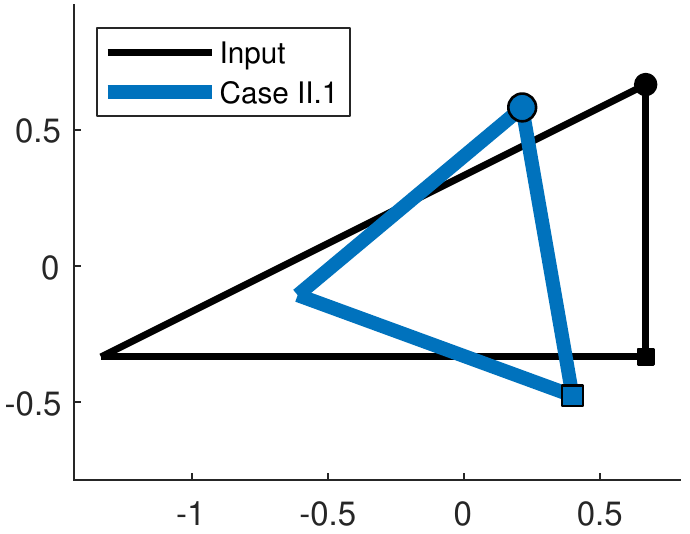}}} \\ \HHline{2pt}{|~*{3}{-}*{3}{~}|}
 	\multirow{3}{*}{$\tilde{\mathbf{v}} = \begin{bmatrix}  0.666 & 0.666 \\ 0.666 & -0.333 \\ -1.333 & -0.333 \end{bmatrix}$}
 	& \multicolumn{1}{V{4}c|}{0.334} & -0.500
 	& \multicolumn{1}{c V{4}}{\cellcolor{matyell}}
 	& & & \\ \HHline{2pt}{|~*{3}{-}*{3}{~}|}
    & 5.065 & 7.629 & \cellcolor{matpurp} & \multirow{4}{*}{\parbox[c]{2.8cm}{\includegraphics[scale=0.36]{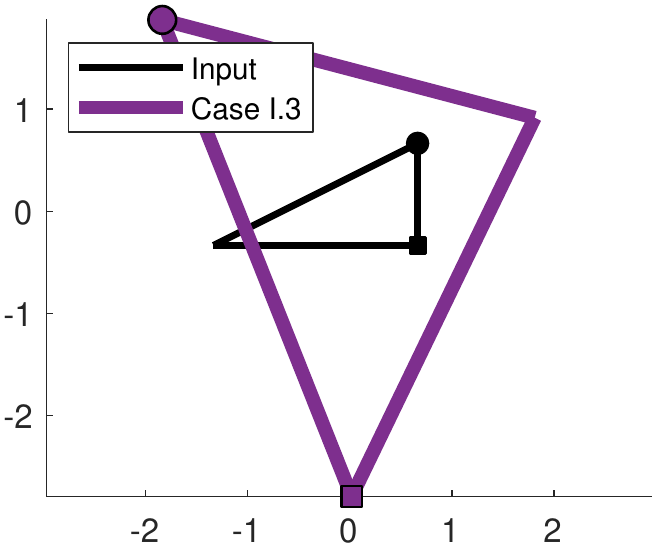}}} &
	\multirow{4}{*}{\parbox[c]{2.8cm}{\includegraphics[scale=0.36]{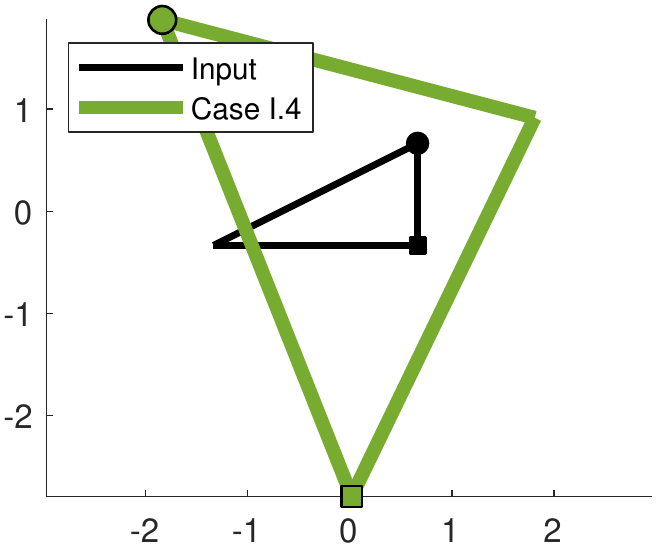}}} & \\ \cline{2-4}
    & 5.065 & 7.629 & \cellcolor{matgree} & & & 
     \\\cline{2-4}
    $A^*(\tilde{\mathbf{v}}) = -1.000$ & \multicolumn{3}{c|}{Case II} & & & \\\cline{2-4}
    $A_o = 0.500$ & 0.937 & -0.500 & \cellcolor{matblue} &  & & \\\hlineB{4} 
	\multirow{2}{*}{\vspace{-3mm}\shortstack[c]{Input: Positive \\ Oriented Triangle}} & Cost & $A^*(\mathbf{v})$ & \cellcolor{grayxgray} & 
	\multirow{4}{*}{\parbox[c]{2.8cm}{\includegraphics[scale=0.36]{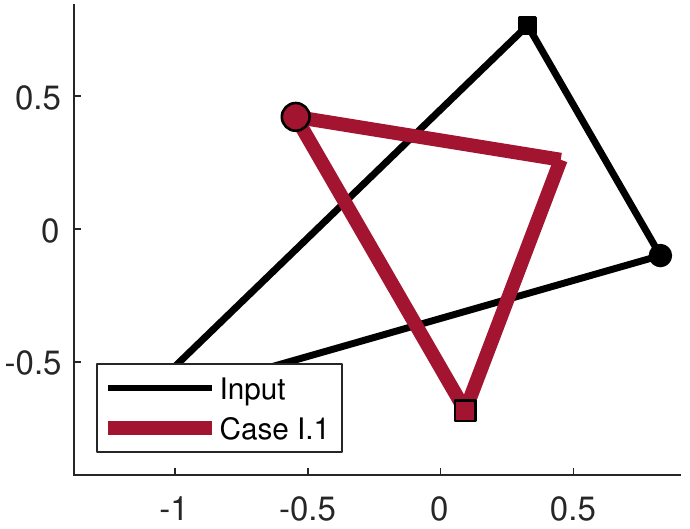}}} &
	\multirow{4}{*}{\parbox[c]{2.8cm}{\includegraphics[scale=0.36]{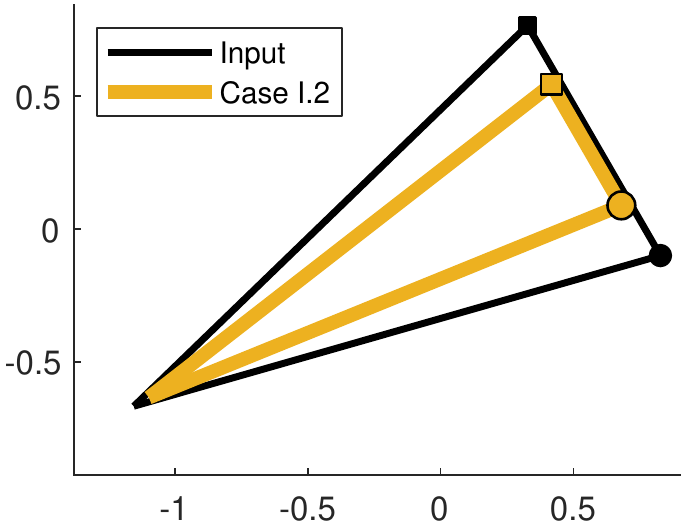}}} & \\ \cline{2-4} 
 	 & \multicolumn{3}{c|}{Case I} & & & \\ \cline{2-4}
 	 $s=1$ & 2.785 & 0.500 & \cellcolor{matred} & & &\multirow{4}{*}{\parbox[c]{2.8cm}{\includegraphics[scale=0.36]{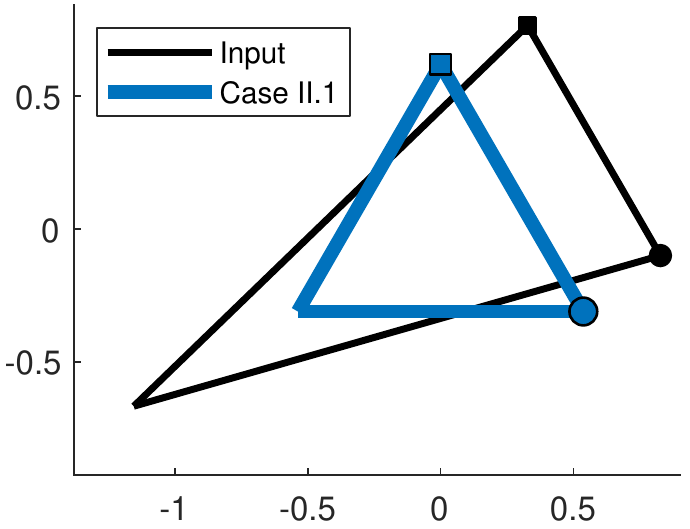}}} \\\HHline{2pt}{|~*{3}{-}*{3}{~}|}
 	 \multirow{3}{*}{$\tilde{\mathbf{v}} = \begin{bmatrix}  0.827 &  -0.100 \\ 0.327 & 0.766 \\ -1.155 & -0.667 \end{bmatrix}$}
 	& \multicolumn{1}{V{4}c|}{0.345} & 0.500
 	& \multicolumn{1}{c V{4}}{\cellcolor{matyell}} & & & \\ \HHline{2pt}{|~*{3}{-}*{3}{~}|}
    & 5.078 & -7.919 & \cellcolor{matpurp} & \multirow{4}{*}{\parbox[c]{2.8cm}{\includegraphics[scale=0.36]{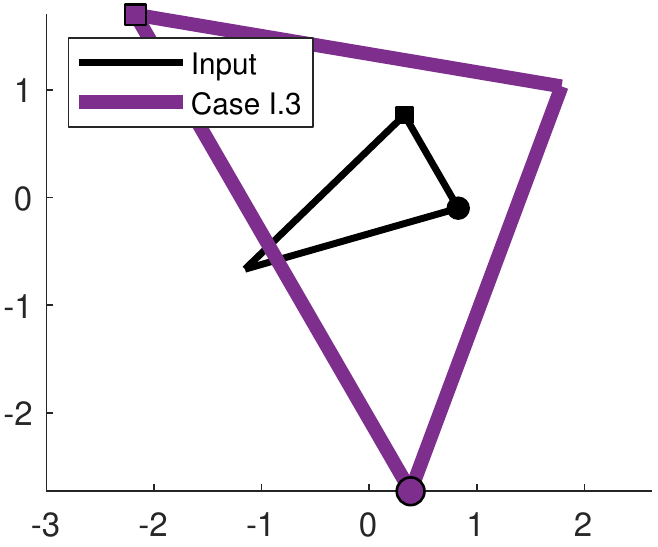}}} &
	\multirow{4}{*}{\parbox[c]{2.8cm}{\includegraphics[scale=0.36]{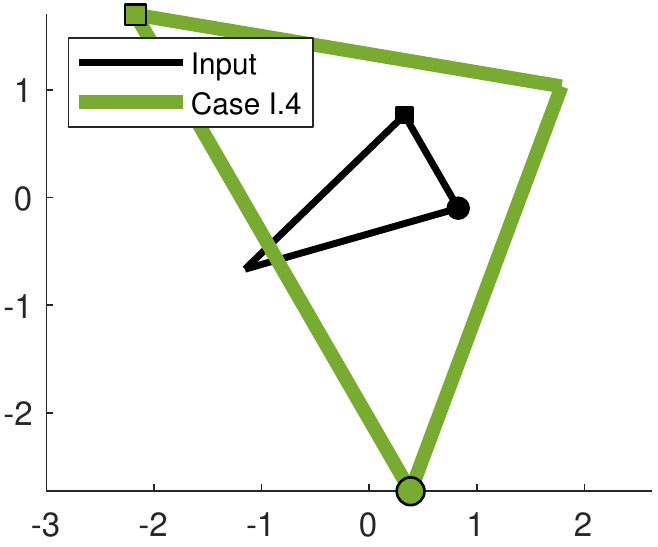}}} & \\ \cline{2-4}
    & 5.078 & -7.919 & \cellcolor{matgree} & & & 
    \\\cline{2-4}
    $A^*(\tilde{\mathbf{v}})=1.000$ & \multicolumn{3}{c|}{Case II} & & & \\\cline{2-4}
    $A_o = 0.500$ & 0.875  & 0.500 & \cellcolor{matblue} &  & & \\\hlineB{4} 
    \multirow{2}{*}{\vspace{-3mm}\shortstack[c]{Input: Positive \\ Oriented Triangle}} & Cost & $A^*(\mathbf{v})$ & \cellcolor{grayxgray} & 
	\multirow{4}{*}{\parbox[c]{2.8cm}{\centering\includegraphics[scale=0.36]{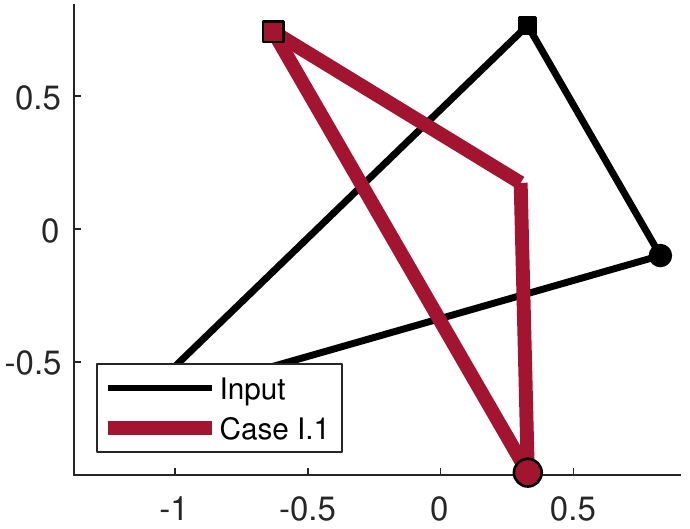}}} &
	\multirow{4}{*}{\parbox[c]{2.8cm}{\centering\includegraphics[scale=0.36]{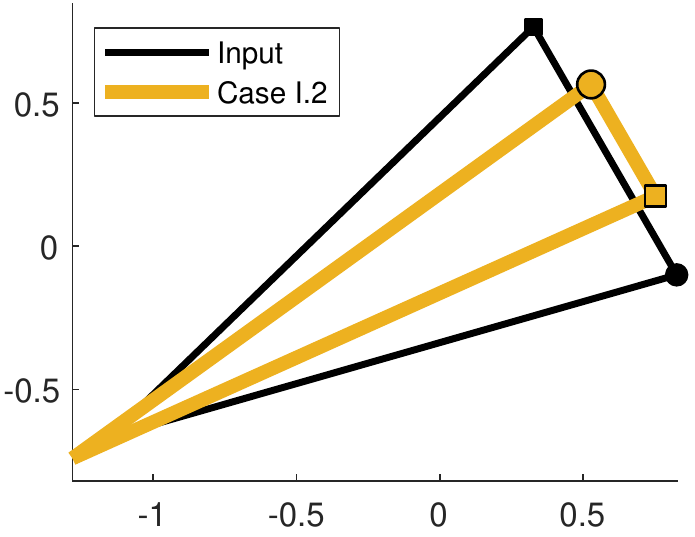}}} & \\ \cline{2-4} 
 	 & \multicolumn{3}{c|}{Case I} & & & \\ \cline{2-4}
    $s=-1$ & 2.158 & -0.500 & \cellcolor{matred} & & &\multirow{4}{*}{\parbox[c]{2.8cm}{\centering\includegraphics[scale=0.36]{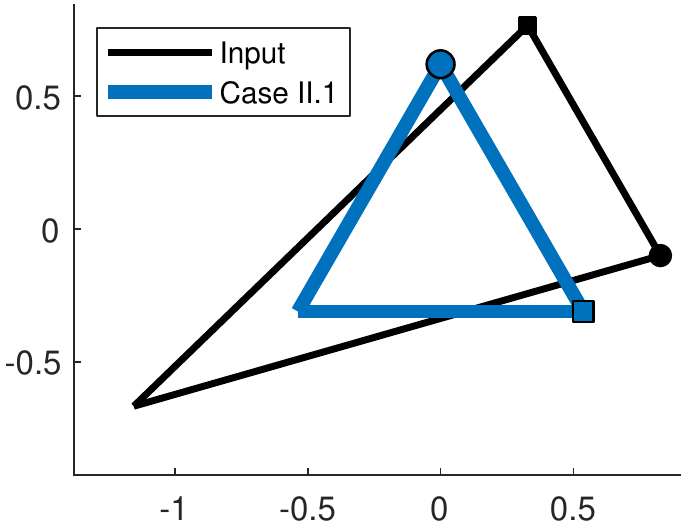}}} \\ \HHline{2pt}{|~*{3}{-}*{3}{~}|}
 	\multirow{3}{*}{$\tilde{\mathbf{v}} = \begin{bmatrix}  0.827 &  -0.100 \\ 0.327 & 0.766 \\ -1.155 & -0.667 \end{bmatrix}$} & \multicolumn{1}{V{4}c|}{1.041} & -0.500
 	& \multicolumn{1}{c V{4}}{\cellcolor{matyell}}
 	& & & \\ \HHline{2pt}{|~*{3}{-}*{3}{~}|}
    & 40.35 & 648.58 & \cellcolor{matpurp} & \multirow{4}{*}{\parbox[c]{2.8cm}{\centering\includegraphics[scale=0.36]{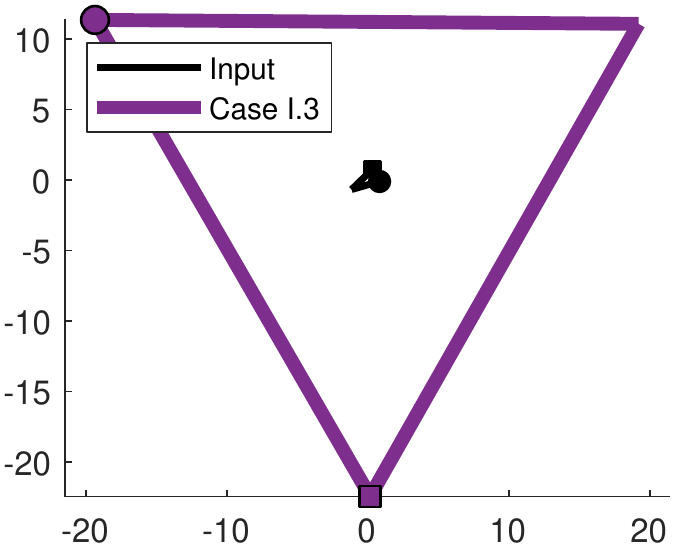}}} &  
	\multirow{4}{*}{\parbox[c]{2.8cm}{\centering\includegraphics[scale=0.36]{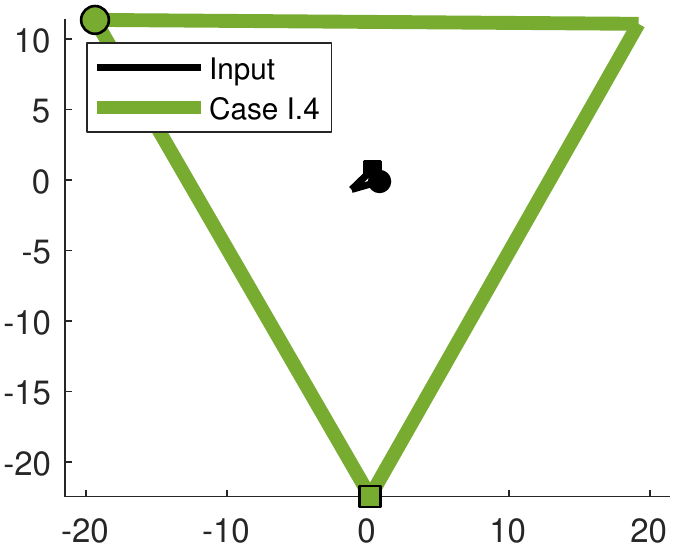}}} & \\ \cline{2-4}
     & 40.35 & 648.58 & \cellcolor{matgree} & & & 
     \\\cline{2-4}
    $A^*(\tilde{\mathbf{v}})=1.000$ & \multicolumn{3}{c|}{Case II} & & & \\\cline{2-4}
    $A_o = 0.500$ & 1.707 & -0.500 & \cellcolor{matblue} &  & & \\\hlineB{4} 
	\end{tabular}
	\caption{Numerical examples with single triangles. The left column shows the type of input, prescribed orientation $s$, input vertices $\tilde{\mathbf{v}}$, input area $A^*(\tilde{\mathbf{v}})$ and prescribed area $A_o$. The middle column shows the cost and area obtained for the triangles computed by our algebraic procedure. All four solutions from Case I and one solution from Case II are shown, assigned a colour for visual representation and the optimal solution is highlighted. The right column shows the computed triangles superimposed with the input triangle.}
	\label{tab:numexa1}
\end{table}}

\begingroup
\setlength{\tabcolsep}{3pt} 
{\renewcommand{\arraystretch}{1.4}
\begin{table}
    \footnotesize
    \centering
	\begin{tabular}{V{4}>{\centering\arraybackslash}m{3.8cm}|>{\centering\arraybackslash}m{1.1cm}|>{\centering\arraybackslash}m{1.1cm}|>{\centering\arraybackslash}m{0.25cm}|>{\centering\arraybackslash}m{2.4cm} >{\centering\arraybackslash}m{2.4cm} >{\centering\arraybackslash}m{2.4cm}V{4}}
	\hlineB{4} 
	Input & \multicolumn{3}{c|}{Output} & \multicolumn{3}{c?}{Generated Triangles}  \\\hlineB{4}
     \multirow{2}{*}{\vspace{-3mm}\shortstack[c]{Input: Colinear\\ Vertices}} & Cost & $A^*(\mathbf{v})$ & \cellcolor{grayxgray} & 
	\multirow{4}{*}{\parbox[c]{2.8cm}{\centering\includegraphics[scale=0.4]{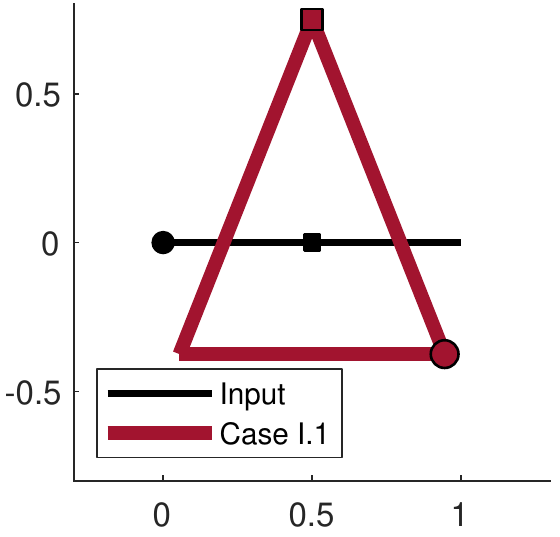}}} &
	\multirow{4}{*}{\parbox[c]{2.8cm}{\centering\includegraphics[scale=0.4]{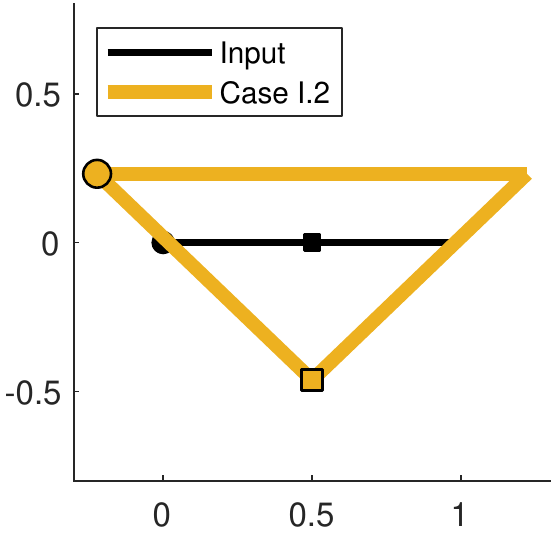}}} & \\ \cline{2-4} 
 	 & \multicolumn{3}{c|}{Case I} & & & \\ \cline{2-4}
 	 $s=1$  & 1.621 & 0.500 & \cellcolor{matred} & & &\multirow{4}{*}{\parbox[c]{2.8cm}{\centering\includegraphics[scale=0.4]{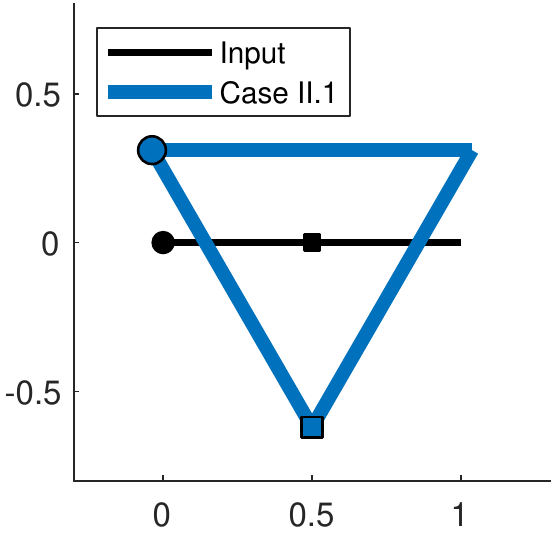}}} \\ \HHline{2pt}{|~*{3}{-}*{3}{~}|}
 	\multirow{3}{*}{$\tilde{\mathbf{v}} = \begin{bmatrix}  0.000 & 0.000 \\ 0.500 & 0.000 \\ 1.000 & 0.000 \end{bmatrix}$}
 	 & \multicolumn{1}{V{4}c|}{0.647} & 0.500
 	 & \multicolumn{1}{c V{4}}{\cellcolor{matyell}} & & & \\\HHline{2pt}{|~*{3}{-}*{3}{~}|}
     & 88.049 & -3318 & \cellcolor{matpurp} & \multirow{4}{*}{\parbox[c]{2.8cm}{\centering\includegraphics[scale=0.38]{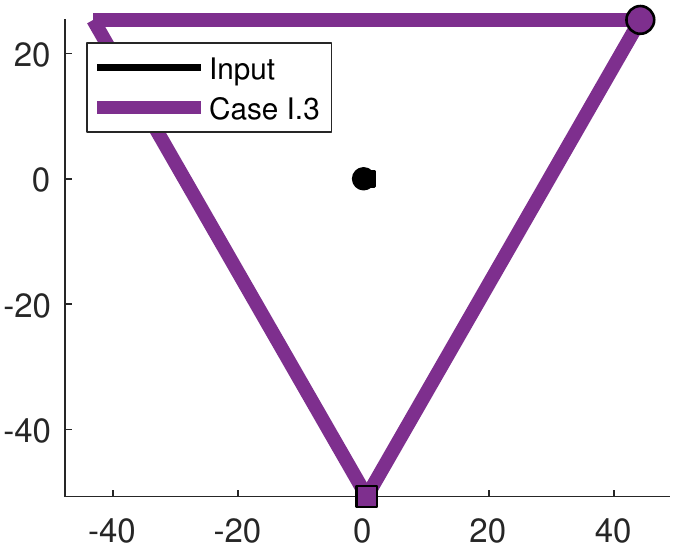}}} &
	\multirow{4}{*}{\parbox[c]{2.8cm}{\centering\includegraphics[scale=0.38]{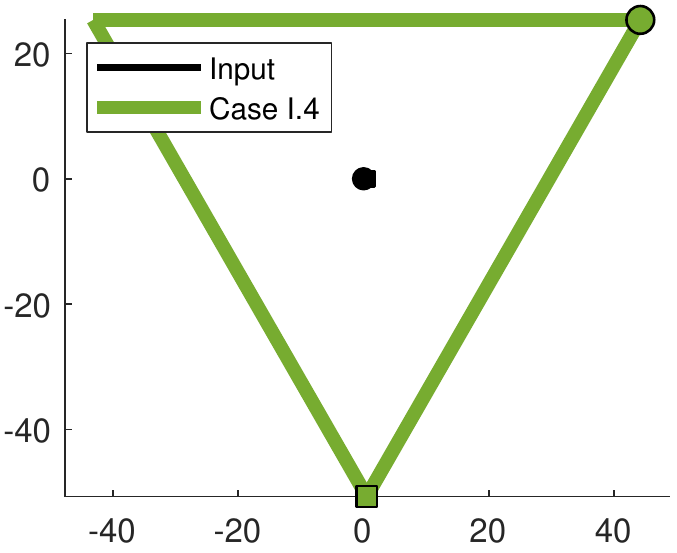}}} & \\ \cline{2-4}
     & 88.049 & -3318 & \cellcolor{matgree} & & &
     \\\cline{2-4}
    $A^*(\tilde{\mathbf{v}})=0.000$ & \multicolumn{3}{c|}{Case II} & & & \\\cline{2-4}
    $A_o = 0.500$ & 0.762 & 0.500 & \cellcolor{matblue} & & & \\\hlineB{4} 
    \multirow{2}{*}{\vspace{-3mm}\shortstack[c]{Input: Colocated\\ Vertices}} & Cost & $A^*(\mathbf{v})$ & \cellcolor{grayxgray} & 
	\multirow{4}{*}{\parbox[c]{2.8cm}{\includegraphics[scale=0.4]{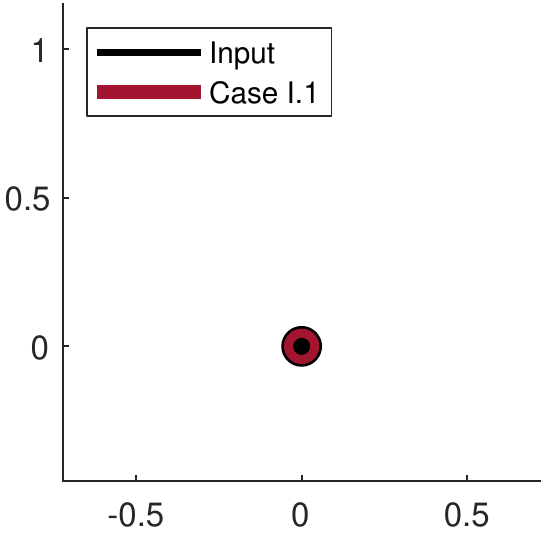}}} &
	\multirow{4}{*}{\parbox[c]{2.8cm}{\includegraphics[scale=0.4]{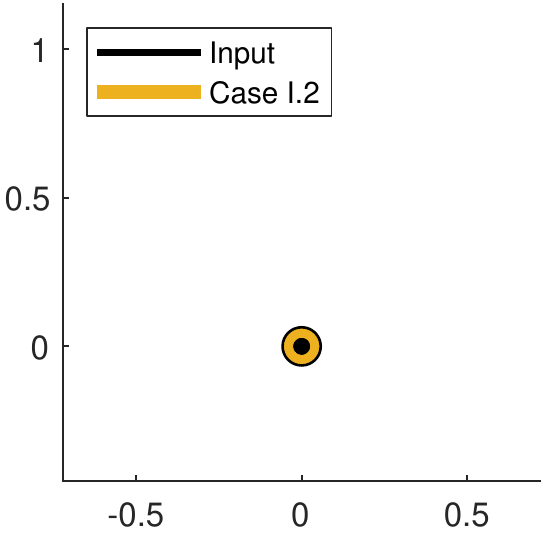}}} & \\ \cline{2-4} 
 	 & \multicolumn{3}{c|}{Case I} & & & \\ \cline{2-4}
 	 $s=1$ & 0.000 & 0.000 & \cellcolor{matred} & & &\multirow{4}{*}{\parbox[c]{2.8cm}{\includegraphics[scale=0.4]{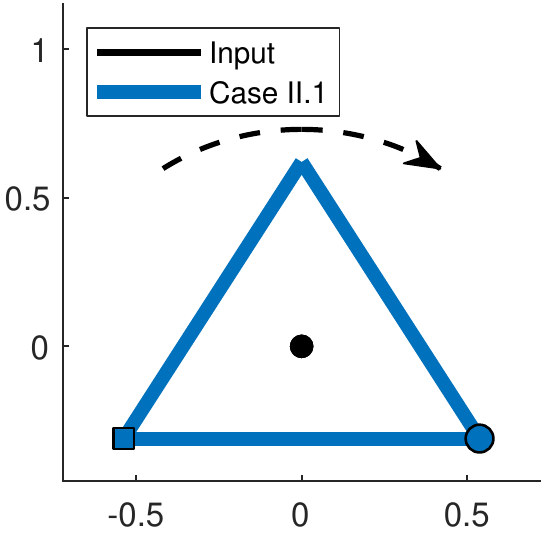}}} \\ \cline{2-4}
 	\multirow{3}{*}{$\tilde{\mathbf{v}} = \begin{bmatrix}  0.000 & 0.000 \\ 0.000 & 0.000 \\ 0.000 & 0.000  \end{bmatrix}$} 
 	& 0.000 & 0.000 & \cellcolor{matyell} & & & \\ \cline{2-4}
    & 0.000 & 0.000 & \cellcolor{matpurp} & \multirow{4}{*}{\parbox[c]{2.8cm}{\includegraphics[scale=0.4]{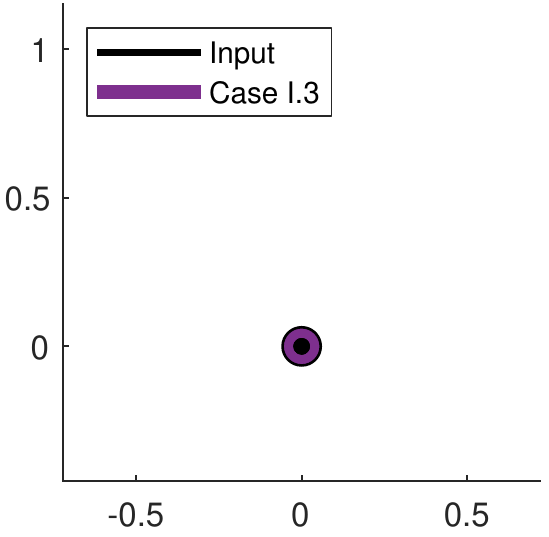}}} &
	\multirow{4}{*}{\parbox[c]{2.8cm}{\includegraphics[scale=0.4]{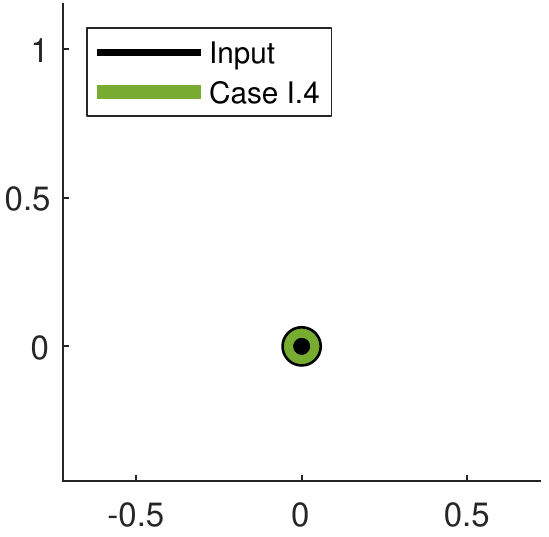}}} & \\ \cline{2-4}
     & 0.000 & 0.000 & \cellcolor{matgree} & & &
     \\\cline{2-4}
    $A^*(\tilde{\mathbf{v}})=0.000$ & \multicolumn{3}{c|}{Case II} & & & \\
    \HHline{2pt}{|~*{3}{-}*{3}{~}|}
    $A_o = 0.500$ & \multicolumn{1}{V{4}c|}{1.075} & 0.500
 	 & \multicolumn{1}{c V{4}}{\cellcolor{matblue}} &  & & \\\hlineB{4}
 	\multirow{2}{*}{\vspace{-3mm}\shortstack[c]{Input: 
 	Positive Oriented  \\ Equilateral Triangle}} & Cost & $A^*(\mathbf{v})$ & \cellcolor{grayxgray} & 
	\multirow{4}{*}{\parbox[c]{2.8cm}{\centering\includegraphics[scale=0.4]{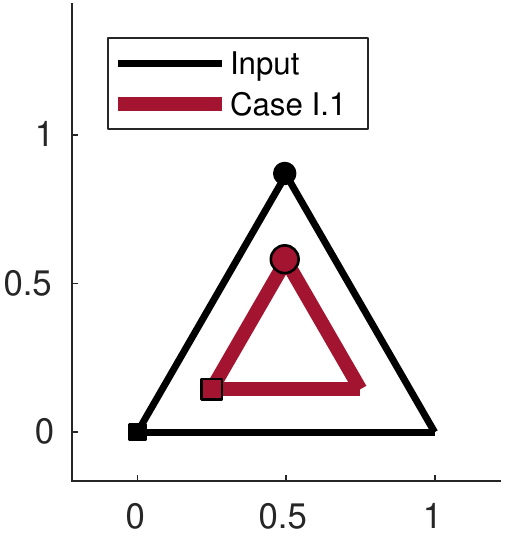}}} &
	\multirow{4}{*}{\parbox[c]{2.8cm}{\centering\includegraphics[scale=0.4]{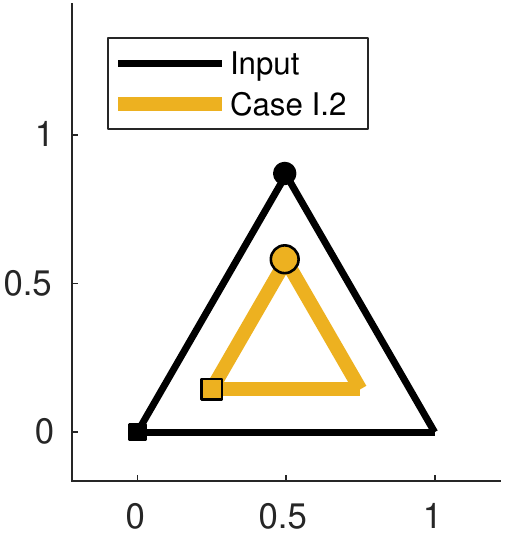}}} & \\ \cline{2-4} 
 	 & \multicolumn{3}{c|}{Case I} & & & \\ \cline{2-4}
 	 $s=-1$ & 0.500 & 0.108 & \cellcolor{matred} & & &\multirow{4}{*}{\parbox[c]{2.8cm}{\centering\includegraphics[scale=0.4]{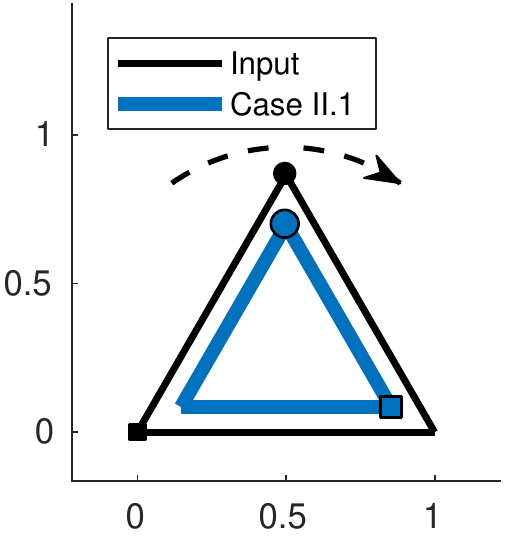}}} \\ \cline{2-4} 
 	\multirow{3}{*}{$\tilde{\mathbf{v}} = \begin{bmatrix}  0.000 & 0.000 \\ 1.000 & 0.000 \\ 0.500 & 0.866 \end{bmatrix}$}
 	& 0.500 & 0.108
 	& \cellcolor{matyell}
 	& & & \\ \cline{2-4}
    & \num[exponent-product = \cdot]{3.60e+15} & \num{5.62e+30} & \cellcolor{matpurp} & \multirow{4}{*}{\parbox[c]{2.8cm}{\centering\includegraphics[scale=0.4]{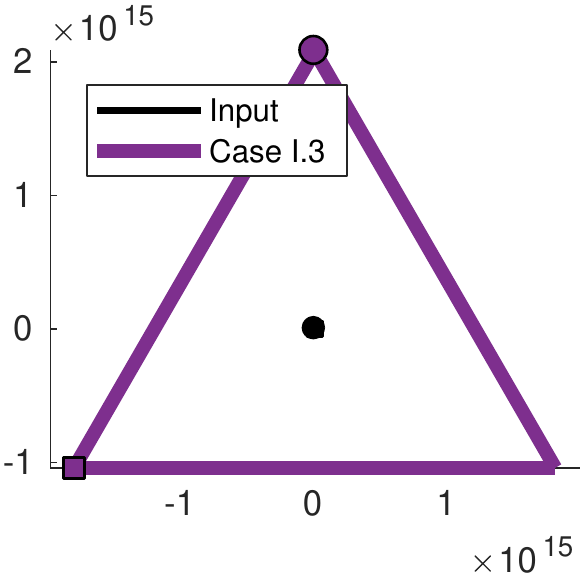}}} & 
	\multirow{4}{*}{\parbox[c]{2.8cm}{\centering\includegraphics[scale=0.4]{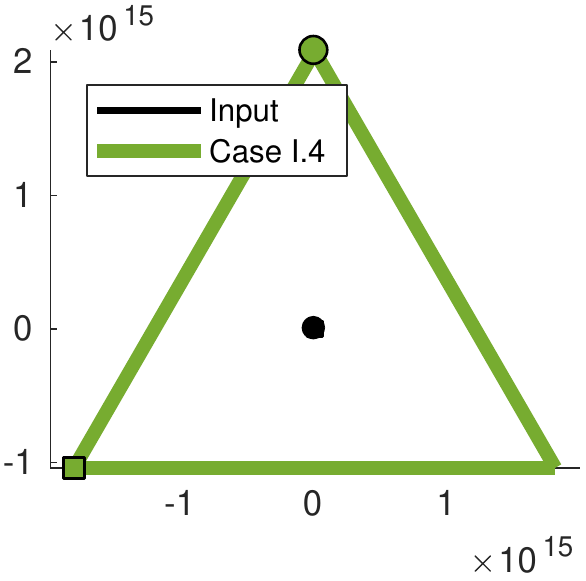}}} & \\ \cline{2-4}
     & \num{3.60e+15} & \num{5.62e+30} & \cellcolor{matgree} & & & 
     \\\cline{2-4}
    $A^*(\tilde{\mathbf{v}})=0.433$ & \multicolumn{3}{c|}{Case II} & & & \\\HHline{2pt}{|~*{3}{-}*{3}{~}|}
    $A_o = 0.216$ & \multicolumn{1}{V{4}c|}{1.225} & -0.216
 	& \multicolumn{1}{c V{4}}{\cellcolor{matblue}} 
    & & & \\\hlineB{4} 
    \end{tabular}
	\caption{Numerical examples with special triangles.
	The left column shows the type of input, prescribed orientation $s$, initial vertices $\tilde{\mathbf{v}}$, input area $A^*(\tilde{\mathbf{v}})$ and prescribed area $A_o$. The middle column shows the cost and area obtained for the triangles computed by our algebraic procedure. All four solutions from Case I and one solution from Case II are shown and assigned a colour for visual representation. The right column shows the computed triangles superimposed with the input triangle. In the case of colocated or equilateral input vertices, the generated equilateral triangle can be rotated arbitrarily without changing the cost.}
	\label{tab:numexa3}
\end{table}}

\section{Area-based 2D Mesh Editing}
We use our method in triangular 2D mesh editing. The implementation is similar to PBD~\citep{muller_position_2007} but instead of linearising the area constraint, we perform an optimal projection for each triangle in the mesh. The $N_p$ mesh vertices are in $\mathbf{P} \in \mathbb{R}^{N_p\times2}$ and the $N_t$ triangles in $\mathbf{M} \in \mathbb{R}^{N_t\times3}$ with prescribed areas $\mathbf{A}_o \in \mathbb{R}^{N_t}$ and prescribed orientation $\sign(A^*(\tilde{\mathbf{v}}))$. The implementation is given in Algorithm \ref{alg:meshedit}. 
\begin{algorithm}
    \caption{Prescribed Area Preservation 2D Mesh PBD}
    \label{alg:meshedit}
    \begin{algorithmic}[1]
    \Require $\mathbf{P}$ - mesh vertices, $\mathbf{M}$ - triangles' indices, $\mathbf{A}_o$ - prescribed areas, $T_c$ - displacement threshold, $E$ - area error tolerance
    \Ensure $\mathbf{P}$ - edited mesh points
    \State $C \gets \infty$ 
    \While{$C \geq T_c$} \Comment{Iterate until convergence}
    \State $\tilde{\mathbf{P}} \gets \mathbf{P}$    \Comment{Copy the mesh vertices}
    \For{$t \gets 1,\dots,N_t$}
    \State $p \gets \mathbf{M}(t,:)$ \Comment{Indices of the triangle}
    \State $\tilde{\mathbf{v}} \gets \mathbf{P}(p,:)$ \Comment{Coordinates of the
    triangle}
    \State $A_o \gets \mathbf{A}_o(t)$ \Comment{Prescribed area of the triangle}
    \State $s \gets \sign(A^*(\tilde{\mathbf{v}}))$ \Comment{Orientation of the triangle}
    \State $\mathbf{v}_o \gets$ \Call{OTTPAO}{$\tilde{\mathbf{v}},A_o,s, E$}
    \Comment{Optimal projection}
    \State $\mathbf{P}(p,:) \gets \mathbf{v}_o$ \Comment{Update mesh points}
    \EndFor
    \State $C \gets \frac{1}{N_t}\sum_{i=1}^{N_t}\|\tilde{\mathbf{P}}(i,:)-\mathbf{P}(i,:)\|$ 
    \Comment{Average displacement}
    \EndWhile
    \end{algorithmic}
\end{algorithm}

\subsection{Shape Dataset}
For our dataset, we used Distmesh~\citep{persson_simple_2004} to create a set of synthetic triangular meshes. As shown in Figure~\ref{fig:data}, the shapes ranged from simple convex shapes to nonconvex shapes with different levels of complexity. The dataset was divided into two subsets: one subset of 8 coarse meshes composed approximately of 100 triangles and one subset of 8 fine meshes composed approximately 1000 triangles. The meshes were designed so the distances between connected vertices were approximately the same. We use the areas of each of the triangles as prescribed areas $A_o$.

\begin{figure}
  \begin{subfigure}[b]{0.24\textwidth}
    \includegraphics[width=\textwidth]{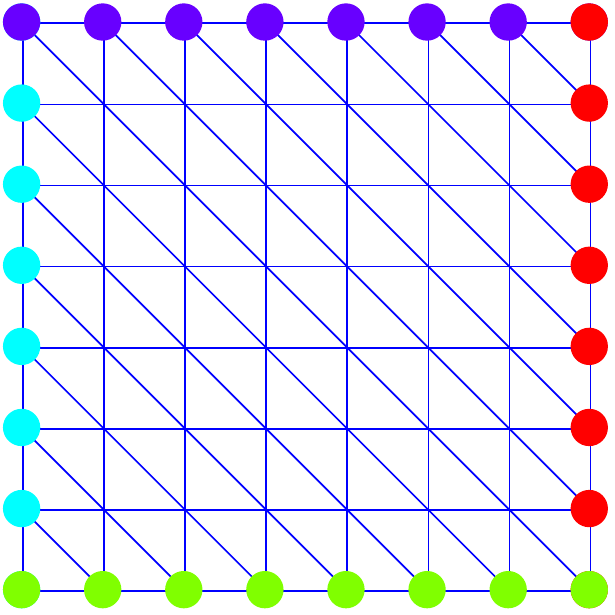}
  \end{subfigure}
  \begin{subfigure}[b]{0.24\textwidth}
    \includegraphics[width=\textwidth]{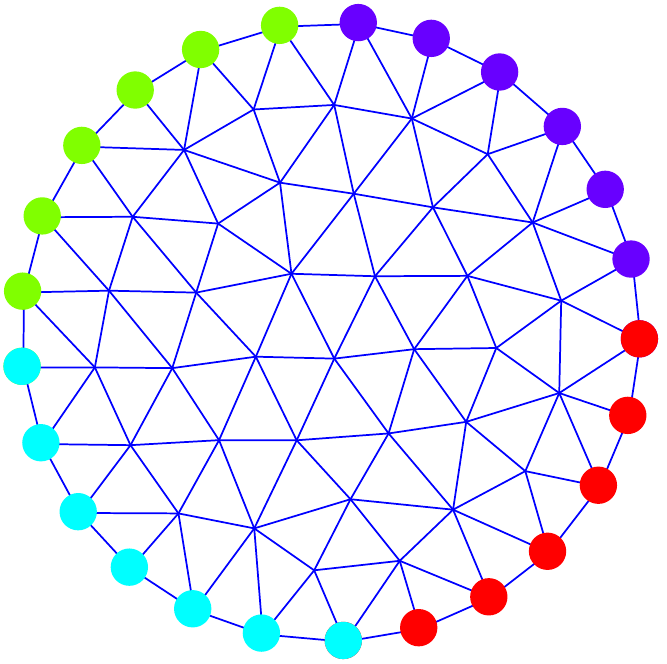}
  \end{subfigure}
  \begin{subfigure}[b]{0.24\textwidth}
    \includegraphics[width=\textwidth]{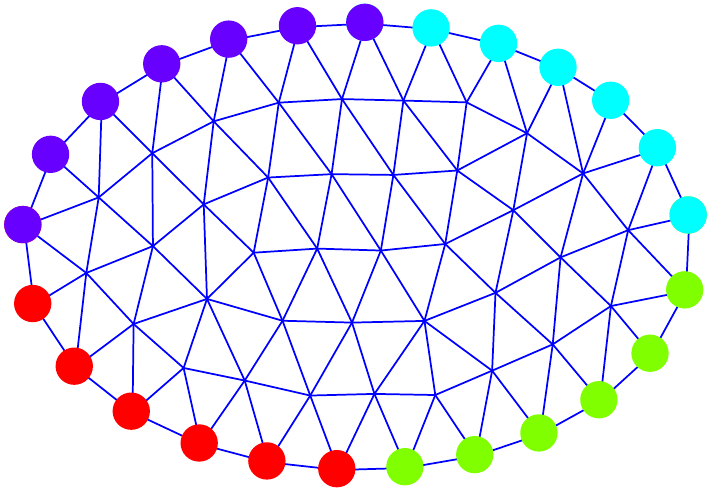}
    \vphantom{\includegraphics[width=0.24\textwidth,valign=c]{images/square}}
  \end{subfigure}
  \begin{subfigure}[b]{0.24\textwidth}
    \includegraphics[width=\textwidth]{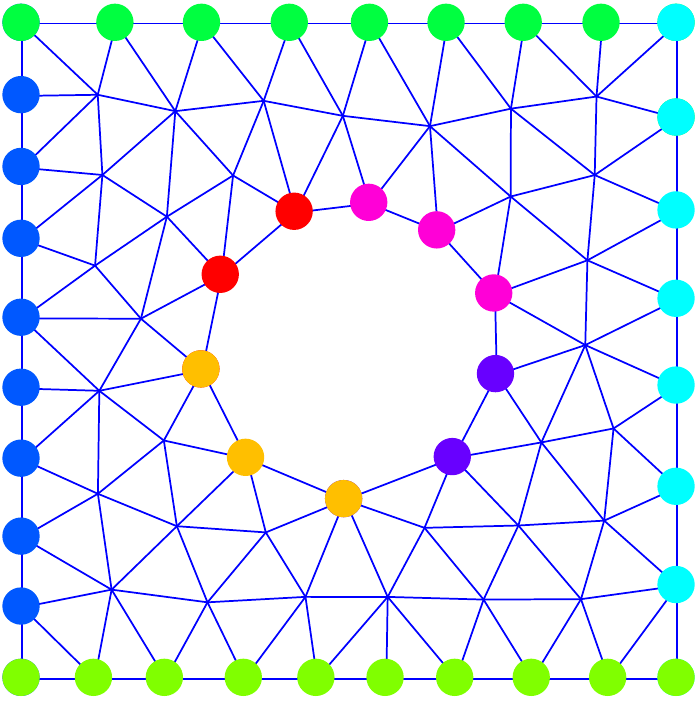}
  \end{subfigure} \\
  \begin{subfigure}[b]{0.24\textwidth}
    \includegraphics[width=\textwidth]{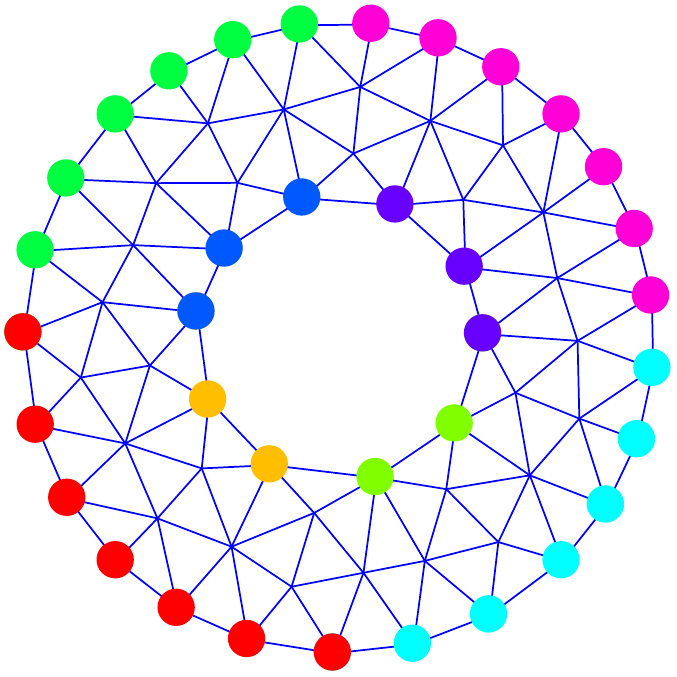}
  \end{subfigure}
  \begin{subfigure}[b]{0.24\textwidth}
    \includegraphics[width=\textwidth]{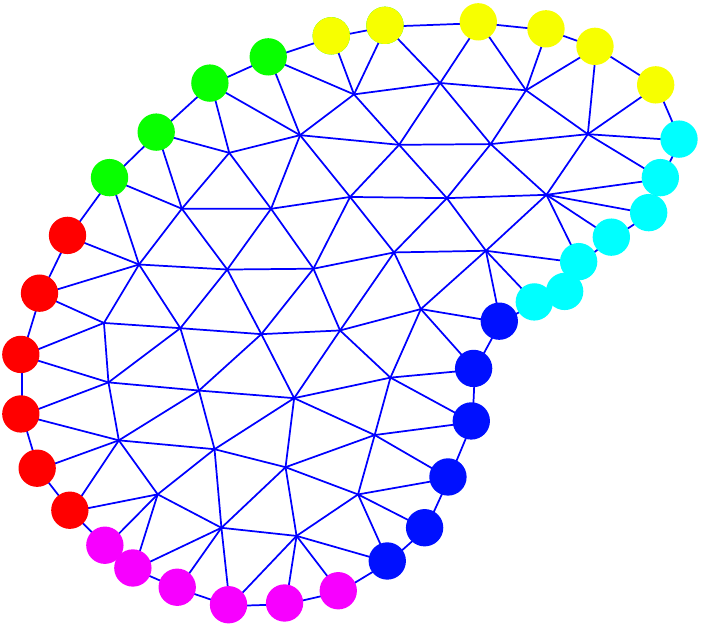}
  \end{subfigure}
  \begin{subfigure}[b]{0.24\textwidth}
    \includegraphics[width=\textwidth]{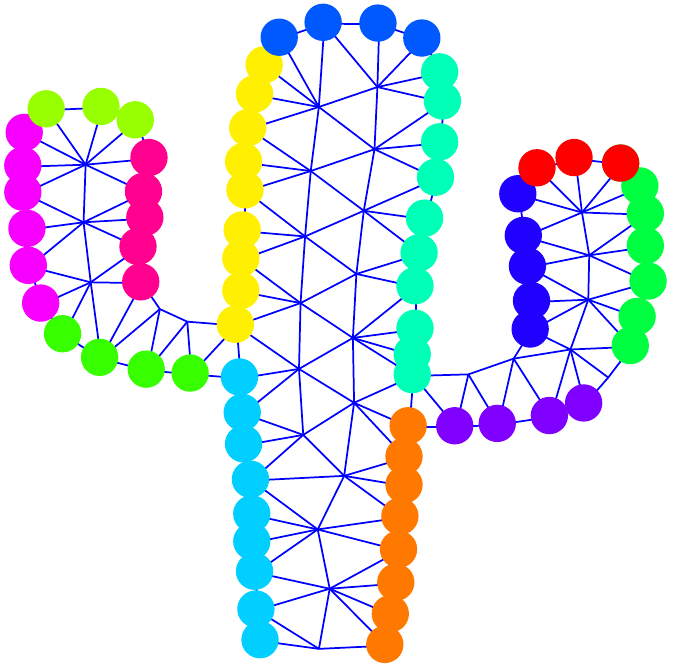}
  \end{subfigure}
  \begin{subfigure}[b]{0.24\textwidth}
    \includegraphics[width=\textwidth]{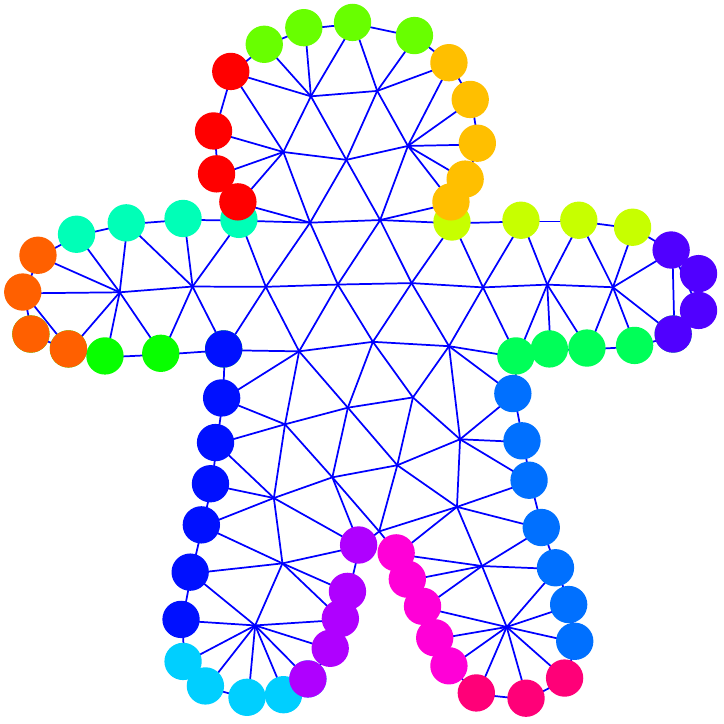}
  \end{subfigure} \\
  \caption{Shapes from the coarse subset of our synthetic polygonal mesh dataset. The vertices located at the edge are divided in sets of connected vertices
  represented with the same colour.}
  \label{fig:data}
\end{figure}

\subsection{Generating Deformation Constraints}
For our experiments we require to calculate the magnitude of the initial deformation. Since we deal with nonconvex shapes of different levels of complexity, size and orientation, we normalise this magnitude with respect to the maximum distance between two vertices in the direction of maximum variance. We treat the meshes as 2D point clouds and calculate the 95\% confidence ellipse that surrounds the vertices~\citep{friendly_elliptical_2013}. We take the maximum distance $D$ as two times the length of the semi-major axis of the ellipse. We divide the vertices located at the edge of the polygonal mesh in sets of connected vertices that represent a line or curve, as shown in Figure~\ref{fig:data}. Then, we apply an initial deformation by translating a given set of edge vertices in a random direction that does not cause self-collision. The magnitude of this translation is a fraction of $D$. 

\subsection{Methodology}
We test the effectiveness of PBD-opt by applying an initial deformation to a synthetic mesh and measuring the number iterations it takes to converge compared to PBD-lin. Convergence is achieved when the average displacement of the mesh’s vertices is lower than some displacement threshold $T_c$. 
We test both methods with the coarse and fine mesh datasets. For each dataset, we perform 200 random deformations per mesh (1600 deformations in total). We applied initial deformations to the meshes of $5\%$, $10\%$ and $20\%$ of the maximum inter-vertex distance $D$. We measure the convergence speed as the number of iterations it takes to reach 3 different displacement thresholds $T_c$ at $5\%$, $2.5\%$ and $1\%$ of $D$. A run stops when a method's cost reaches the lowest threshold ($T_c=1\%$) or after it reaches a stopping time (10,000 iterations\footnote{The stopping time choice was arbitrary. However considering that the  convergence speed  of PBD-opt was lower than 1000 iterations, the stopping time is sufficiently high for our experiments}).

We illustrate our methodology with an example presented in figure~\ref{fig:convgraph}. The input is a circle shaped coarse mesh with an initial random deformation of $10\%$ of maximum inter-vertex distance $D$. As can be seen in figure~\ref{fig:convgraph}a, the initial displacement cost of PBD-lin is lower than PBD-opt, however after some iterations the cost of PBD-opt becomes lower while the cost of PBD-lin takes many more iterations to converge. 
In figure~\ref{fig:convgraph}b, we show the evolution of the triangle area preservation constraint by comparing the difference of between the mesh triangle areas and prescribed areas. We observe that, by the time PBD-lin reaches convergence, its area difference is larger compared to PBD-opt. 

\begin{figure}
    \centering
    \begin{subfigure}[b]{0.49\textwidth}
    \includegraphics[width=\textwidth]{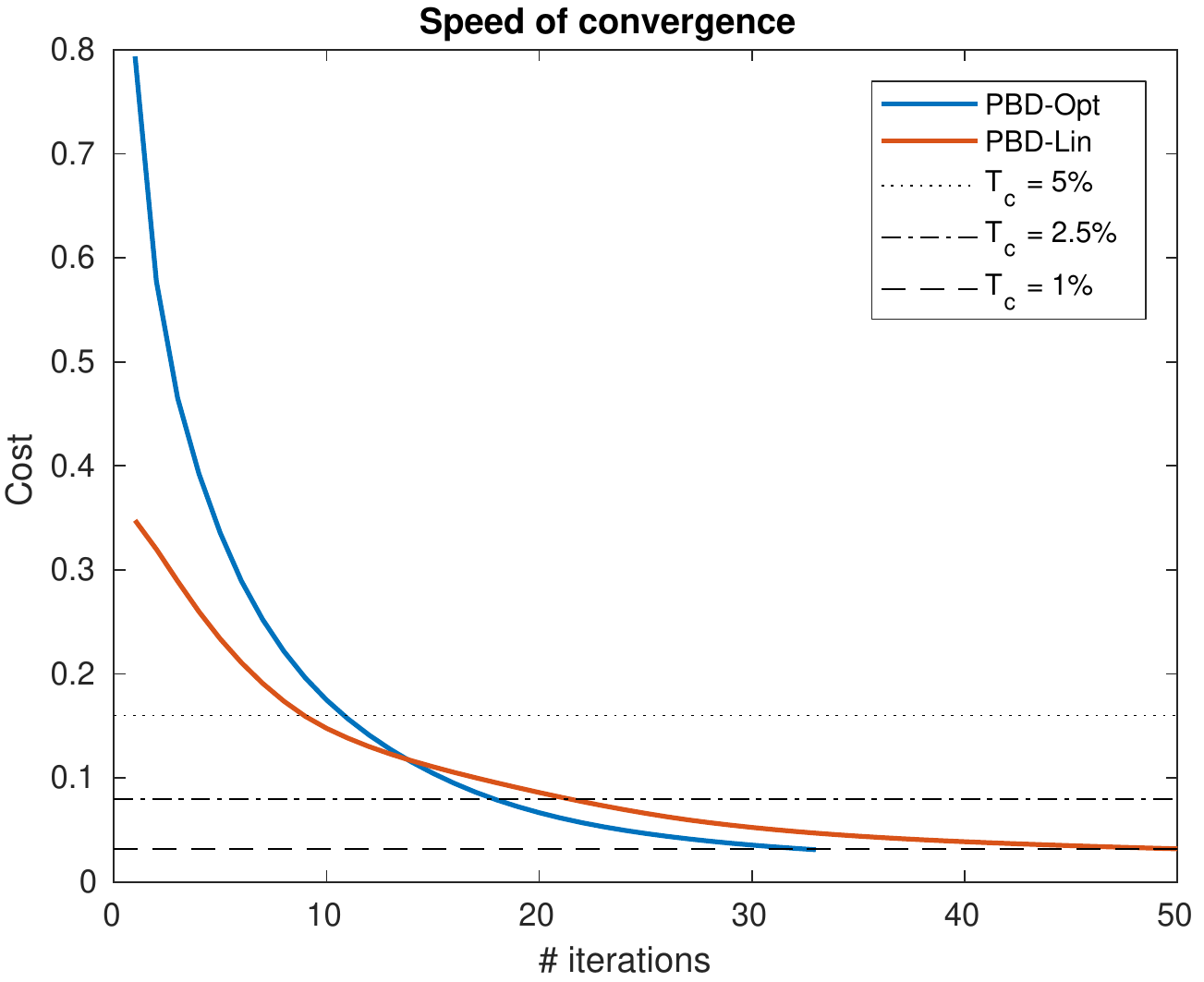}
    \caption{Evolution of cost}
    \end{subfigure}
    \begin{subfigure}[b]{0.49\textwidth}
    \includegraphics[width=\textwidth]{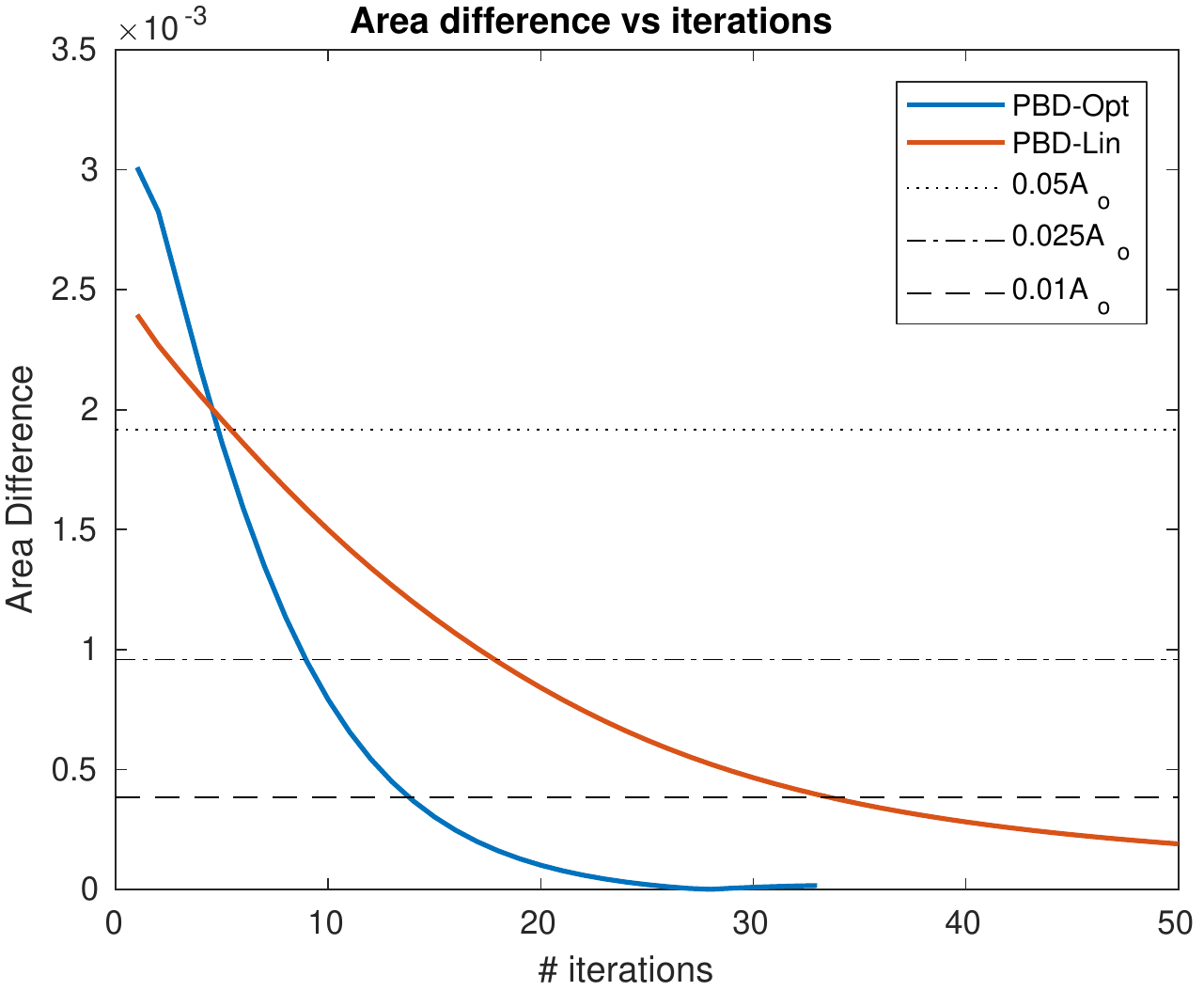}
    \caption{Evolution of area difference}
    \end{subfigure}
    \caption{Displacement cost and area difference comparison of mesh-editing for both PBD-Lin (red) and PBD-opt (blue) across the iterations. (a) Displacement thresholds are used to quantify the evolution of convergence speed (dotted and dashed black lines). (b) Area thresholds are used to quantify the evolution of the constraint convergence speed.}
    \label{fig:convgraph}
\end{figure}

\subsection{Results} 
Results can be found in figures~\ref{fig:datacoarse} and \ref{fig:datafine}. In the left column of each figure we use box plots to compare the median and variability of convergence speed of both methods according to their deformation and displacement threshold. Due to the large number of outliers obtained, especially with PBD-lin, we decided not to include them in the box plots but rather to represent them in stacked bar graphs in the right column of each figure. 

For the coarse database we observe that the median convergence time for PBD-opt is higher compared to PBD-lin for a threshold of $5\%$ and relatively similar or lower for $2.5\%$ and $1\%$. However, since all the box plot pairs overlap each other, we cannot conclude with $95\%$ confidence, that the medians do differ. However, we can claim that the results of PBD-opt are more stable since they have either similar or smaller variances compared to the results of PBD-lin. These differences are further exacerbated when evaluating the fine meshes data-set where the variance of PBD-lin is many times higher compared to the variance in PBD-opt. 

For the outlier analysis we must make a distinction between two types of outliers. The first type are slow convergence (SC) outliers, which surpasses the upper limit of the box plot but did not reach the stopping time. The second type are very slow convergence (VSC) outliers, which reach the stopping time. For PBD-opt, at most $2\%$ of the runs were SC outliers and $0\%$ were VSC outliers. However for PBD-lin, in the coarse meshes dataset, $10\%$ of the runs were SC outliers. In the fine meshes dataset, around $34\%$ of the runs were VSC outliers. This means that PBD-lin has a higher risk of getting stuck in iterations whose convergence time would be way higher compared to their median convergence speed. PBD-opt, on the other hand, provides more stable results with significantly fewer outliers.

\begin{figure}
\centering
  \begin{subfigure}[b]{0.37\textwidth}
    \includegraphics[width=\textwidth]{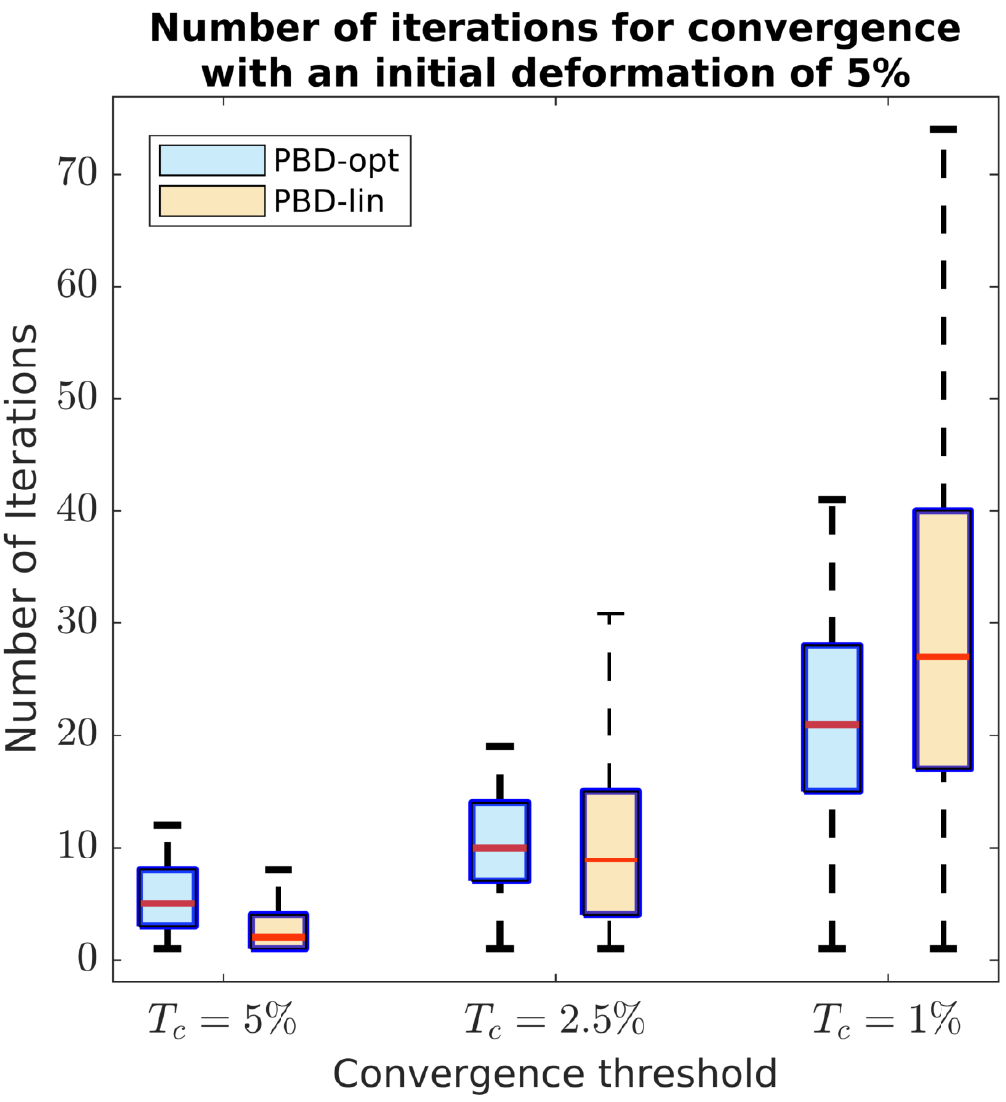}
  \end{subfigure} \quad
  \begin{subfigure}[b]{0.37\textwidth}
    \includegraphics[width=\textwidth]{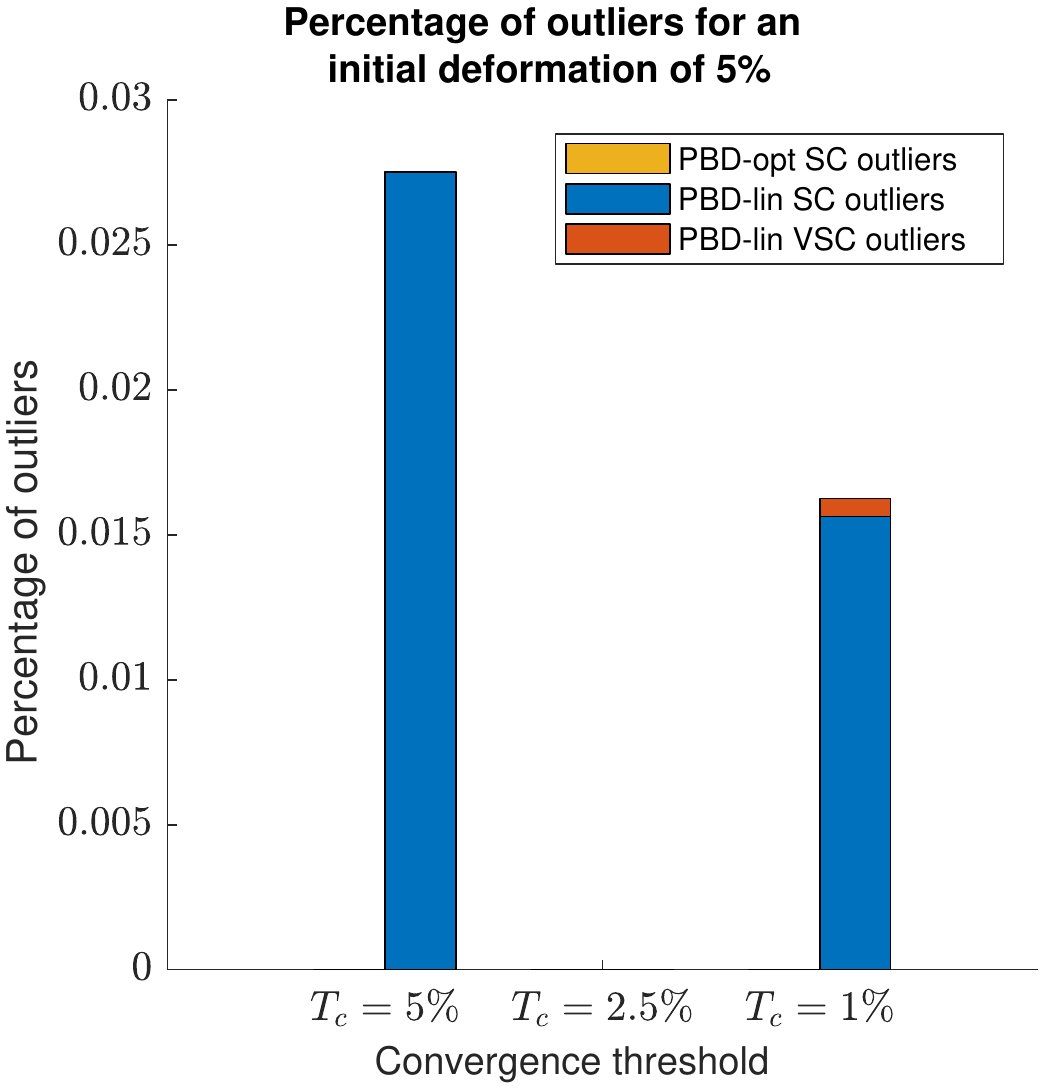}
  \end{subfigure} \\
  \begin{subfigure}[b]{0.37\textwidth}
    \includegraphics[width=\textwidth]{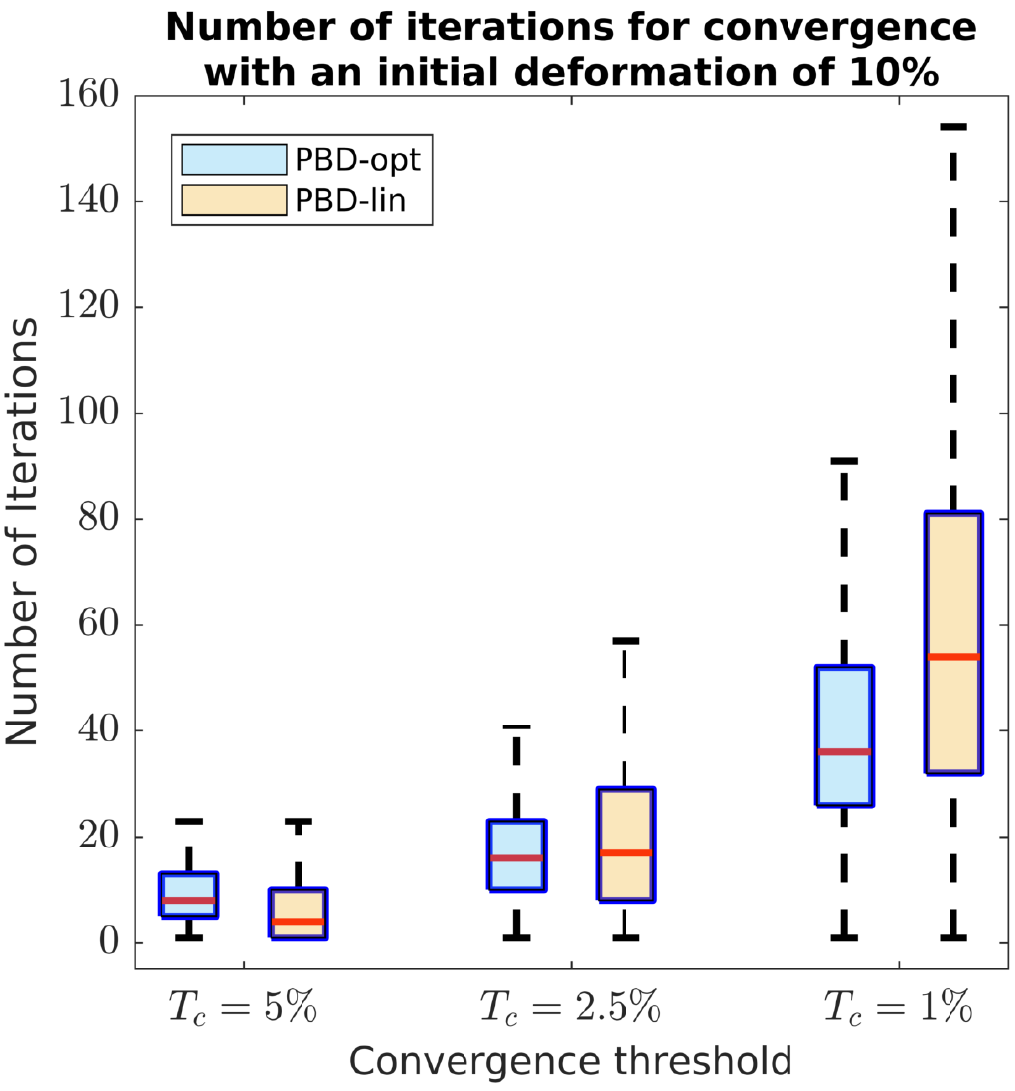}
  \end{subfigure} \quad
  \begin{subfigure}[b]{0.37\textwidth}
    \includegraphics[width=\textwidth]{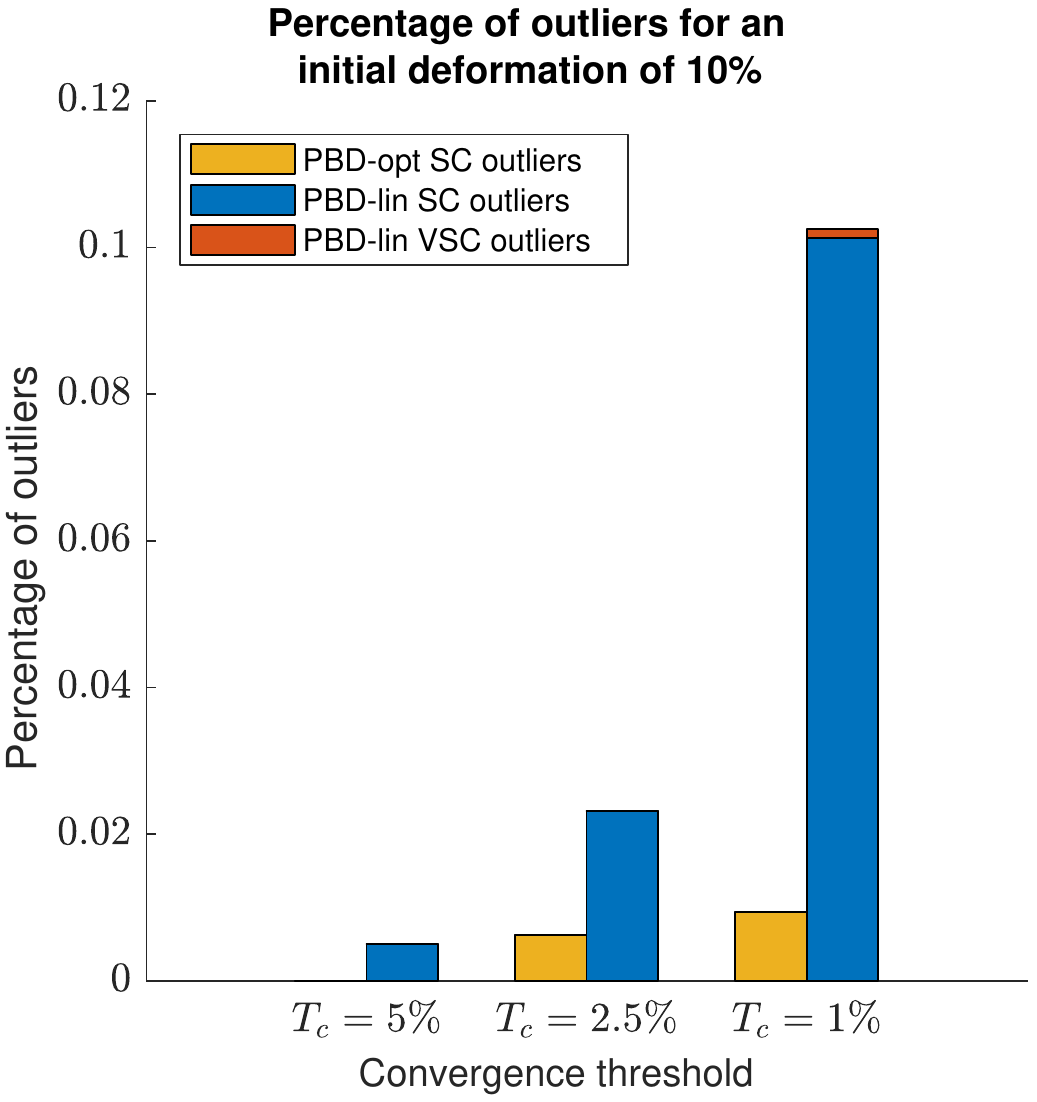}
  \end{subfigure} \\
  \begin{subfigure}[b]{0.37\textwidth}
    \includegraphics[width=\textwidth]{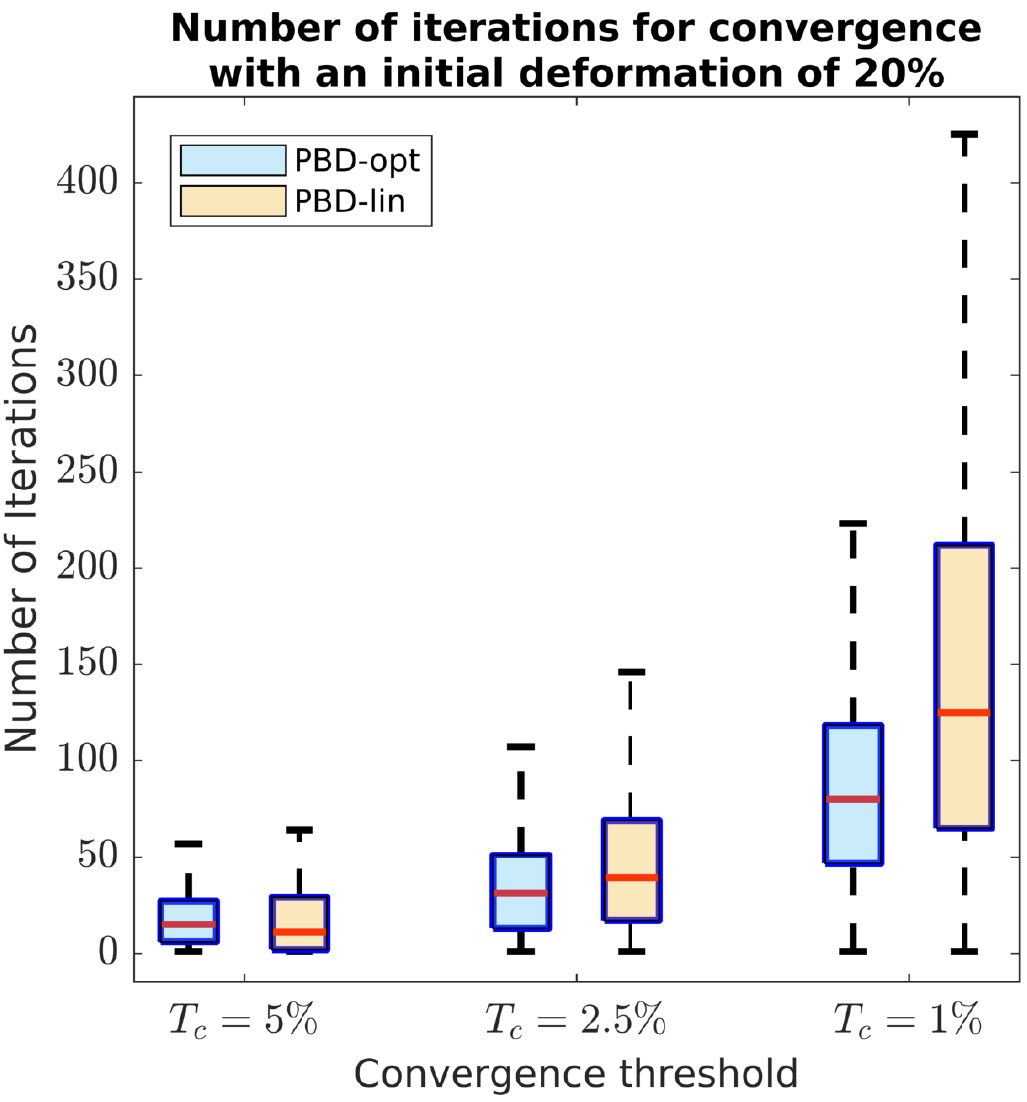}
  \end{subfigure} \quad
  \begin{subfigure}[b]{0.37\textwidth}
    \includegraphics[width=\textwidth]{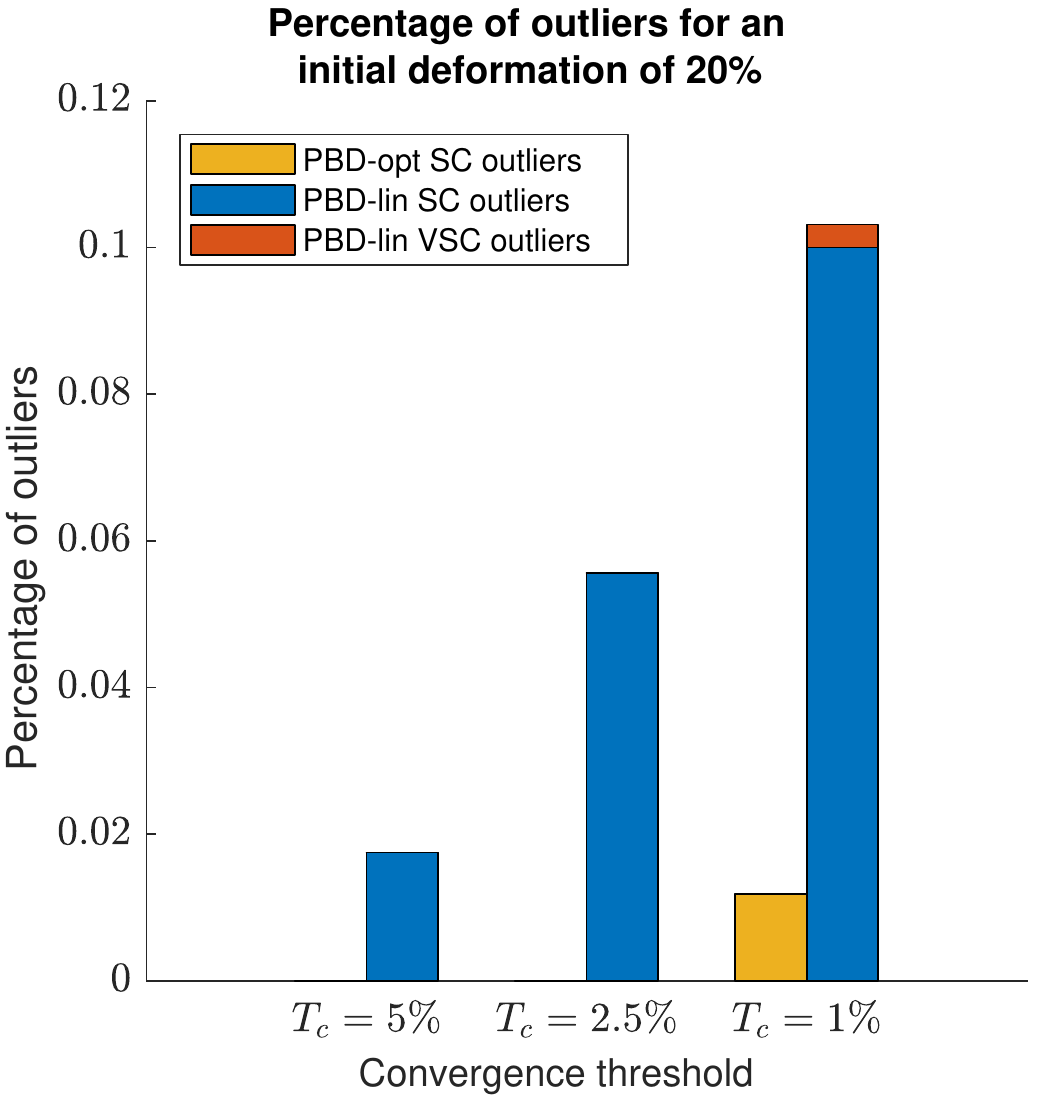}
  \end{subfigure} \\
  \caption{Convergence speed results for coarse 2D meshes. The left column shows the statistics for the number of iterations to reach conversion (omitting outliers). The right column shows the proportion of slow convergence (SC) and very slow convergence (VSC) outliers per method.}
  \label{fig:datacoarse}
\end{figure}

\begin{figure}
\centering
  \begin{subfigure}[b]{0.37\textwidth}
    \includegraphics[width=\textwidth]{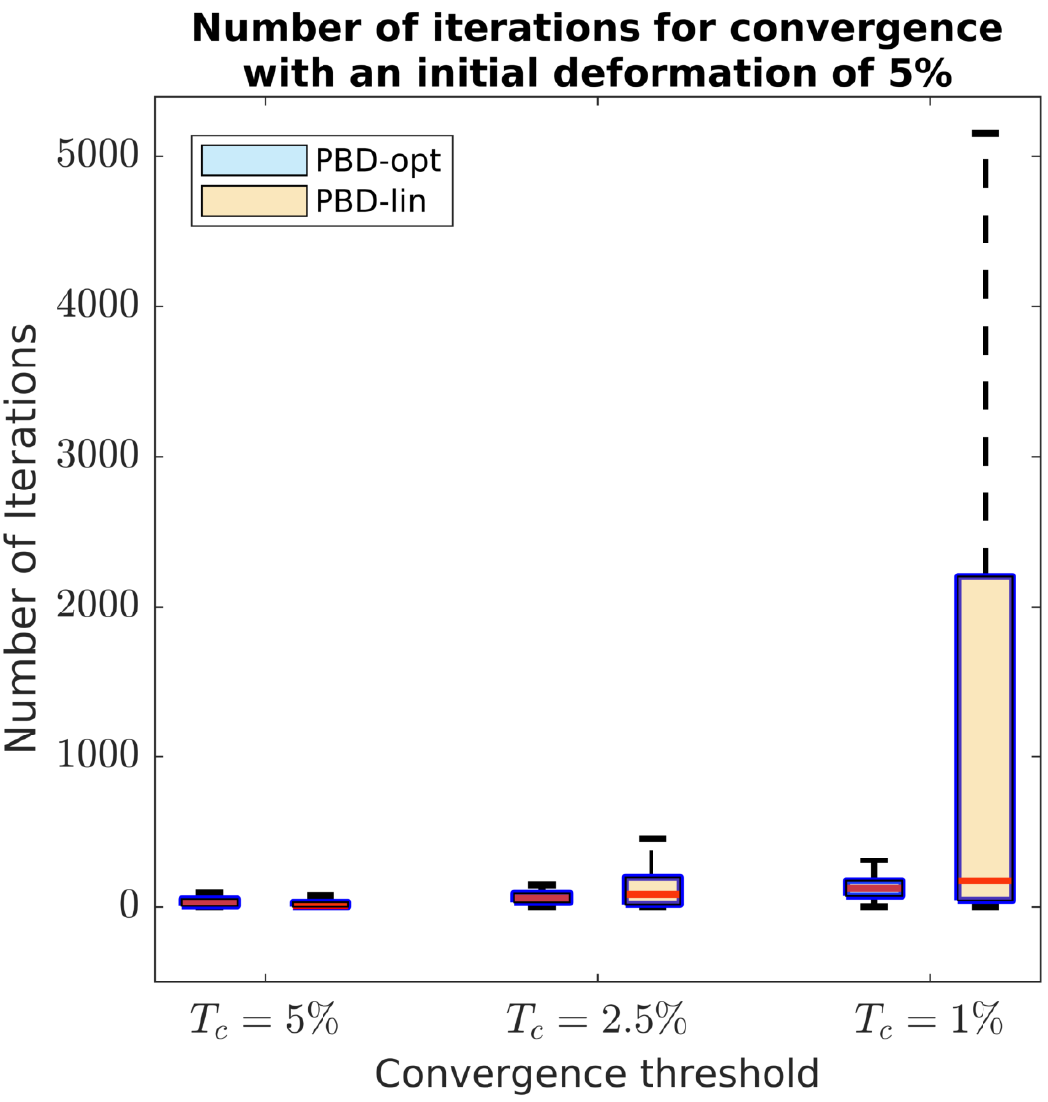}
  \end{subfigure} \quad
  \begin{subfigure}[b]{0.37\textwidth}
    \includegraphics[width=\textwidth]{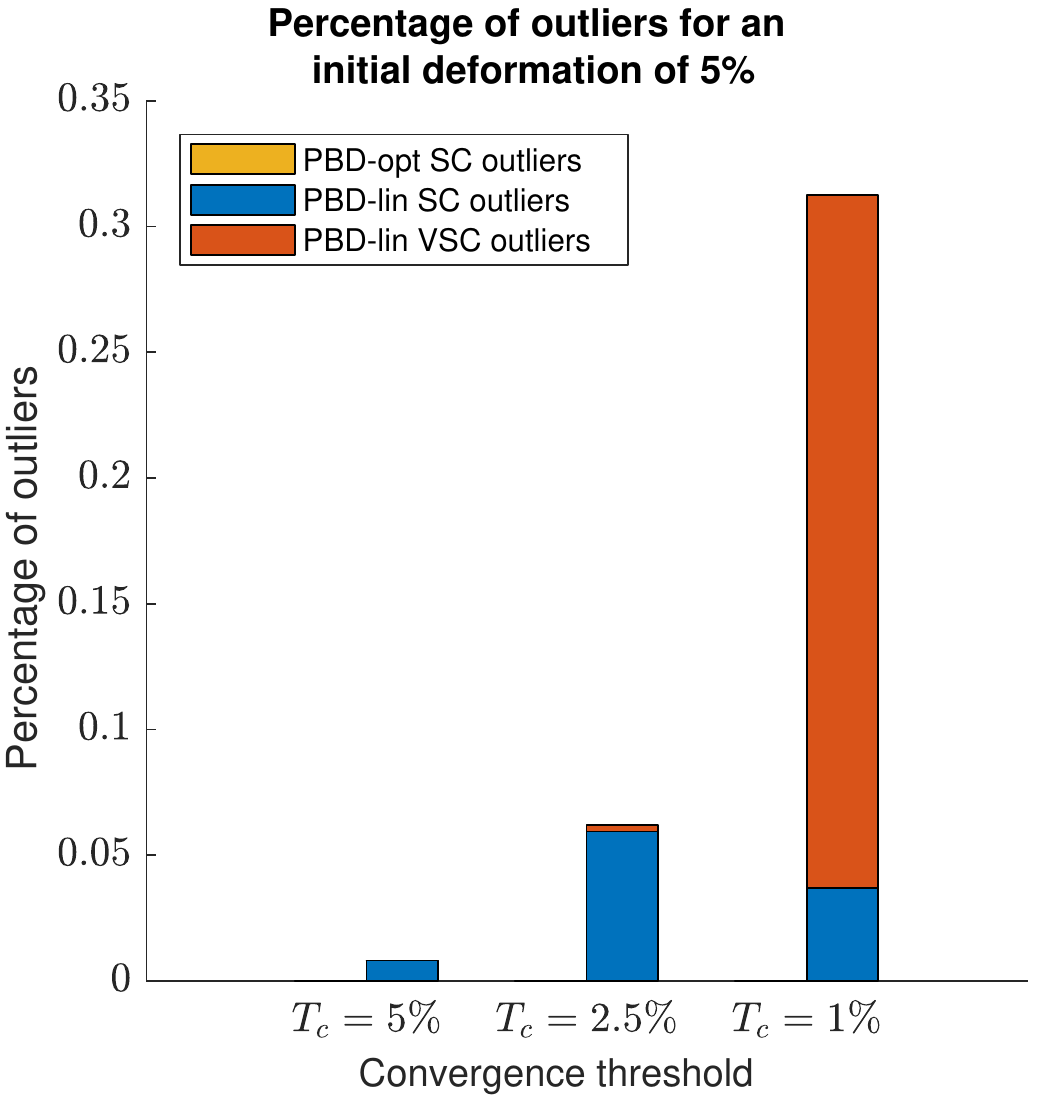}
  \end{subfigure} \\
  \begin{subfigure}[b]{0.37\textwidth}
    \includegraphics[width=\textwidth]{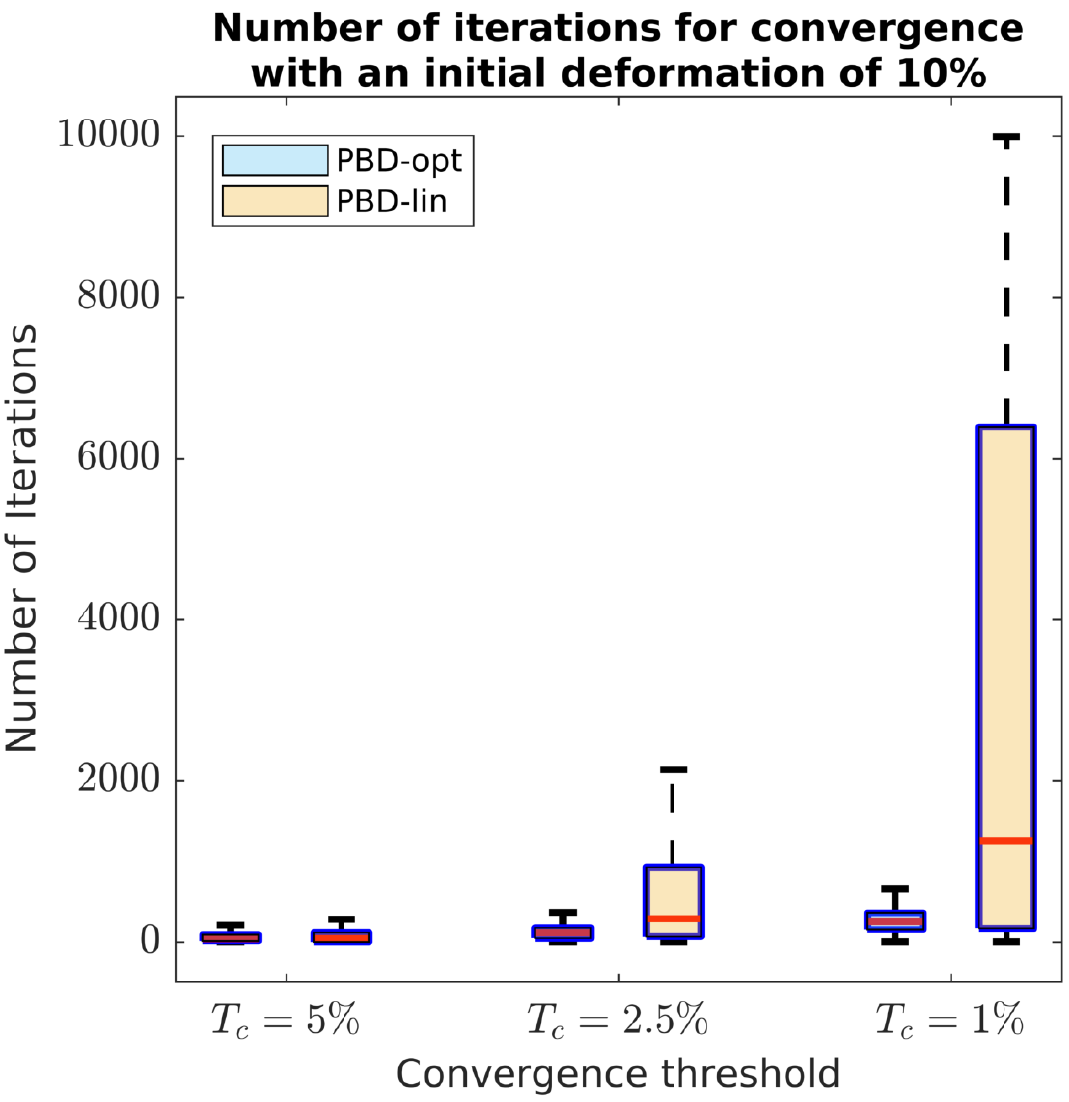}
  \end{subfigure} \quad
  \begin{subfigure}[b]{0.37\textwidth}
    \includegraphics[width=\textwidth]{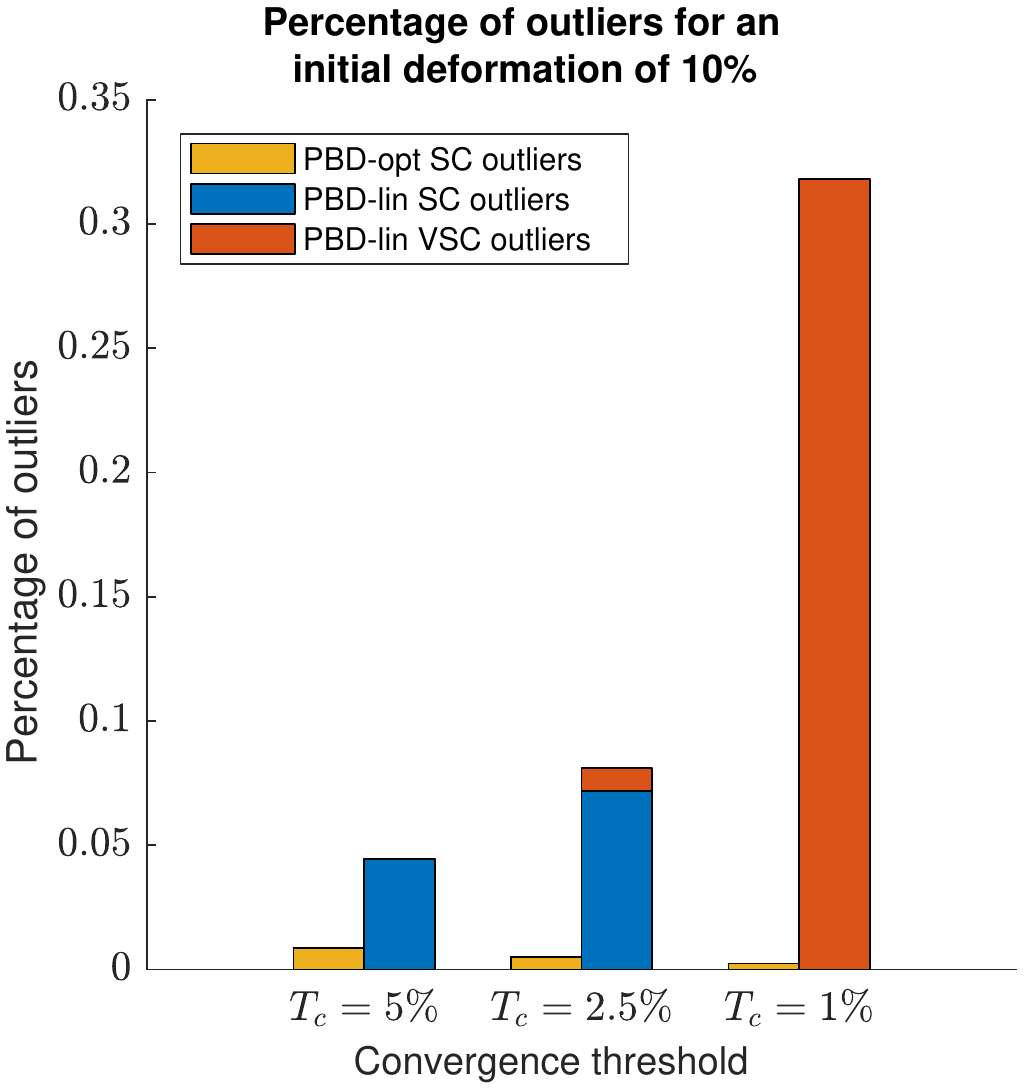}
  \end{subfigure} \\
  \begin{subfigure}[b]{0.37\textwidth}
    \includegraphics[width=\textwidth]{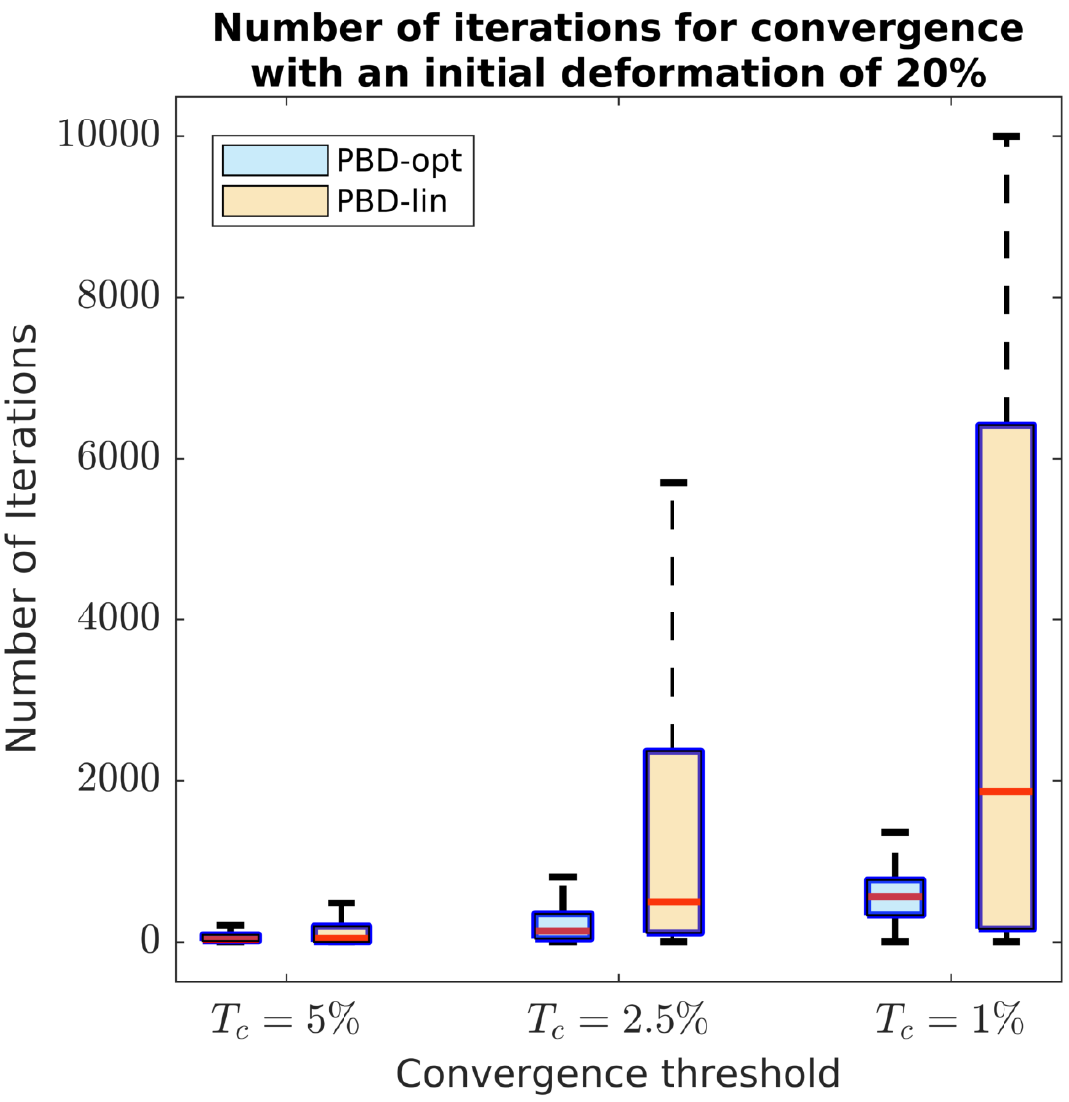}
  \end{subfigure} \quad
  \begin{subfigure}[b]{0.37\textwidth}
    \includegraphics[width=\textwidth]{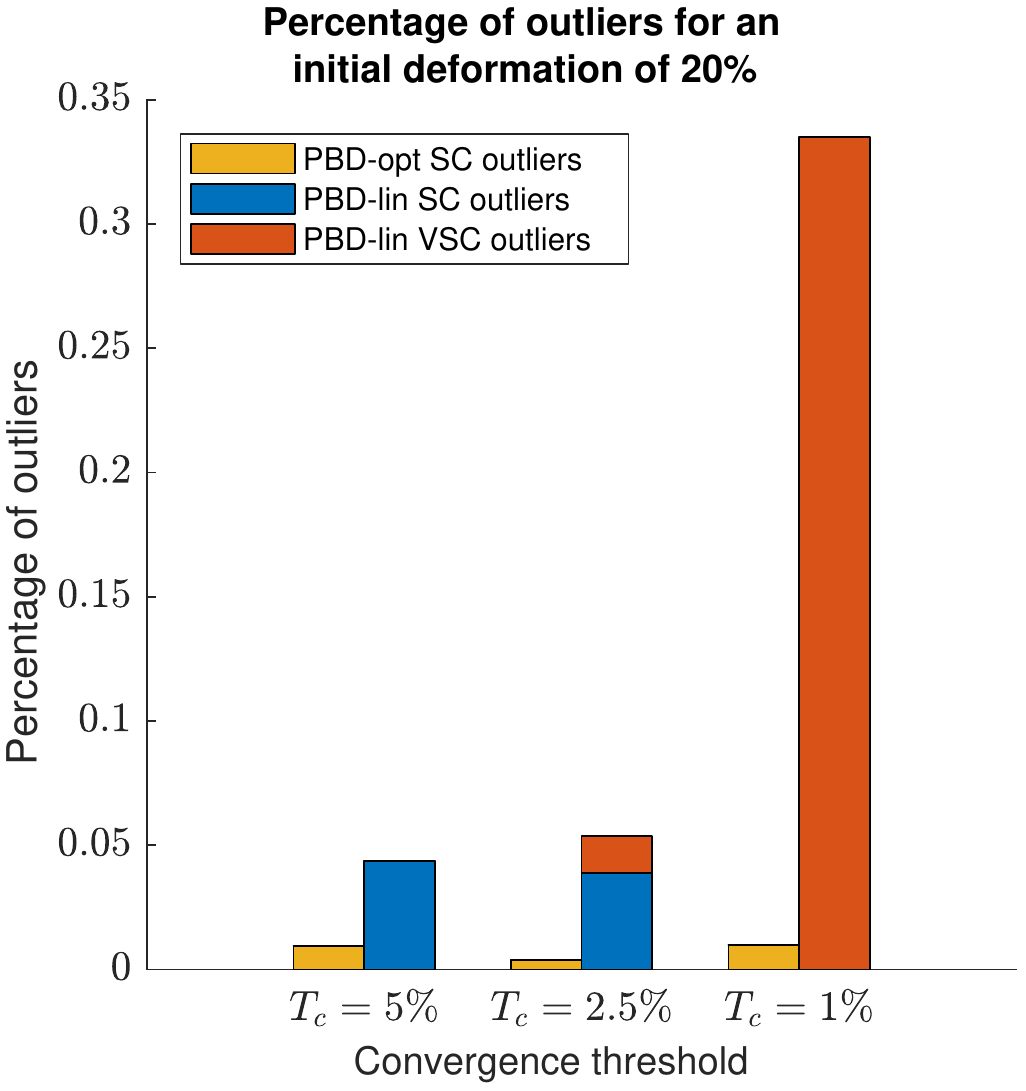}
  \end{subfigure} \\
  \caption{Convergence speed results for fine 2D meshes. The left column shows the statistics for the number of iterations to reach conversion (omitting outliers). The right column shows the proportion of slow convergence (SC) and very slow convergence (VSC) outliers per method.}
  \label{fig:datafine}
\end{figure}

\section{Conclusion}
We have identified two problems related to finding the closest triangle to an input triangle under a prescribed area constraint (the OTPPA problem) and under a prescribed area and orientation constraints (the OTPPAO problem). 
We have given a detailed analysis and a closed-form solution to both of these problems for the first time.
We have then developed a numerically robust algebraic implementation.
We have used it within Point-Based Dynamics, resulting in a 2D triangular mesh editing procedure which has been shown to be faster and more stable than the existing method. 
As future work, we plan to study the equivalent problems in the 3D space, for 3D triangles and tetrahedrons.

\appendix
\section*{Appendices}
\section{Proof of Lemmas}
\label{sec:lempro}
\begin{proof}[Proof of Lemma 1]
We start with the forward implication: $S_1 \Rightarrow A(\tilde{\mathbf{v}})= 0$ and $|\lambda|=\lambda_o$.
In $S_1$, $\tilde{\mathbf{v}}$ represents a single point. This implies $A(\tilde{\mathbf{v}})=0$ and $\sigma^2(\tilde{\mathbf{v}})=0$. Replacing these values in the depressed quartic equation (\ref{eq:depquar}) causes the coefficients $p$ and $r$ to become constants, and coefficient $q$ to vanish (also the orientation constraint vanishes).
The depressed quartic thus transforms into a bi-quadratic:
\begin{equation}
    \lambda^4 - \frac{32}{3}\lambda^2 +\frac{256}{9} = 0,
\end{equation}
whose solutions are:
\begin{equation}
    |\lambda| = \lambda_o. 
\end{equation}
We now turn to the reverse implication: $S_1 \Leftarrow A(\tilde{\mathbf{v}}) = 0$ and $|\lambda|=\lambda_o$. We substitute $|\lambda| = \lambda_o$ in equation (\ref{eq:subconsiar}), giving:
\begin{equation}
    s\sign(A^*(\tilde{\mathbf{v}}))\sign(\lambda)\lambda_o A(\tilde{\mathbf{v}}) - \sigma^2(\tilde{\mathbf{v}})=0.
    \label{eq:contot}
\end{equation}
Since $A(\tilde{\mathbf{v}}) = 0$, then the only solution that satisfies equation (\ref{eq:contot}) for any given value of $s\sign(A^*(\tilde{\mathbf{v}}))\sign(\lambda)$ is with $\sigma^2(\tilde{\mathbf{v}})=0$, which implies that the input triangle is collapsed into a single point, hence to $S_1$. 
\end{proof}

\begin{proof}[Proof of Lemma 2]
We start with the forward implication: $S_2 \Rightarrow A(\tilde{\mathbf{v}}) \neq 0$ and $\lambda=-\lambda_o$.
In $S_2$, $\tilde{\mathbf{v}}$ represents an equilateral triangle and orientation inversion.
This implies $A(\tilde{\mathbf{v}})\neq0$ and $s\sign(A^*(\tilde{\mathbf{v}}))=-1$. 
We perform a similarity transformation to $\tilde{\mathbf{v}}$ by bringing one of its vertices to the origin and another one to the $x$-axis, scaled to normalise their  distance, giving $\tilde{\mathbf{v}}'=[0,0,1,0,1/2,\sign(A^*(\tilde{\mathbf{v}}))\sqrt{3}/2]^\top$ 
where $\sign(A^*(\tilde{\mathbf{v}}))$ determines the orientation of the triangle.
We obtain $A^*(\tilde{\mathbf{v}}')=\sign(A^*(\tilde{\mathbf{v}}))\frac{\sqrt{3}}{4}$ and $\sigma^2(\tilde{\mathbf{v}}')=1$, thus the coefficients of the depressed quartic equation~(\ref{eq:depquar}) become:
\begin{align}
	p &= - \frac{32A_o - 4\sqrt{3}}{3A_o} \\
	q &=  \frac{32}{3A_o} \\
	r &=  \frac{256A_o + 64\sqrt{3}}{9A_o}.
\end{align}
Substituting these coefficients in Cardano's formula we obtain:
\begin{align}
	Q_1 &= -\left(\frac{32A_o + 2\sqrt{3}}{9A_0}\right)^2 \\
	Q_2 &= \left(\frac{32A_o + 2\sqrt{3}}{9A_0}\right)^3,
\end{align}
making $\sqrt{Q_1^3+Q_2^2}=0$ and the real root $\alpha_o$ of Cardano's resolvent cubic to become:
\begin{equation}
    \alpha_o = 2\left(\frac{32A_o + 2\sqrt{3}}{9A_o}\right)+\frac{32A_o-4\sqrt{3}}{9A_o} = \frac{32}{3}.
\end{equation}
We finally use $\alpha_o$ to extract the roots of the depressed quartic:
\begin{equation}
	\lambda = \frac{s_1\sqrt{2}\lambda_o+s_2\sqrt{-\lambda_o\left(\frac{ s_1+1}{A_o}\right)}}{\sqrt{2}},
\end{equation}
where $s_1,s_2 \in \{-1,1\}$. We thus have the following roots:
\begin{equation}
    \lambda \in \left\{ -\lambda_o,-\lambda_o,\lambda_o-i\sqrt{\frac{\lambda_o}{A_o}},\lambda_o+i\sqrt{\frac{\lambda_o}{A_o}}\right\}.
\end{equation}
Considering only the real roots we have $\lambda = -\lambda_o$.

We now turn to the reverse implication: $S_2 \Leftarrow A(\tilde{\mathbf{v}}) \neq 0$ and $\lambda=-\lambda_o$. We substitute $\lambda = -\lambda_o$ in equation (\ref{eq:subconsiar}), giving:
\begin{equation}
    -s\sign(A^*(\tilde{\mathbf{v}}))\lambda_o A(\tilde{\mathbf{v}}) - \sigma^2(\tilde{\mathbf{v}})=0.
    \label{eq:con1}
\end{equation}
Since $A(\tilde{\mathbf{v}})\neq0$ and $\sigma^2(\tilde{\mathbf{v}})>0$ then equation (\ref{eq:con1}) can only be solved when $s\sign(A^*(\tilde{\mathbf{v}}))=-1$. We perform the same similarity transformation used at the beginning of the proof to the unknown input triangle $\tilde{\mathbf{v}}$ giving $\tilde{\mathbf{v}}'=[0,0,1,0,\tilde{x}_c',\tilde{y}_c']^\top$ where one of the vertices remains unknown. We then have $A^*(\tilde{\mathbf{v}}') = \frac{\tilde{y}_c'}{2}$ or \linebreak $\sign(A^*(\tilde{\mathbf{v}}'))A(\tilde{\mathbf{v}}') = \frac{\tilde{y}_c'}{2}$ 
and $\sigma^2(\tilde{\mathbf{v}}') = \frac{2}{3}(\tilde{x}_c'^2 + \tilde{y}_c'^2 - \tilde{x}_c' + 1)$ which we substitute in equation (\ref{eq:con1}) and obtain:
\begin{equation}
    \tilde{x}_c'^2+\tilde{y}_c'^2 -\tilde{x}_c' +\sqrt{3}s\tilde{y}_c' + 1 = 0,
\end{equation}
which we rewrite as:
\begin{equation}
    \left(\tilde{x}_c'-\frac{1}{2}\right)^2 + \left(\tilde{y}_c' + \frac{\sqrt{3}}{2}s\right)^2 = 0.
\end{equation}
This is the equation of a single point, making $\tilde{\mathbf{v}}'$ an equilateral triangle \linebreak 
$\tilde{\mathbf{v}}'=[0,0,1,0,1/2,-\sqrt{3}s/2]^\top$, where $s$ determines the orientation of the triangle. Since $\tilde{\mathbf{v}}'$ was a similarity transformation of $\tilde{\mathbf{v}}$, then $\tilde{\mathbf{v}}$ is also an equilateral triangle when $s\sign(A^*(\tilde{\mathbf{v}}))=-1$, which corresponds to $S_2$.
\end{proof}

\begin{proof}[Proof of Lemma 3]
In $S_3$, $\tilde{\mathbf{v}}$ represents an equilateral triangle with $A(\tilde{\mathbf{v}})/4\geq A_o$ and no orientation inversion.
This implies $A^*(\tilde{\mathbf{v}})\neq0$ and $s\sign(A^*(\tilde{\mathbf{v}}))=1$. 
We perform the same similarity transformation to $\tilde{\mathbf{v}}$ as in lemma 2,
thus the coefficients of the depressed quartic equation~(\ref{eq:depquar}) become:
\begin{align}
	p &= - \frac{32A_o + 4\sqrt{3}}{3A_o} \\
	q &=  \frac{32}{3A_o} \\
	r &=  \frac{256A_o - 64\sqrt{3}}{9A_o}.
\end{align}
Substituting these coefficients in Cardano's formula we obtain:
\begin{align}
	Q_1 &= -\left(\frac{32A_o - 4}{9A_0}\right)^2 \\
	Q_2 &= -\left(\frac{32A_o - 4}{9A_0}\right)^3,
\end{align}
making $\sqrt{Q_1^3+Q_2^2}=0$.
After some factoring we obtain the real root $\alpha_o$ of Cardano's resolvent cubic as:
\begin{equation}
    \alpha_o = -2\sqrt[3]{\left(
    \frac{32A_o - 8sA^*(\tilde{\mathbf{v}}')}{3A_o}\right)^3}+\frac{32A_o+8 s A^*(\tilde{\mathbf{v}}')}{3A_o}.
\end{equation}
Since $s\sign(A^*(\tilde{\mathbf{v}}))=1$ and $A^*(\tilde{\mathbf{v}}') = \sign(A^*(\tilde{\mathbf{v}}))A(\tilde{\mathbf{v}}')$, then $\alpha_o \in \mathbb{R}$ only when $A(\tilde{\mathbf{v}}')/8 \geq A_o$. 
Under this condition, we obtain $\alpha_o = 32$.
Substituting in equation (\ref{eq:roots}) we obtain:
\begin{equation}
    \lambda = \frac{s_1\sqrt{2}\lambda_o+s_2\sqrt{\lambda_o\left(\frac{ s_1-1}{A_o}\right)}}{\sqrt{2}}
\end{equation}
where $s_1,s_2 \in \{-1,1\}$. We thus have the following roots:
\begin{equation}
    \lambda \in \left\{ -\lambda_o-\sqrt{\frac{\lambda_o}{A_o}},-\lambda_o+\sqrt{\frac{\lambda_o}{A_o}},\lambda_o,\lambda_o\right\}.
    \label{eq:eqcase1}
\end{equation}
Thus, we have two roots where $|\lambda|\neq\lambda_o$ (which correspond to solutions for Case I) and there exists at least one solution $\lambda = \lambda_o$ for $S_3$ (which correspond to the solution of Case II).
\end{proof}

\begin{proof}[Proof of Lemma 4]
We substitute $\lambda = \lambda_o$ in equation (\ref{eq:subconsiar}), giving:
\begin{equation}
    s\sign(A^*(\tilde{\mathbf{v}}))\lambda_o A(\tilde{\mathbf{v}}) - \sigma^2(\tilde{\mathbf{v}})=0.
    \label{eq:con2}
\end{equation}
Since $A(\tilde{\mathbf{v}})\neq0$ then equation (\ref{eq:con2}) can only be solved when $s\sign(A^*(\tilde{\mathbf{v}}))=1$. We perform the same similarity transformation to $\tilde{\mathbf{v}}$ as in lemma 2,
substitute $\sign(A^*(\tilde{\mathbf{v}}'))A(\tilde{\mathbf{v}}')$ and $\sigma^2(\tilde{\mathbf{v}}')$ in equation (\ref{eq:con2}) and obtain:
\begin{equation}
    \tilde{x}_c'^2+\tilde{y}_c'^2 -\tilde{x}_c' +\sqrt{3}s\tilde{y}_c' + 1 = 0,
\end{equation}
which we rewrite as:
\begin{equation}
    \left(\tilde{x}_c'-\frac{1}{2}\right)^2 + \left(\tilde{y}_c' + \frac{\sqrt{3}}{2}s\right)^2 = 0.
\end{equation}
This is the equation of a single point making $\tilde{\mathbf{v}}'$ an equilateral triangle \linebreak $\tilde{\mathbf{v}}'= \left[0,0,1,0,1/2,-\sqrt{3}s/2\right]^\top$ where $s$ determines the orientation of the triangle. Since $\tilde{\mathbf{v}}'$ was a similarity transformation of $\tilde{\mathbf{v}}$, then $\tilde{\mathbf{v}}$ is also an equilateral triangle when $s\sign(A^*(\tilde{\mathbf{v}}))=1$ for any value $A(\tilde{\mathbf{v}})$ (including $A(\tilde{\mathbf{v}}')/4 \geq A_o$) which corresponds to $S_3$.
\end{proof}

\begin{proof}[Proof of Lemma 5]
We perform the same similarity transformation to $\tilde{\mathbf{v}}$ as in lemma 2 and obtain $A^*(\tilde{\mathbf{v}}')=\sign(A^*(\tilde{\mathbf{v}}))\frac{\sqrt{3}}{4}=\frac{\sign(A^*(\tilde{\mathbf{v}}))}{\lambda_o}$. 
We relax the condition of $S_3$ where $\frac{A(\tilde{\mathbf{v}}')}{4} \geq A_o$ by reformulating $A_o = \frac{1}{z\lambda_o}$ where $z$ is a scaling factor $z>0 \in \mathbb{R}$. We substitute this in equation~(\ref{eq:eqcase1}) and extract the roots :
\begin{equation}
    \lambda = \left[-\lambda_o-\lambda_o\sqrt{z},-\lambda_o+\lambda_o\sqrt{z},\lambda_o,\lambda_o\right]^\top.
\end{equation}
We start with the solution given by Case I where $|\lambda|\neq\lambda_o$. We compact both values as $\lambda=-\lambda_o+s_3\lambda_o\sqrt{z}$ where $s_3 \in \{-1,1\}$. We substitute this in equation (\ref{eq:mateqadj2}) and obtain:
\begin{equation}
    \mathbf{v}' = \begin{bmatrix}
    \frac{\sqrt{z}-s_3}{2\sqrt{z}} & \frac{\sqrt{3}(\sqrt{z}-s_3 )}{6\sqrt{z}} &  \frac{\sqrt{z}+s_3}{2\sqrt{z}} & \frac{\sqrt{3}(\sqrt{z}-s_3)}{6\sqrt{z}} & \frac{1}{2} & \frac{\sqrt{3}(\sqrt{z}+2s_3)}{6\sqrt{z}}
    \end{bmatrix}^\top.
\end{equation}
We calculate the cost of the solution of Case I by substituting this and $\tilde{\mathbf{v}}'$ in equation (\ref{eq:cost}) and obtain:
\begin{equation}
    \mathscr{C}_1(\mathbf{v}') = \frac{(\sqrt{z}-s_3)^2}{z}
\end{equation}
Now we turn to the solution given by Case II where $\lambda = \lambda_o$. We calculate the basis vector $\mathbf{v}'_c$, the particular solution $\mathbf{v}'_p$ and substitute them in equation~(\ref{eq:case2}) and obtain:
\begin{equation}
    \mathbf{v}' = \begin{bmatrix}
    \frac{1}{4} - \frac{\sqrt{3}\sqrt{z-4}}{12\sqrt{z}} \\ 
    \frac{\sqrt{3}}{12} - \frac{\sqrt{z-4}}{4\sqrt{z}} \\  
    \frac{3}{4} - \frac{\sqrt{3}\sqrt{z-4}}{12\sqrt{z}} \\ 
    \frac{\sqrt{3}}{12} + \frac{\sqrt{z-4}}{4\sqrt{z}} \\
    \frac{1}{2} + \frac{\sqrt{3}\sqrt{z-4}}{6\sqrt{z}} \\
    \frac{\sqrt{3}}{3}
    \end{bmatrix}.
\end{equation}
We calculate the cost of the solution of Case II by substituting this and $\tilde{\mathbf{v}}'$ in equation (\ref{eq:cost}) and obtain:
\begin{equation}
    \mathscr{C}_2(\mathbf{v}') = \frac{1}{2} - \frac{1}{z}
\end{equation}
We compare the cost of both solutions $\mathscr{C}_1(\mathbf{v}')\geq\mathscr{C}_2(\mathbf{v}')$ and obtain:
\begin{equation}
     \frac{(\sqrt{z}-s_3)^2}{z} \geq \frac{1}{2} - \frac{1}{z}
\end{equation}
After some minor manipulations we obtain:
\begin{equation}
    \frac{z-4s_3\sqrt{z}+6}{2z} \geq 0
\end{equation}
We substitute $\sqrt{z}=a$ where $a>0 \in \mathbb{R}$ and obtain the quadratic expression:
\begin{equation}
    \frac{a^2-4s_3a+6}{2a^2} \geq 0
\end{equation}
which represents an upward opening parabola which is always positive for any value of $a$ and thus $z$. This implies that $\mathscr{C}_1(\mathbf{v}')\geq\mathscr{C}_2(\mathbf{v}')$ for any value $z$ including $z>4$. Since $\tilde{\mathbf{v}}'$ was a similarity transformation of $\tilde{\mathbf{v}}$, then $\mathscr{C}_1(\mathbf{v})\geq\mathscr{C}_2(\mathbf{v})$.
This means that in $S_3$ the solution provided by case II has the lowest cost, thus is the optimal solution.
\end{proof}

\section{Optimal Triangle Projection with Prescribed Area}
\label{sec:OTPPA}
In OTPPA, only the prescribed area needs to be preserved. This means that a solution may freely choose the orientation which minimises the cost, as long as the area constraint is satisfied. Consequently, we expect that OTPPA has a larger set of local extrema than OTPPAO and hence more candidate algebraic solutions. Specifically, OTPPA is stated as:
\begin{equation}
    \min_{\mathbf{v} \in \mathbb{R}^6} \mathscr{C}(\mathbf{v}) \quad \text{s.t.} \quad f(\mathbf{v})=0 .
    \label{eq:mini3}
\end{equation}
The area constraint in OTPPA is technically more complex to handle than the signed area constraint in OTPPAO, because it involves an absolute value. 
Fortunately, a solution may be obtained by exploiting a reformulation in terms of two rounds of OTPPAO.
Similarly to OTPPAO, we start by replacing $A(\mathbf{v})$ by $A^*(\mathbf{v})$ in the area constraint $f(\mathbf{v})$, expressing $f$ as the disjunction of  two cases:
\begin{align}
    f(\mathbf{v}) =  0  &\quad\Leftrightarrow\quad \left(f^+(\mathbf{v})=0\right) \lor \left(f^-(\mathbf{v})=0\right) \quad \mbox{ with} \\
    f^+(\mathbf{v}) &= A^*(\mathbf{v})-A_o  \\
    f^-(\mathbf{v}) &=-A^*(\mathbf{v})-A_o.  
\end{align}
We seek a solution which satisfies either $f^+$ or $f^-$.
This can be achieved by solving OTPPAO for $s=1$ and $s=-1$, and simply selecting the minimal cost solution a posteriori. We present the numerically robust procedure in Algorithm \ref{alg:solOTPPA}.

\begin{algorithm}
    \caption{Optimal Triangle Projection with Prescribed Area}
    \label{alg:solOTPPA}
    \begin{algorithmic}[1]
    \Require $\tilde{\mathbf{v}}$ - input vertices, $A_o$ - prescribed area, $s$ - prescribed orientation, $E$ - area error tolerance
    \Ensure $\mathbf{v}_o$ - optimal triangle, $\mathcal{v}_1$ - Case I triangle set, $\mathcal{v}_2$ - Case II triangle set, $\mathcal{v}_c, \mathcal{v}_t$ - Case II basis and translations sets
    \Function{OTPPA}{$\tilde{\mathbf{v}},A_o,s, E = 10^{-3}$}
    \State $\mathcal{v}_1^- \gets $ \Call{SolveCase1}{$\tilde{\mathbf{v}},A_o,-1,E$} \Comment{Compute Case I solutions}
    \State $\mathcal{v}_1^+ \gets $ \Call{SolveCase1}{$\tilde{\mathbf{v}},A_o,1,E$} 
    \State $(\mathbf{v}_2^-, \mathbf{v}_c^-, \mathbf{v}_t^-) \gets$ \Call{SolveCase2}{$\tilde{\mathbf{v}},A_o,-1,E$}  \Comment{Compute Case II solutions}
    \State $(\mathbf{v}_2^+, \mathbf{v}_c^+, \mathbf{v}_t^+) \gets$ \Call{SolveCase2}{$\tilde{\mathbf{v}},A_o,1,E$}  
    
    \State $\mathcal{v}_1 \gets \mathcal{v}_1^- \cup \mathcal{v}_1^+$ \Comment{Add the vertices to the solution set}
    \State $\mathcal{v}_2 \gets \{ \mathbf{v}_2^- \} \cup \{ \mathbf{v}_2^+ \}$ 
    \State $\mathcal{v}_c \gets \{ \mathbf{v}_c^- \} \cup \{ \mathbf{v}_c^+ \}$
    \State $\mathcal{v}_t \gets \{ \mathbf{v}_t^- \} \cup \{ \mathbf{v}_t^+ \}$
    \State $\mathbf{v}_o\gets$ \Call{FindTriangleOfMinimalCost}{$\mathcal{v}_1 \cup\mathcal{v}_2$} \Comment{Select the optimal solution}
    \State \textbf{return} $\mathbf{v}_o, \mathcal{v}_1, \mathcal{v}_2, \mathcal{v}_c, \mathcal{v}_t$    
    \EndFunction
    \end{algorithmic}
\end{algorithm}
\section{Ferrari's Method for the Depressed Quartic Roots}
\label{sec:ferrari}
We give Ferrari's method for solving the depressed quartic equation $\lambda^4 - p\lambda^2 + q\lambda + r = 0$ (\cite{tignol_galois_2001}). 
\begin{lem}
If $\lambda^4 - p\lambda^2 + q\lambda + r = 0$ and $q\neq 0$ then there exists an $\alpha_o \neq 0$ such that \linebreak $\lambda = \frac{s_1\sqrt{2\alpha_o}+s_2\sqrt{-\left(2p+2\alpha_o+s_1\frac{2q}{\sqrt{2\alpha_o}}\right)}}{2}$, where $s_1,s_2\in\{-1,1\}$.
\label{lem:quarqnot0}
\end{lem}
\begin{proof}[Proof of Lemma 6]
First, one rewrites the depressed quartic equation as:
\begin{equation}
	\lambda^2 + \frac{p}{2} = -q\lambda -r.
\end{equation}
One adds $\frac{p^2}{4}$ to both sides to complete the square on the left-hand side as:
\begin{equation}
	\left(\lambda^2 + \frac{p}{2} \right)^2 = -q\lambda -r + \frac{p^2}{4}.
\end{equation}
One then introduces a variable factor $\alpha$ into the left-hand side by adding $2\lambda^2\alpha + p\alpha + \alpha^2$ to both sides. Grouping the coefficients by powers of $\lambda$ in the right-hand side gives:
\begin{equation}
	\left(\lambda^2 + \frac{p}{2} + \alpha\right)^2 = 2\alpha \lambda^2-q\lambda+\left(\alpha^2+\alpha p+\frac{p^2}{4}-r\right).
	 \label{eq:ferrari}
\end{equation}
A quadratic expression $ax^2 + bx + c$ is considered a perfect square when its discriminant $b^2 - 4ac = 0$ vanishes, allowing one to rewrite it as  $(\sqrt{a}x + \sqrt{c})^2$. We use this idea to  choose a value for $\alpha$ such that the bracketed expression in the right-hand side of equation~(\ref{eq:ferrari}), which is a quadratic in $\lambda$, becomes a perfect square.
Specifically, vanishing the discriminant gives:
\begin{equation}
	q^2-8 \alpha \left(\alpha^2+\alpha p+\frac{p^2}{4}-r\right) = 0.
\end{equation}
Upon expanding, it forms a cubic equation in $\alpha$, called the resolvent cubic of the quartic equation:
\begin{equation}
	8\alpha^3 + 8p\alpha^2+(2p^2-8r)\alpha-q^2 = 0.
\end{equation}
This equation implies $\alpha\neq 0$. Indeed, $\alpha = 0$ would imply $q=0$, contradicting our hypothesis $q\neq 0$. 
A real root $\alpha_o\neq 0$ is obtained from Cardano's formula, given in section~\ref{sec:cardan}. Substituting in equation~(\ref{eq:ferrari}), we obtain:
\begin{equation}
	\left(\lambda^2 + \frac{p}{2} + \alpha_o\right)^2 = \left(\lambda\sqrt{2\alpha_o}-\frac{q}{2\sqrt{2\alpha_o}}\right)^2.
	\label{eq:ferrari2}
\end{equation}
This equation is of the form $M^2 = N^2$, which can be rearranged as $M^2 - N^2 = 0$ or $(M+N)(M-N)=0$:
\begin{equation}
	\left(\lambda^2 + \frac{p}{2} + \alpha_o + \lambda\sqrt{2\alpha_o}-\frac{q}{2\sqrt{2\alpha_o}}\right)\left(\lambda^2 + \frac{p}{2} + \alpha_o - \lambda\sqrt{2\alpha_o}+\frac{q}{2\sqrt{2\alpha_o}}\right) = 0.
\end{equation}
This is easily solved by applying the quadratic formula to each factor, leading to:
\begin{equation}
	\lambda = \frac{s_1\sqrt{\alpha_o}+s_2\sqrt{-\left(p+\alpha_o+s_1\frac{q}{\sqrt{2\alpha_o}}\right)}}{\sqrt{2}},
\end{equation}
where $s_1,s_2 \in \{-1,1\}$. 
\end{proof}

\section{Cardano's Method for the Cubic Roots}
\label{sec:cardan}
We consider the cubic equation:
\begin{equation*}
	ax^3+bx^2+cx+d=0  \quad\text{with} \quad a\neq0.
\end{equation*}
Its solutions are:
\begin{align*}
	x_1 &= S_1 + S_2 -\frac{b}{3a}  \\
	x_2 &= -\frac{S_1 + S_2}{2} -\frac{b}{3a} +\frac{i\sqrt{3}}{2}(S_1-S_2)\\
	x_3 &= -\frac{S_1 + S_2}{2} -\frac{b}{3a} -\frac{i\sqrt{3}}{2}(S_1-S_2),
\end{align*}
where:
\begin{align*}
	S_1 &= \sqrt[3]{Q_2+\sqrt{Q_1^3+Q_2^2}} \\
	S_2 &= \sqrt[3]{Q_2-\sqrt{Q_1^3+Q_2^2}} \\
	Q_1 &= \frac{3ac-b^2}{9a^2} \\
	Q_2 &= \frac{9abc-27a^2d-2b^3}{54a^3},
\end{align*}
and $D = Q_1^3+Q_2^2$ is the discriminant of the equation.
For $a,b,c,d \in \mathbb{R}$, three cases can occur:
\begin{align*}
	(1): &\qquad \text{if}\,D>0, \text{ one root is real and two are complex conjugates} \\
	(2): &\qquad \text{if} \,D=0, \text{ all roots are real, and at least two are equal} \\
	(3): &\qquad \text{if} \,D<0, \text{ all roots are real and unequal.}
\end{align*}
\section{Formulation for Restricted Cases}
\label{sec:specas}
Our original formulation assumes that the three triangle vertices are free to move. However, there exist cases when one or two of the vertices are fixed. Typically, this occurs for triangles lying on the domain boundary in mesh editing. We here adapt the proposed optimal projection formulation to these cases. 

\subsection{One Fixed Vertex}
\subsubsection{A Two Case Formulation}
We assume that $v_c$ is fixed. Thus, we have $\mathbf{v} = [\mathbf{u}, \tilde{x}_c,\tilde{y}_c]$ and $\tilde{\mathbf{v}} = [\tilde{\mathbf{u}}, \tilde{x}_c,\tilde{y}_c]$, where the moving vertices are represented by $\mathbf{u} = [x_a,y_a,x_b,y_b] \in \mathbb{R}^4$ and the corresponding input vertices by 
$\tilde{\mathbf{u}}= [\tilde{x}_a, \tilde{y}_a,\tilde{x}_b, \tilde{y}_b]\in \mathbb{R}^4$.
We take $\frac{\partial \mathscr{L}}{\partial \mathbf{u}}=0$ which is formed by the first four equalities of equation~(\ref{eq:lagmult}), which we rewrite in matrix form $X\mathbf{u} = \mathbf{b}$ as:
\begin{equation}
	\begin{bmatrix}
	4 		  & 	    0 & 	    0 & s\lambda \\
	0		  &   		4 & -s\lambda &		   0 \\
	0 		  & -s\lambda & 	    4 &		   0 \\
	s\lambda  & 		0 &         0 &        4 \\
	\end{bmatrix} 
	\begin{bmatrix}
	x_a \\ y_a \\ x_b \\ y_b 
	\end{bmatrix}
	=
	4
	\begin{bmatrix}
	\tilde{x}_a + s\frac{\lambda}{4}(\tilde{y}_c) \\ 
	\tilde{y}_a - s\frac{\lambda}{4}(\tilde{x}_c) \\
	\tilde{x}_b - s\frac{\lambda}{4}(\tilde{y}_c) \\
	\tilde{y}_b + s\frac{\lambda}{4}(\tilde{x}_c) 
	\end{bmatrix}.
	\label{eq:mat_lam2}
\end{equation}
We check the invertibility of $X$ from its determinant:
\begin{equation}
	\det(X) = {(\lambda^2-16)}^2.
\end{equation}
We thus have:
\begin{equation}
	\det(X)=0 \quad \Leftrightarrow \quad |\lambda| = 4.
	\label{eq:lambdanull2}
\end{equation}
We will see that the particular case of $\det(X)=0$ may occur in practice. We thus solve system~(\ref{eq:mat_lam2}) with two cases. In Case I, which is the most general, we have $|\lambda| \neq 4$. In Case II, we have $|\lambda| = 4$. Similarly to the general case in section~\ref{sec:twocas}, the following lemmas establish the relationship between the Lagrange multiplier and the linear deficiency of the input vertices. 

\begin{prop}
\label{prop:gen2}
Most settings (input vertices $\tilde{\mathbf{v}}$, prescribed area $A_o$ and orientation $s$) correspond to $F_o$ and fall in Case I. Exceptions handled with Case II are:
\begin{itemize}
    \item $F_1$: $\tilde{\mathbf{v}}$ is a single point
    \item $F_2$: $\tilde{\mathbf{v}}$ is a right isosceles triangle and
            $s\sign(A^*(\tilde{\mathbf{v}}))=-1$
    \item $F_3$: $\tilde{\mathbf{v}}$ is a right isosceles triangle, $A(\tilde{\mathbf{v}})/4\geq A_o$ and $s\sign(A^*(\tilde{\mathbf{v}}))=1$
\end{itemize}
\label{prop:main2}
\end{prop}
\noindent The proof of proposition~\ref{prop:main2} is  based on the following five lemmas.
\begin{lem}
$F_1 \iff A(\tilde{\mathbf{v}})=0$ and $|\lambda|=4$.
\end{lem}
\begin{lem}
$F_2 \iff A(\tilde{\mathbf{v}}) \neq 0$ and $\lambda=-4$.
\end{lem}
\begin{lem}
$F_3 \Rightarrow A(\tilde{\mathbf{v}}) \neq 0$ and $\lambda \in \bigg\{ 4, -4+2\sqrt{\frac{2}{A_o}} , -4-2\sqrt{\frac{2}{A_o}}\bigg\}$.
\end{lem}
\begin{lem}
$F_3 \Leftarrow A(\tilde{\mathbf{v}}) \neq 0$ and $\lambda = 4$.
\end{lem}
\begin{lem}
$\lambda=4$ leads to the optimal solution for $F_3$.
\end{lem}
\noindent The proofs of these lemmas are given in Appendix~\ref{sec:lemproofspec}. It is important to clarify for proposition 2, that in the right isosceles triangle, the fixed vertex corresponds to the one opposite to the triangle's hypotenuse. 
\begin{proof}[Proof of proposition 2]
We recall that Case I occurs for $|\lambda| \neq 4$ and Case II for $|\lambda| = 4$. Lemmas 7, 8 and 10 show that $F_1$, $F_2$ and $F_3$ are the only possible settings corresponding to $|\lambda| = \lambda_o$, hence possibly to Case II. This proves that Case I is the general case. Lemmas 7 and 8 then trivially prove that $F_1$ and $F_2$ are handled by Case II. Finally, lemmas 9 and 11 prove that $F_3$ is also handled by Case II. 
\end{proof}

\subsubsection{Case I}
This case occurs for $|\lambda| \neq 4$. In other words, this is the case where at least one of the initial vertices in $\tilde{\mathbf{u}}$ is different from the other two and where the rank of the input matrix $M$ in equation (\ref{eq:Mmat}) is $\rank(M)>1$. Similarly to Case I for three moving vertices, the problem is reformulated as a depressed quartic polynomial and the roots of this polynomial are found using  Ferrari's method. 
We start by multiplying equation~(\ref{eq:mat_lam2}) by the adjugate $X^*$ of $X$ and obtain:
\begin{equation}
    \det(X)\mathbf{u} = X^*\mathbf{b},
    \label{eq:mateqadjb}
\end{equation}
where the adjugate is:
\begin{equation}
	X^*=\delta Y =\lambda^2-16
	\begin{bmatrix}
        -4 & 0 & 0 & s\lambda\\ 
        0 & -4 & -s\lambda\ & 0\\ 
        0 & -s\lambda\ & -4 & 0\\ 
        s\lambda\ & 0 & 0 & -4
	\end{bmatrix},
	\label{eq:adjY2}
\end{equation}
with $\delta = \lambda^2-16$ and $Y\in\mathbb{R}^{4\times4}$. We note that $\det(X) = \delta^2$. We substitute this and equation~(\ref{eq:adjY2}) in equation~(\ref{eq:mateqadjb}) and obtain:
\begin{equation}
    \delta\mathbf{u} = Y\mathbf{b}.
    \label{eq:mateqadj2b}
\end{equation}
We observe that the signed area $A^*(\delta\mathbf{v}) = \delta^2 A^*(\mathbf{v})$. Also that $\delta\mathbf{v} = [\mathbf{u}, \delta\tilde{x}_c,\delta\tilde{y}_c]$.
Thus, we calculate the signed area of $\delta\mathbf{v}$ and obtain after some minor expanding:
\begin{equation}
    \delta^2 A^*(\mathbf{v}) = a_2 \lambda^2 + a_1 \lambda + a_o,
    \label{eq:areayv2}
\end{equation}
where:
\begin{align}
	a_0  &=  128((\tilde{x}_a-\tilde{x}_c)(\tilde{y}_b-\tilde{y}_a)-(\tilde{x}_a-\tilde{x}_b)(\tilde{y}_c-\tilde{y}_a)) \\
	a_1  &= - 32\left(\tilde{x}_a^2+\tilde{x}_b^2 + 2\tilde{x}_c^2
	+\tilde{y}_a^2+\tilde{y}_b^2 +2\tilde{y}_c^2\right) \\
	& \quad + 64(\tilde{x}_a\tilde{x}_b+\tilde{x}_a\tilde{x}_c+\tilde{x}_b\tilde{x}_c+ \tilde{y}_a\tilde{y}_b+\tilde{y}_a\tilde{y}_c+\tilde{y}_b\tilde{y}_c) \\
	a_2  &= 8((\tilde{x}_a-\tilde{x}_c)(\tilde{y}_b-\tilde{y}_a)-(\tilde{x}_a-\tilde{x}_b)(\tilde{y}_c-\tilde{y}_a)).
\end{align}
We rewrite these coefficients more compactly. Concretely, $a_0$ and $a_2$ contain the signed area $A^*(\tilde{\mathbf{v}})$ of the input vertices and $a_1$ contains the sum of the squared distances $D_o$ of the two moving vertices to the fixed vertex:
\begin{equation}
	D_o(\tilde{\mathbf{v}}) = (\tilde{x}_a-\tilde{x}_c)^2 + (\tilde{y}_a-\tilde{y}_c)^2 + (\tilde{x}_b-\tilde{x}_c)^2 + (\tilde{y}_b-\tilde{y}_c)^2 .
\end{equation}
We substitute equation~(\ref{eq:areayv2}) in the signed area constraint~(\ref{eq:sigconstarea}) multiplied by $\delta^2$ and obtain:
\begin{equation}
    sA^*(\mathbf{v}) - \delta^2 A_o = 0.
    \label{eq:subconsiar2}
\end{equation}
This way, the signed area only depends on the known initial vertices $\tilde{\mathbf{v}}$ and prescribed sign $s$. Because the signed area is quadratic in the vertices, and the vertices are quadratic rational in $\lambda$, the resulting equation is a quartic in $\lambda$:
\begin{equation}
	A_o\lambda^4 - 16(2A_o+sA^*(\tilde{\mathbf{v}}))\lambda^2 + 32 D_o(\tilde{\mathbf{v}})\lambda +256(A_o - s A^*(\tilde{\mathbf{v}})) = 0.
	\label{eq:quart2}
\end{equation}
This is a depressed quartic because it does not have a cubic term. 
We can thus rewrite it to the standard form by simply dividing by $A_o$, giving:
\begin{equation}
	\lambda^4 + p\lambda^2 + q\lambda + r = 0,
	\label{eq:depquar2}
\end{equation}
with:
\begin{align}
	p &= - \frac{16(2A_o + sA^*(\tilde{\mathbf{v}}))}{A_o} \\
	q &=  \frac{32D_o(\tilde{\mathbf{v}})}{A_o} \\
	r &=  \frac{256(A_o- sA^*(\tilde{\mathbf{v}}))}{A_o}.
\end{align}
This can be solved using Ferrari's method. 
We observe that when $\rank(M) = 2$, implying $A^*(\tilde{\mathbf{v}})=0$, coefficients $p$ and $r$ become constants. However, the equation remains a general depressed quartic which can be solved with our procedure. 

\subsubsection{Case II}
\label{sec:case2spec}
Case II occurs for $|\lambda| = 4$. From proposition 2, this means that the initial vertices $\tilde{\mathbf{v}}$ are colocated as $\tilde{v}_a=\tilde{v}_b=\tilde{v}_c$ or represent right isosceles triangles under conditions from proposition 2. We show that the problem is represented by translated homogeneous and linearly dependent equations. We find their null space, particular solution and then a subset constrained by the prescribed area.

We first translate the coordinate system to bring the fixed vertex to the origin $\tilde{\mathbf{u}}'=\tilde{\mathbf{u}}'-v_c$ translating the unknown vertices to $\mathbf{u}'=\mathbf{u}-v_c$. 
Substituting $|\lambda| =  4$ in matrix $X$, we obtain:
\begin{equation}
    4
	\begin{bmatrix}
	1 &  0 &  0 & \sign(\lambda)s \\
	0 &  1 & -\sign(\lambda)s & 0 \\
	0 & -\sign(\lambda)s &  1 & 0 \\
	\sign(\lambda)s &  0 &  0 & 1 
	\end{bmatrix} .
	\label{eq:mat_lamcas2spec}
\end{equation}
This matrix has a rank of two and is thus non-invertible. Thus, $X\mathbf{u}' = \tilde{\mathbf{u}}'$ is solvable if and only if $\tilde{\mathbf{u}}'$ lies in the column space $C(X)$. The column space can be calculated by factoring $X$ into its singular value decomposition (SVD) $X = W\Sigma W^\top$ and taking the first $\rank(X)$ columns of the unitary matrix $W$.
For each value of $s\sign(\lambda)$, we have column spaces expressed as two-dimensional linear subspaces $\{\gamma_1\mathbf{w}_1^- +\gamma_2\mathbf{w}_2^-\}$ and $\{\gamma_1\mathbf{w}_1^+ +\gamma_2\mathbf{w}_2^+ \}$ where $\gamma_1,\gamma_2 \in \mathbb{R}$ and with bases $\mathbf{w}_1^-,\mathbf{w}_2^- \in \mathbb{R}^4$ and $\mathbf{w}_1^+,\mathbf{w}_2^+ \in \mathbb{R}^4$ such that:
\begin{equation}
	\begin{bmatrix}
		\mathbf{w}_1^- & \mathbf{w}_2^-
	\end{bmatrix} = 
	\begin{bmatrix}
		                  0 &   \frac{1}{\sqrt{2}} \\
		 \frac{1}{\sqrt{2}} &                    0 \\
         \frac{1}{\sqrt{2}} &                    0 \\
		                  0 &   -\frac{1}{\sqrt{2}}
	\end{bmatrix},
\end{equation}
and:
\begin{equation}
	\begin{bmatrix}
		\mathbf{w}_1^+ & \mathbf{w}_2^+
	\end{bmatrix} = 
	\begin{bmatrix}
		                  0 &   \frac{1}{\sqrt{2}} \\
		 \frac{1}{\sqrt{2}} &                    0 \\
        -\frac{1}{\sqrt{2}} &                    0 \\
		                  0 &    \frac{1}{\sqrt{2}}
	\end{bmatrix},
\end{equation}
We have that $\mathbf{w}_1^-,\mathbf{w}_2^-,\mathbf{w}_1^+$ and $\mathbf{w}_2^+$ represent right isosceles triangles of the same area of $\frac{1}{4}$ with orientation $s\sign(\lambda)$. The linear combinations $\gamma_1\mathbf{w}_1^- +\gamma_2\mathbf{w}_2^-$ and $\gamma_1\mathbf{w}_1^+ +\gamma_2\mathbf{w}_2^+$ represent right isosceles triangles of any area and opposite orientations (or colocated points if $\gamma_1 = \gamma_2 = 0$). 
This shows that the system is solvable if and only if $\tilde{\mathbf{v}}'$ represents a right isosceles triangle of orientation $\sign(A^*(\tilde{\mathbf{v}}')) = s\sign(\lambda)$ or colocated vertices.

The system $X\mathbf{u}' = \tilde{\mathbf{u}}'$ is solved by first finding the solutions of the homogeneous system $X\mathbf{u}_h = 0$ and translating them by a particular solution $\mathbf{u}_p$, obtaining $\mathbf{u}'=\mathbf{u}_h + \mathbf{u}_p$. The homogeneous system has an infinite number of solutions which come from the null space of $X$. This can be represented as a two-dimensional linear subspace $\mathbf{u}_h= \beta_1\mathbf{u}_1 +\beta_2\mathbf{u}_2$ where the coefficients $\beta_1,\beta_2 \in \mathbb{R}$ are not both zero and with bases $\mathbf{u}_1,\mathbf{u}_2 \in \mathbb{R}^4$ such that:
\begin{equation}
	\begin{bmatrix}
		\mathbf{u}_1 & \mathbf{u}_2
	\end{bmatrix} = 
	\begin{bmatrix}
		0 & -s\sign(\lambda) \\ s\sign(\lambda) & 0 \\ 1 & 0 \\ 0 & 1 
	\end{bmatrix}.
	\label{eq:family2}
\end{equation}
We have that $\mathbf{u}_1$ and $\mathbf{u}_2$ represent two pairs of vertices which form right isosceles triangles when adding one third vertex in the origin of the same area $\frac{1}{4}$. The linear combination of $\mathbf{u}_h=\beta_1\mathbf{u}_1 +\beta_2\mathbf{u}_2$ generate right isosceles triangles of any area when adding one third vertex in the origin of orientation $-s\sign(\lambda)$. We then calculate the particular solution $\mathbf{u}_p$ using the pseudo-inverse as:
\begin{equation}
    \mathbf{u}_p = X^\dagger\tilde{\mathbf{u}}' = \frac{1}{4}
    \begin{bmatrix}
		\tilde{x}_a' + \sign(A^*(\tilde{\mathbf{v}}'))\tilde{y}_b' \\ 
		\tilde{y}_a' - \sign(A^*(\tilde{\mathbf{v}}'))\tilde{x}_b' \\ 
		\tilde{x}_b' - \sign(A^*(\tilde{\mathbf{v}}'))\tilde{y}_a'\\ 
		\tilde{y}_b' + \sign(A^*(\tilde{\mathbf{v}}'))\tilde{x}_a'  
	\end{bmatrix}.
\end{equation}
We then translate the null space with the particular solution and obtain $\mathbf{u}' = \beta_1\mathbf{u}_1 +\beta_2\mathbf{u}_2 + \mathbf{u}_p$. This linear combination represents pairs of vectors (which are perpendicular only if $\mathbf{u}_p=0$).

The next step is to constrain these to the prescribed area and orientation. We have $\mathbf{v}'= [ \mathbf{u}',0,0] \in \mathbb{R}^6$. After some minor algebraic manipulations, we obtain the signed area as:
\begin{equation}
    A^*(\mathbf{v}') = \left(1 + s\sign(A^*(\tilde{\mathbf{v}}'))\sign(\lambda)\right)\frac{A^*(\tilde{\mathbf{v}}')}{4} - s\sign(\lambda)\frac{\beta_1^2 + \beta_2^2}{2}.
\end{equation}
When $A^*(\tilde{\mathbf{v}}')\neq0$ we have $s\sign(\lambda) = \sign(A^*(\tilde{\mathbf{v}}'))$ thus $\sign(\lambda) = s\sign(A^*(\tilde{\mathbf{v}}'))$. 
However, when $A^*(\tilde{\mathbf{v}}')=0$ we have $\sign(A^*(\mathbf{v}'))=-s\sign(\lambda)$, which implies that $\sign(\lambda)=-s$. Using the orientation constraint~(\ref{eq:orconst}), we can express $\mathbf{u}'$ as:
\begin{equation}
\begin{aligned}
    \qquad & \mathbf{u}' = \beta_1\mathbf{u}_1 + \beta_2\mathbf{u}_2 + \mathbf{u}_p \\
    \text{s.t.} \qquad&  \left(s + \sign(\lambda)\sign(A^*(\tilde{\mathbf{v}}'))\right)\frac{A^*(\tilde{\mathbf{v}}')}{4} - \sign(\lambda)\frac{(\beta_1^2 + \beta_2^2)}{2} - A_o = 0
\end{aligned}
\end{equation}
Because of the area and orientation constraints, and because $\mathbf{u}_1$ and $\mathbf{u}_2$ are rotated copies of each other, the family defined by equation~(\ref{eq:family2}) can be generated by simply rotating $\mathbf{u}_1$ by
\begin{equation}
    \rho = \sqrt{\beta_1^2+\beta_2^2} = \sqrt{\frac{\sign(A^*(\tilde{\mathbf{v}}'))A^*(\tilde{\mathbf{v}}')-4kA_o}{2}}
\end{equation}
where $k$ depends on the type of input:
\begin{equation}
    k = \begin{cases}
s\sign(A^*(\tilde{\mathbf{v}}')) & \text{if $A^*(\tilde{\mathbf{v}})\neq0$}\\
-1 &\text{if $A^*(\tilde{\mathbf{v}})=0$}
\end{cases}
\end{equation}
so that the area constraint is met and then rotate $\mathbf{u}_1$ by some arbitrary angle $\theta$. 
We note that when $k=s\sign(A^*(\tilde{\mathbf{v}}'))=1$, then $\rho \in \mathbb{R}$ as long as $A(\tilde{\mathbf{v}})/4\geq A_o$ (which corresponds to setting $F_3$).
We define a new basis vector $\mathbf{u}_c$ as:
\begin{equation}
    \mathbf{u}_c = \rho
    \begin{bmatrix}
    0 & ks & 1 & 0 
    \end{bmatrix}.
\end{equation}
Finally the triangle vertices are translated to the original coordinates by adding the fixed vertex $v_c$, we obtain:
\begin{equation}
    \mathbf{v} =  \begin{bmatrix}\mathcal{R}(\theta)\mathbf{u}_c+\mathbf{u}_p & 0 & 0 \end{bmatrix}^\top + \begin{bmatrix} v_c & v_c & v_c \end{bmatrix}^\top
    \label{eq:case2b}
\end{equation}
where $\mathcal{R}(\theta)$ is a block diagonal matrix replicating the 2D rotation matrix $R(\theta)$ two times as $\mathcal{R}=\mydiag(R(\theta),R(\theta))$. All the possible solutions form triangles and have the same cost.

\subsubsection{Proof of Lemmas}
\label{sec:lemproofspec}
\begin{proof}[Proof of Lemma 7]
We start with the forward implication: $F_1 \Rightarrow A(\tilde{\mathbf{v}})= 0$ and $|\lambda|=4$.
In $F_1$, $\tilde{\mathbf{v}}$ represents a single point. This implies $A(\tilde{\mathbf{v}})=0$ and $D_o(\tilde{\mathbf{v}})=0$. Replacing these values in the depressed quartic equation (\ref{eq:depquar}) causes the coefficients $p$ and $r$ to become constants, and coefficient $q$ to vanish (also the orientation constraint vanishes).
The depressed quartic thus transforms into a bi-quadratic:
\begin{equation}
    \lambda^4 - 32\lambda^2 +256 = 0,
\end{equation}
whose solutions are:
\begin{equation}
    |\lambda| = 4. 
\end{equation}
We now turn to the reverse implication: $F_1 \Leftarrow A(\tilde{\mathbf{v}}) = 0$ and $|\lambda|=4$. We substitute $|\lambda| = 4$ in equation (\ref{eq:subconsiar2}), giving:
\begin{equation}
    4s\sign(A^*(\tilde{\mathbf{v}}))\sign(\lambda)|A^*(\tilde{\mathbf{v}})| - D_o(\tilde{\mathbf{v}})=0.
    \label{eq:contot2}
\end{equation}
Since $A(\tilde{\mathbf{v}}) = 0$, then the only solution that satisfies equation (\ref{eq:contot2}) for any given value of $s\sign(A^*(\tilde{\mathbf{v}}))\sign(\lambda)$ is with $D_o(\tilde{\mathbf{v}})=0$, which implies that the input triangle is collapsed into a single point, hence to $F_1$. 
\end{proof}

\begin{proof}[Proof of Lemma 8]
We start with the forward implication: $F_2 \Rightarrow A(\tilde{\mathbf{v}}) \neq 0$ and $\lambda=-4$. In $F_2$, $\tilde{\mathbf{v}}$ represents a right isosceles triangle and orientation inversion. This implies $A(\tilde{\mathbf{v}})\neq0$ and $s\sign(A^*(\tilde{\mathbf{v}}))=-1$. We perform a similarity transformation to $\tilde{\mathbf{v}}$ to bring the fixed vertex to the origin and another to the $x$-axis, scaled to normalise their  distance, giving $\tilde{\mathbf{v}}'=[1,0,0,\sign(A^*(\tilde{\mathbf{v}})),0,0]^\top$ 
where $\sign(A^*(\tilde{\mathbf{v}}))$ determines the orientation of the triangle.
We obtain $A^*(\tilde{\mathbf{v}}')=\frac{\sign(A^*(\tilde{\mathbf{v}}))}{2}$ and $D_o(\tilde{\mathbf{v}}')=2$, thus the coefficients of the depressed quartic equation~(\ref{eq:depquar2}) become:
\begin{align}
	p &= - \frac{32A_o - 8}{A_o} \\
	q &=  \frac{64}{A_o} \\
	r &=  \frac{256A_o + 128}{A_o}.
\end{align}
Substituting these coefficients in Cardano's formula we obtain:
\begin{align}
	Q_1 &= -\left(\frac{32A_o + 4}{3A_0}\right)^2 \\
	Q_2 &= \left(\frac{32A_o + 4}{3A_0}\right)^3,
\end{align}
making $\sqrt{Q_1^3+Q_2^2}=0$ and the real root $\alpha_o$ of Cardano's resolvent cubic to become:
\begin{equation}
    \alpha_o = 2\left(\frac{32A_o + 4}{3A_o}\right)+\frac{32A_o-8}{3A_o} = 32.
\end{equation}
We finally use $\alpha_o$ to extract the roots of the depressed quartic:
\begin{equation}
	\lambda = \frac{s_1\sqrt{32}\lambda_o+s_2\sqrt{8\left(\frac{1-s_1}{A_o}\right)}}{\sqrt{2}}
	\label{eq:rootsb}
\end{equation}
where $s_1,s_2 \in \{-1,1\}$. We thus have the following roots:
\begin{equation}
    \lambda \in \left\{ -4,-4,4-2i\sqrt{\frac{2}{A_o}},4+2i\sqrt{\frac{2}{A_o}}\right\}.
\end{equation}
Considering only the real roots we have $\lambda = -4$.

\noindent We now turn to the reverse implication: $F_2 \Leftarrow A(\tilde{\mathbf{v}}) \neq 0$ and $\lambda=-4$. We substitute $\lambda = -4$ in equation (\ref{eq:subconsiar2}), giving:
\begin{equation}
    -4s\sign(A^*(\tilde{\mathbf{v}}))A(\tilde{\mathbf{v}}) - D_o(\tilde{\mathbf{v}})=0.
    \label{eq:con1b}
\end{equation}
Since $A(\tilde{\mathbf{v}})\neq0$ and $D_o(\tilde{\mathbf{v}})>0$ then equation (\ref{eq:con1b}) can only be solved when $s\sign(A^*(\tilde{\mathbf{v}}))=-1$. We perform the same similarity transformation used at the beginning of the proof to the unknown input triangle $\tilde{\mathbf{v}}$ giving $\tilde{\mathbf{v}}'=[1,0,x_b',y_b',0,0]^\top$ where one of the vertices remains unknown. We then have $A^*(\tilde{\mathbf{v}}') = \frac{\tilde{y}_b'}{2}$ or \linebreak $\sign(A^*(\tilde{\mathbf{v}}'))A(\tilde{\mathbf{v}}') = \frac{\tilde{y}_b'}{2}$and $ D_o(\tilde{\mathbf{v}}') = x_b'^2 + y_b'^2 + 1$. which we substitute in equation (\ref{eq:con1b}) and obtain:
\begin{equation}
    x_b'^2+y_b'^2  + 2sy_b' + 1 = 0,
\end{equation}
which we rewrite as:
\begin{equation}
    x_b'^2 + \left(y_b'+s\right)^2 = 0.
\end{equation}
This is the equation of a single point, making $\tilde{\mathbf{v}}'$ a right isosceles triangle $\tilde{\mathbf{v}}'=[0,1,0,-s,0,0]^\top$ where $s$ determines the orientation of the triangle. Since $\tilde{\mathbf{v}}'$ was a similarity transformation of $\tilde{\mathbf{v}}$, then $\tilde{\mathbf{v}}$ is also a right isosceles triangle when $s\sign(A^*(\tilde{\mathbf{v}}))=-1$, which corresponds to $F_2$.
\end{proof}

\begin{proof}[Proof of Lemma 9]
In $F_3$, $\tilde{\mathbf{v}}$ represents a right isosceles triangle with $A(\tilde{\mathbf{v}})/4\geq A_o$ and no orientation inversion.
This implies $A^*(\tilde{\mathbf{v}})\neq0$ and $s\sign(A^*(\tilde{\mathbf{v}}))=1$. 
We perform the same similarity transformation to $\tilde{\mathbf{v}}$ as in lemma 7,
thus the coefficients of the depressed quartic equation~(\ref{eq:depquar}) become:
\begin{align}
	p &= - \frac{32A_o + 8}{A_o} \\
	q &=  \frac{32}{A_o} \\
	r &=  \frac{256A_o - 128}{A_o}.
\end{align}
Substituting these coefficients in Cardano's formula we obtain:
\begin{align}
	Q_1 &= -\left(\frac{32A_o-4}{3A_0}\right)^2 \\
	Q_2 &= -\left(\frac{32A_o-4}{3A_0}\right)^3,
\end{align}
making $\sqrt{Q_1^3+Q_2^2}=0$.
After some factoring we obtain the real root $\alpha_o$ of Cardano's resolvent cubic as:
\begin{equation}
    \alpha_o = -2\sqrt[3]{\left(
    \frac{32A_o -8 s A^*(\tilde{\mathbf{v}}')}{3A_o}\right)^3}+\frac{32A_o+16 s A^*(\tilde{\mathbf{v}}')}{3A_o}.
\end{equation}
Since $s\sign(A^*(\tilde{\mathbf{v}}))=1$ and $A^*(\tilde{\mathbf{v}}') = \sign(A^*(\tilde{\mathbf{v}}))A(\tilde{\mathbf{v}}')$, then $\alpha_o \in \mathbb{R}$ only when $A(\tilde{\mathbf{v}}')/4 \geq A_o$. Under this condition, we obtain $\alpha_o = 32$.
Substituting in equation (\ref{eq:roots}) we obtain:
\begin{equation}
     \lambda = \frac{s_1\sqrt{32}\lambda_o+s_2\sqrt{8\left(\frac{ 1-s_1}{A_o}\right)}}{\sqrt{2}}
\end{equation}
where $s_1,s_2 \in \{-1,1\}$. We thus have the following roots:
\begin{equation}
    \lambda \in \left\{-4-2\sqrt{\frac{2}{A_o}},-4+2\sqrt{\frac{2}{A_o}},4,4\right\}^\top.
    \label{eq:eqcase1b}
\end{equation}
Thus, we have two roots where $|\lambda|\neq 4$ (which correspond to solutions for Case I) and there exists at least one solution $\lambda = 4$ for $F_3$ (which correspond to the solution of Case II).
\end{proof}

\begin{proof}[Proof of Lemma 10]
We substitute $\lambda = 4$ in equation (\ref{eq:subconsiar2}), giving:
\begin{equation}
   4s\sign(A^*(\tilde{\mathbf{v}}))A(\tilde{\mathbf{v}}) - \sigma^2(\tilde{\mathbf{v}})=0.
    \label{eq:con2b}
\end{equation}
Since $A(\tilde{\mathbf{v}})\neq0$ then equation (\ref{eq:con2b}) can only be solved when $s\sign(A^*(\tilde{\mathbf{v}}))=1$. We perform the same similarity transformation to $\tilde{\mathbf{v}}$ as in lemma 8, substitute $\sign(A^*(\tilde{\mathbf{v}}'))A(\tilde{\mathbf{v}}')$ and $\sigma^2(\tilde{\mathbf{v}}')$ in equation (\ref{eq:con2b}) and obtain:
\begin{equation}
    x_b'^2+y_b'^2 - 2sy_b' + 1 = 0,
\end{equation}
which we rewrite as:
\begin{equation}
    x_b'^2 + \left(y_b'-s\right)^2 = 0.
\end{equation}
This is the equation of a single point making $\tilde{\mathbf{v}}'$ a right isosceles triangle $\tilde{\mathbf{v}}'=[0,1,0,s,0,0]^\top$ where $s$ determines the orientation of the triangle. Since $\tilde{\mathbf{v}}'$ was a similarity transformation of $\tilde{\mathbf{v}}$, then $\tilde{\mathbf{v}}$ is also a right isosceles triangle when $s\sign(A^*(\tilde{\mathbf{v}}))=1$ for any value $A(\tilde{\mathbf{v}})$ (including $A(\tilde{\mathbf{v}}')/4 \geq A_o$) which corresponds to $F_3$.
\end{proof}

\begin{proof}[Proof of Lemma 11]
We perform the same similarity transformation to $\tilde{\mathbf{v}}$ as in lemma 8 and obtain $A^*(\tilde{\mathbf{v}}')=\frac{\sign(A^*(\tilde{\mathbf{v}}))}{2}$. 
We relax the condition of $F_3$ where $\frac{A(\tilde{\mathbf{v}}')}{4} \geq A_o$ by reformulating $A_o = \frac{1}{2z}$ where $z$ is a scaling factor $z>0 \in \mathbb{R}$. We substitute this in equation~(\ref{eq:eqcase1b}) and extract the roots :
\begin{equation}
    \lambda = \left[-4-4\sqrt{z},-4+4\sqrt{z},4,4\right]^\top.
\end{equation}
We start with the solution given by Case I where $|\lambda|\neq 4$. We compact both values as $\lambda=-4+s_3 4\sqrt{z}$ where $s_3 \in \{-1,1\}$. We substitute this in equation (\ref{eq:mateqadj2b}) and obtain:
\begin{equation}
    \mathbf{u}' = \begin{bmatrix}
    \frac{s_3}{\sqrt{z}} & 0 & 0 & \frac{s_3}{\sqrt{z}} 
    \end{bmatrix}^\top.
\end{equation}
We calculate the cost of the solution of Case I by substituting $\mathbf{v}' = [\mathbf{u}', \tilde{x}_c,\tilde{y}_c]$ and $\tilde{\mathbf{v}}'$ in equation (\ref{eq:cost}) and obtain:
\begin{equation}
    \mathscr{C}_1(\mathbf{v}') = \frac{2(\sqrt{z}-s_3)^2}{z}
\end{equation}
Now we turn to the solution given by Case II where $\lambda = 4$. We calculate the basis vector $\mathbf{u}'_c$, the particular solution $\mathbf{u}'_p$ and substitute them in equation~(\ref{eq:case2b}) and obtain:
\begin{equation}
    \mathbf{v}' = \begin{bmatrix}
    \frac{1}{2} & 
    \frac{\sqrt{z-4}}{2\sqrt{z}} &  
    \frac{\sqrt{z-4}}{2\sqrt{z}} & 
    \frac{1}{2} & 0 & 0 
    \end{bmatrix}^\top.
\end{equation}
We calculate the cost of the solution of Case II by substituting this and $\tilde{\mathbf{v}}'$ in equation (\ref{eq:cost}) and obtain:
\begin{equation}
    \mathscr{C}_2(\mathbf{v}') = 1 - \frac{2}{z}
\end{equation}
We compare the cost of both solutions $\mathscr{C}_1(\mathbf{v}')\geq\mathscr{C}_2(\mathbf{v}')$ and obtain:
\begin{equation}
     \frac{2(\sqrt{z}-s_3)^2}{z} \geq 1 - \frac{2}{z}
\end{equation}
After some minor manipulations we obtain:
\begin{equation}
    \frac{z-4s_3\sqrt{z}+4}{2z} \geq 0
\end{equation}
We substitute $\sqrt{z}=a$ where $a>0 \in \mathbb{R}$ and obtain the quadratic expression:
\begin{equation}
    \frac{(a-4s_3)^2}{2a^2} \geq 0
\end{equation}
which represents an upward opening parabola which is always positive for any value of $a$ and thus $z$. This implies that $\mathscr{C}_1(\mathbf{v}')\geq\mathscr{C}_2(\mathbf{v}')$ for any value $z$ including $z>4$. Since $\tilde{\mathbf{v}}'$ was a similarity transformation of $\tilde{\mathbf{v}}$, then $\mathscr{C}_1(\mathbf{v})\geq\mathscr{C}_2(\mathbf{v})$.
This means that in $F_3$ the solution provided by case II has the lowest cost, thus is the optimal solution.
\end{proof}

\subsubsection{Numerical Implementation}
We use the theory developed in the previous sections to construct a numerically robust procedure, given in Algorithm \ref{alg:OTPPAO1fix}, to solve OTPPAO with one fixed vertex.  
It uses the input vertices $\tilde{\mathbf{v}}$, prescribed area $A_o$ and orientation $s$ as inputs. It also uses an area error tolerance $E$ to handle round-off in the area constraint~(\ref{eq:sigconstarea}). Similar to Algorithm 1, Algorithm \ref{alg:OTPPAO1fix} starts by generating the solutions from Case I, then Case II, and chooses the optimal one.
For Case I, we obtain a list $\mathcal{v}_1$ of at most 4 solutions.
For Case II, we obtain a single best solution $\mathbf{v}_2$, the optimally rotated one, the vertices basis and translation to generate all solutions following equation~(\ref{eq:case2b}).
The overall optimal solution $\mathbf{v}_o$ is chosen amongst $\mathcal{v}_1$ and $\mathbf{v}_2$.
The algorithm returns the optimal solution, along with all the solutions from Case I and Case II. 

\begin{algorithm}
    \caption{Optimal Triangle Projection with a Prescribed Area and Orientation with a Fixed Vertex}
    \label{alg:OTPPAO1fix}
    \begin{algorithmic}[1]
    \Require $\tilde{\mathbf{v}}$ - input vertices, $A_o$ - prescribed area, $s$ - prescribed orientation, $E$ - area error tolerance
    \Ensure $\mathbf{v}_o$ - optimal triangle, $\mathcal{v}_1$ - Case I triangle set, $\mathbf{v}_2$ - Case II optimal triangle, $\mathbf{u}_c, \mathbf{v}_t $ - Case II basis and offset 
    \Function{OTTPAO$\_1$}{$\tilde{\mathbf{v}},A_o,s, E = 10^{-3}$}
    \State $\mathcal{v}_1 \gets$ \Call{SolveCase1OneFixedVertex}{$\tilde{\mathbf{v}},A_o,s,E$}  \Comment{Case I solutions}
    \State $(\mathbf{v}_2, \mathbf{v}_c, \mathbf{u}_c, \mathbf{v}_t) \gets$ \Call{SolveCase2OneFixedVertex}{$\tilde{\mathbf{v}},A_o,s,E$}  \Comment{Case II solutions}
    \State $\mathbf{v}_o\gets$ \Call{FindTriangleOfMinimalCost}{$\mathcal{v}_1 \cup \{ \mathbf{v}_2 \}$}
    \Comment{Optimal solution}
    \State \textbf{return} $\mathbf{v}_o, \mathcal{v}_1, \mathbf{v}_2, \mathbf{u}_c, \mathbf{v}_t$
    \EndFunction
    \end{algorithmic}
\end{algorithm}

\begin{algorithm}
    \caption{Closed-form Analytic Solution to Case I of OTPPAO with One Fixed Vertex}
    \label{alg:case11fix}
    \begin{algorithmic}[1]
        \Require $\tilde{\mathbf{v}}$ - input vertices, $A_o$ - prescribed area, $s$ - prescribed orientation, $E$ - area error tolerance
        \Ensure $\mathcal{v}_1$ - solution list
		\Function{SolveCase1OneFixedVertex}{$\tilde{\mathbf{v}},A_o,s,E$}
		\State $D_o(\tilde{\mathbf{v}})) \gets  (\tilde{x}_a-\tilde{x}_c)^2 + (\tilde{y}_a-\tilde{y}_c)^2 + (\tilde{x}_b-\tilde{x}_c)^2 + (\tilde{y}_b-\tilde{y}_c)^2$
		\State $p \gets - \frac{16(2A_o + s A^*(\tilde{\mathbf{v}}))}{A_o}$ \Comment{Compute the coefficients of the depressed quartic} 
		\State $q \gets  \frac{32 D_o(\tilde{\mathbf{v}})}{A_o}$
		\State $r \gets  \frac{256(A_o- s A^*(\tilde{\mathbf{v}}))}{A_o}$ 
		\State $\boldsymbol{\lambda} \gets $ \Call{FerrariSolution}{$p,q,r$} \Comment{Solve for the four possible Lagrange multipliers}
		\State $\mathcal{v}_1 \gets \emptyset$ \Comment{Create an empty set of solutions}
    	\For{$t\gets 1,\dots,4$}
    	\Comment{Generate and select the triangles}
    	    \State $\lambda_o \gets \operatorname{Re}(\boldsymbol{\lambda}(t))$ \Comment{Keep the real part} 
    	    \State $h \gets (\lambda_o^2-16)^\dagger$ \Comment{Compute the inverse denominator}
    	    \State $\mathbf{u} \gets h
			    \begin{bmatrix}
	                \tilde{x}_c\lambda_o^2 + 4s(\tilde{y}_b-\tilde{y}_c)\lambda_o-16\tilde{x}_a \\
		            \tilde{y}_c\lambda_o^2 + 4s(\tilde{x}_c-\tilde{x}_b)\lambda_o-16\tilde{y}_a \\
		            \tilde{x}_c\lambda_o^2 + 4s(\tilde{y}_c-\tilde{y}_a)\lambda_o-16\tilde{x}_b \\
		            \tilde{y}_c\lambda_o^2 + 4s(\tilde{x}_a-\tilde{x}_c)\lambda_o-16\tilde{y}_b \\
			    \end{bmatrix}$ 
			    \Comment{Compute the vertices}\label{lst:linevert2}
			\State $\mathbf{v} = [\mathbf{u}, \tilde{x}_c,\tilde{y}_c]$
            \If{$|s A^*(\mathbf{\mathbf{v}})-A_o|\leq E$} 
            \Comment{Check the area constraint} \label{lst:vertverif2}
            \State $\mathcal{v}_1 \gets \mathcal{v}_1 \cup \{ \mathbf{v} \}$ \Comment{Add the vertices to the solution set}
    	    \EndIf
    	\EndFor 
    	\State \textbf{return} $\mathcal{v}_1$
		\EndFunction
	\end{algorithmic}
\end{algorithm}

\begin{algorithm}
    \caption{Closed-form Analytic Solution to Case II of OTPPAO with One Fixed Vertex}
    \label{alg:case21fix}
    \begin{algorithmic}[1]
    \Require $\tilde{\mathbf{v}}$ - input vertices, $A_o$ - prescribed area, $s$ - prescribed orientation, $E$ - area error tolerance
    \Ensure $\mathbf{v}_2$ - optimal triangle, $\mathbf{u}_c, \mathbf{v}_t$ vertices basis and translation.
    \Function{SolveCase2OneFixVert}{$\tilde{\mathbf{v}},A_o,s,E$}
        \If{$A(\mathbf{\mathbf{v}})\leq E$} 
            \Comment{Check the input's area}
            \State $k \gets s\sign(A^*(\mathbf{\mathbf{v}}))$ \Comment{Compute $k$ for a right isosceles triangle}
        \Else
            \State $k \gets -1$ \Comment{Compute $k$ for a single point}
    	\EndIf
    	\State $\tilde{\mathbf{v}}' \gets \tilde{\mathbf{v}} - v_c$ \Comment{Translate the input vertices}
    	\State $\rho \gets  \sqrt{\frac{\sign(A^*(\tilde{\mathbf{v}}'))A^*(\tilde{\mathbf{v}}')-4kA_o}{2}}$ \Comment{Computes the area constraint parameter}
        \State $\mathbf{u}_c \gets \operatorname{Re}(\rho)
                \begin{bmatrix}
                    0 & ks & 1 & 0
                \end{bmatrix}^\top $ \Comment{Compute the solution basis} \label{lst:linenull2}
        \State $\mathbf{u}_p \gets
                \frac{1}{4}
			    \begin{bmatrix}
		            \tilde{x}_a'+\sign(A^*(\tilde{\mathbf{v}}))\tilde{y}_b' \\ 
		            \tilde{y}_a'-\sign(A^*(\tilde{\mathbf{v}}))\tilde{x}_b' \\ 
		            \tilde{x}_b'-\sign(A^*(\tilde{\mathbf{v}}))\tilde{y}_a' \\ 
		            \tilde{y}_b'+\sign(A^*(\tilde{\mathbf{v}}))\tilde{x}_a' 
	            \end{bmatrix}$
			    \Comment{Compute the particular solution} \label{lst:linepart2}
        \State $\tilde{\mathbf{u}}_2 \gets$ rearrange $\tilde{\mathbf{u}}'$  into a $2\times2$ matrix
        \State $\mathbf{u}_{c2} \gets$ rearrange $\mathbf{u}_{c}$  into a $2\times2$ matrix
        \State $(U_1,\Sigma, U_2) \gets \SVD\left(\tilde{\mathbf{u}_2}\mathbf{u}_{c2}^{\top}\right)$\Comment{Compute the optimal rotation} 
        \State $D \gets \mydiag(1,\det(U_2U_1^\top))$
        \State $R \gets U_2 D U_1^\top$ 
        \State $\mathbf{v}_t \gets  \begin{bmatrix}\mathbf{u}_p & 0 & 0\end{bmatrix}^\top + \begin{bmatrix} v_c & v_c & v_c \end{bmatrix}^\top$ \Comment{Compute the translation vector}
        \State $\mathbf{v}_2 =  \begin{bmatrix}\mathcal{R}(\theta)\mathbf{u}_c  & 0 & 0 \end{bmatrix}^\top + \mathbf{v}_t$
            \Comment{Compute the optimal solution} 
	    \State \textbf{return} $\mathbf{v}_2, \mathbf{u}_c, \mathbf{v}_t$ 
	\EndFunction
    \end{algorithmic}
\end{algorithm}

\subsection{Two Fixed Vertices}
We assume that $v_b$ and $v_c$ are fixed. Thus, we redefine $\mathbf{v} = [v_a, \tilde{x}_b,\tilde{y}_b,\tilde{x}_c,\tilde{y}_c]^\top$ where $v_a$ is the moving vertex.
We take $\frac{\partial \mathscr{L}}{\partial v_a}=0$ which are essentially the first two equalities of equation~(\ref{eq:lagmult}).
We rearrange these equations into a set of linear equations:
\begin{equation}
	\begin{bmatrix}
	x_a \\ y_a 
	\end{bmatrix} = 
	\begin{bmatrix}
		\tilde{x}_a - \frac{s\lambda}{4}(\tilde{y}_c-\tilde{y}_b) \\
		\tilde{y}_a - \frac{s\lambda}{4}(\tilde{x}_b-\tilde{x}_c)
	\end{bmatrix}.
	\label{eq:v_res2}
\end{equation}
We thus solve system~(\ref{eq:v_res2}) with two cases. In Case I, which is the most general, we have $v_b \neq v_c$. In Case II, we have $v_b = v_c$. The solution for the latter case is trivial, since no matter what value we give to $\lambda$, the non-fixed vertex remains the same ($x_a = \tilde{x}_a$ and $y_a = \tilde{y}_a$). 
For Case I, we substitute the result of equation~(\ref{eq:v_res2}) in $\mathbf{v}$ and calculate the signed area constraint~(\ref{eq:sigconstarea}). The resulting linear equation in $\lambda$ is $a_1\lambda + a_0=0$ where:
\begin{align*}
    a_0  &= 4s\left((\tilde{x}_a-\tilde{x}_c)(\tilde{y}_b-\tilde{y}_a)-(\tilde{x}_a-\tilde{x}_b)(\tilde{y}_c-\tilde{y}_a)\right) - 8A_o \\
	a_1  &=  -(\tilde{x}_b^2 + \tilde{x}_c^2 + \tilde{y}_b^2 + \tilde{y}_c^2 - 2\tilde{x}_b\tilde{x}_c - 2\tilde{y}_b\tilde{y}_c).
\end{align*}
We can rewrite these coefficients more compactly. Concretely, $a_0$  contain the area $A^*(\tilde{\mathbf{v}})$ of the input vertices and $a_1$ contains the square distance $P_o$ between the two fixed vertices:
\begin{equation}
    P_o = (\tilde{x}_b-\tilde{x}_c)^2 + (\tilde{y}_b-\tilde{y}_c)^2.
\end{equation}
We solve for $\lambda$ and obtain:
\begin{equation}
    \lambda = 8\frac{sA^*(\tilde{\mathbf{v}})-A_o}{P_o}.
    \label{eq:lambda2}
\end{equation}
We substitute $\lambda$ from equation~(\ref{eq:lambda2}) in equation~(\ref{eq:v_res2}) and obtain:
\begin{equation}
	\begin{bmatrix}
	x_a \\ y_a 
	\end{bmatrix} = 
	\begin{bmatrix}
	\tilde{x}_a \\ \tilde{y}_a 
	\end{bmatrix}
		- 4\frac{sA_o+A^*(\tilde{\mathbf{v}})}{P_o}
	\begin{bmatrix}
	(\tilde{y}_c-\tilde{y}_b) \\ (\tilde{x}_b-\tilde{x}_c)
	\end{bmatrix}.
\end{equation}
The numerically robust procedure of this solution is given in Algorithm \ref{alg:OTPPAO2fix}.

\begin{algorithm}
    \caption{Optimal Triangle Projection with a Prescribed Area and Orientation with Two Fixed Vertices}
    \label{alg:OTPPAO2fix}
    \begin{algorithmic}[1]
    \Require $\tilde{\mathbf{v}}$ - input vertices, $A_o$ - prescribed area, $s$ - prescribed orientation, $E$ - distance error tolerance
    \Ensure $\mathbf{v}_o$ - optimal triangle 
    \Function{OTTPAO$\_2$}{$\tilde{\mathbf{v}},A_o,s, E = 10^{-3}$}
    \State $P_o = (\tilde{x}_b-\tilde{x}_c)^2 + (\tilde{y}_b-\tilde{y}_c)^2$ \Comment{Compute square distance between fixed vertices}
    \If{$P_o>E$}
        \State $\begin{bmatrix}
                	x_a \\ y_a 
                	\end{bmatrix} = 
                	\begin{bmatrix}
                	\tilde{x}_a \\ \tilde{y}_a 
                	\end{bmatrix}
                		- 4P_o^\dagger(sA_o+A^*(\tilde{\mathbf{v}}))
                	\begin{bmatrix}
                	(\tilde{y}_c-\tilde{y}_b) \\ (\tilde{x}_b-\tilde{x}_c)
                \end{bmatrix}$
    \Else
    \State $\begin{bmatrix}
                	x_a \\ y_a 
                	\end{bmatrix} = 
                	\begin{bmatrix}
                	\tilde{x}_a \\ \tilde{y}_a 
                	\end{bmatrix}$
    \EndIf
    \State $\mathbf{v}_o\gets
                    \begin{bmatrix} 
                        x_a & y_a & \tilde{x}_b & \tilde{y}_b & \tilde{x}_c & \tilde{y}_c 
                	\end{bmatrix}^{\top}$
    \Comment{Compute the optimal solution}
    \State \textbf{return} $\mathbf{v}_o$
    \EndFunction
    \end{algorithmic}
\end{algorithm}


\bibliographystyle{elsarticle-num}
\bibliography{EnCoV_Library}


\end{document}